\documentclass[12pt]{report}
\usepackage{amsmath}     
\usepackage{amsthm}         
\usepackage{amssymb}                                               
\usepackage{amsfonts}
\usepackage{nccmath}     
\usepackage{amssymb}  
\usepackage[pdftex]{graphicx}
\usepackage[chapter]{algorithm}
\usepackage{algorithm,algorithmicx,algpseudocode}
\usepackage{color}
\usepackage[margin=3.5cm]{geometry}
\usepackage{setspace}
\usepackage{sectsty}
\usepackage[none]{hyphenat}

\sectionfont{\fontsize{15}{15}\selectfont}

\newcommand{\esssup}{\mathop{\mathrm{ess\,sup}}} 
\newcommand{\essinf}{\mathop{\rm ess\,inf}}

\newtheoremstyle{mystyle}
  {}
  {}
  {\itshape}
  {}
  {\bfseries}
  {.}
  { }
  {\thmname{#1}\thmnumber{ #2}\thmnote{ (#3)}}

\theoremstyle{mystyle}

\newtheorem{definition}{Definition}[chapter]
\newtheorem{remark}{Remark}[chapter]
\newtheorem{theorem}{Theorem}[chapter]
\newtheorem{lemma}{Lemma}[chapter]

\begin{document}

\begin{titlepage}
    \begin{center}
        \vspace*{0.5cm}

        \huge {Adaptive discontinuous Galerkin methods \\ for nonlinear parabolic problems}
                
        \vspace{4cm}

        \large {Thesis submitted for the degree of \\ \vspace{0.2cm} Doctor of Philosophy \\ \vspace{0.2cm} at the University of Leicester}
       
        \vspace{1.2cm}
                
        by

\vspace{1.2cm}

        Stephen Arthur Metcalfe MMath \\ \vspace{0.2cm} Department of Mathematics \\ \vspace{0.2cm} University of Leicester

        \vspace{3cm}

2014

    \end{center}
\end{titlepage}

\begin{center}
\LARGE Adaptive discontinuous Galerkin methods \\  for nonlinear parabolic problems
\end{center}
\begin{center}
Stephen Arthur Metcalfe
\end{center}

\bigskip
\bigskip

\noindent This work is devoted to the study of a posteriori error estimation and adaptivity in parabolic problems with a particular focus on spatial discontinuous Galerkin (dG) discretisations.

We begin by deriving an a posteriori error estimator for a linear non-stationary convection-diffusion problem that is discretised with a backward Euler dG method. An adaptive algorithm is then proposed to utilise the error estimator. The effectiveness of both the error estimator and the proposed algorithm is shown through a series of numerical experiments.

Moving on to nonlinear problems, we investigate the numerical approximation of blow-up. To begin this study, we first look at the numerical approximation of blow-up in nonlinear ODEs through standard time stepping schemes. We then derive an a posteriori error estimator for an implicit-explicit (IMEX) dG discretisation of a semilinear parabolic PDE with quadratic nonlinearity. An adaptive algorithm is proposed that uses the error estimator to approach the blow-up time. The adaptive algorithm is then applied in a series of test cases to gauge the effectiveness of the error estimator.

Finally, we consider the adaptive numerical approximation of a nonlinear interface problem that is used to model the mass transfer of solutes through semi-permiable membranes. An a posteriori error estimator is proposed for the IMEX dG discretisation of the model and its effectiveness tested through a series of numerical experiments.

\chapter*{Acknowledgements}

Foremost, I wish to thank my supervisors Dr. Andrea Cangiani and Dr. Emmanuil Georgoulis for their invaluable advice and guidance throughout the last four years. Additionally, I would like to thank Dr. Irene Kyza for inviting me to collaborate with her in Crete and for her insights into blow-up problems. I would also like to thank everyone in the mathematics department at the University of Leicester for the many fruitful discussions that helped contribute to this work. Finally, I acknowledge the funding of the EPSRC; without their support this work would not have been possible.

\tableofcontents

\listoftables
 
\listoffigures

\onehalfspacing

\chapter{Introduction}

Partial differential equations are key in the modelling of various physical and biological phenomena. The solutions to PDEs are usually unavailable through analytical means, so numerical methods are employed in order to approximate the solution. Furthermore, a number of nonlinear PDE problems exhibit local multiscale behaviour such as boundary or interior layers, interfaces or even local space-time blow-up. Such local multiscale features require high local resolution of the numerical methods employed in their approximation. Hence, the use of numerical methods which can automatically detect and resolve such multiscale features is of interest.

Adaptive algorithms that are driven by a posteriori error estimators lie at the heart of finite element analysis. For elliptic problems, there are a wide variety of different error estimators available \cite{AO00,V96}. Moreover, adaptive algorithms for elliptic problems are relatively well understood; at least for simple elliptic problems {\bf --} see \cite{CKNS08, D96, MNS00} for the standard conforming finite element method and \cite{BN10, HKW09} for the interior penalty dG method.

For linear parabolic problems, there are many error estimators available in the literature for popular discretisations (typically a standard time stepping scheme paired with a spatial finite element discretisation). These error estimators are usually composed of an initial condition estimator, a space estimator and a time estimator.  However, generally speaking, it is unclear how to utilise each of these individual estimators to drive adaptivity. Furthermore, while mesh change is crucial for the efficient numerical approximation of mobile solutions, it is well known that careless mesh refinement and/or coarsening can lead to destabilisation of the finite element solution \cite{BKM13,D82}. While some progress has been made on the construction of adaptive algorithms for parabolic problems \cite{CF04,EJ91,EJ95a,EJ95b,EJ95c,EJL98,JNT90,M98,PWW,SS05}, none of the algorithms in the literature have been shown to reduce the error estimator that they utilise at the correct rate of convergence with respect to the average number of degrees of freedom and the total number of time steps.

This work will investigate adaptive algorithms for spatial discontinuous \mbox{ }\mbox{ } Galerkin discretisations of parabolic problems with a focus on nonlinear problems. To achieve this, we shall look at adaptivity in the context of three different problems, to be detailed below. 

In Chapter 2, we introduce notation and state some approximation results and general theorems that shall be used throughout this work. We also discuss the issue of robustness that arises in the a posteriori error estimation of discontinuous Galerkin discretisations of  stationary convection-diffusion equations. We then introduce a robust error bound for the stationary problem, taken from \cite{SZ09}, that shall be used extensively in this work.

Chapter 3 deals with linear non-stationary convection-diffusion equations.  In particular, we derive an a posteriori error estimator for a backward Euler dG discretisation of the problem. The primary challenge in such a derivation is the robustness of the error estimator with respect to the diffusion parameter $\varepsilon$. We address this through the elliptic reconstruction framework of Makridakis and Nochetto \cite{MN03} paired with a robust error estimator for the stationary problem \cite{SZ09} together with a robust treatment of the temporal residual. The proposed error estimator can be viewed as the analogue of that given in \cite{V05b} by Verf\"{u}rth but with a discontinuous Galerkin spatial discretisation instead of a streamline upwind Petrov-Galerkin spatial discretisation. As well as \cite{V05b}, there are other error estimators available  in the literature for different discretisations of linear non-stationary convection-diffusion equations. In particular, we mention \cite{DEV13} wherein the authors produce a robust error estimator through a flux reconstruction approach and the work of Picasso and Prachittham \cite{PP09} wherein they develop an error estimator for a Crank-Nicolson temporal discretisation; their discretisation and error estimator also provide for the use of anisotropic elements. There are also a variety of different a posteriori error estimators available for the $L^2(L^2)$ norm, \emph{cf.} \cite{EP05,HS01,S06}. In addition to the error estimator we also propose an adaptive algorithm, based on that given in \cite{CF04}, that utilises different parts of the error estimator to control space and time adaptivity. The effectiveness of both the error estimator and the proposed algorithm is then tested in four numerical experiments. It is also worth noting that adaptive algorithms designed specifically for non-stationary convection-diffusion problems are explored in \cite{PP09,S06}.

In Chapter 4, we investigate the numerical approximation of blow-up in ODEs. More specifically, we derive an a posteriori error estimator for an ODE with polynomial nonlinearity that is discretised using standard time stepping schemes. The biggest difficulty in the construction of such an error estimator is having to deal with a nonlinear error equation {\bf --} this can be handled through a local \emph{continuation argument}. A \emph{continuation argument} is a special type of proof by contradiction that is often used to prove existence results for nonlinear parabolic PDEs; such arguments have been instrumental in the derivation of a posteriori bounds for a variety of nonlinear parabolic problems \cite{B05,GM14,KNS04}. A posteriori error estimators produced by a continuation argument are \emph{conditional} in the sense that they only hold providing that the estimators involved are sufficiently small. In order to use the proposed error estimator to approximate the blow-up time, we investigate the design of suitable adaptive algorithms. To that end, two adaptive algorithms are proposed and then applied in two test cases under different time stepping schemes to compare their effectiveness. Although we choose to investigate the numerical approximation of blow-up through a posteriori error estimation, there are other ways of approaching this problem that have been published in the literature. In \cite{JW14}, the authors prove existence results for numerical approximations to a nonlinear ODE with a  polynomial growth condition provided that the time step lengths are sufficiently small. For the particular case of a polynomial nonlinearity, they show that selecting the time step lengths in a certain way yields approach to the blow-up time. In \cite{H06}, the authors transform an ODE with polynomial nonlinearity through an arc length transformation.  They then use a forward Euler method to approximate the transformed equation and they show that their adaptive algorithm, which is based on their transformation plus a tolerance controlled ODE integrator, converges towards the blow-up time linearly with respect to the total number of time steps. Finally, in \cite{SF90} the authors approximate a nonlinear ODE using a $\theta$-method along with a temporal rescaling of the ODE and they show that their numerical solution has the same asymptotic behaviour as the exact solution.

In Chapter 5, we build on the results of the previous chapter by investigating blow-up in semilinear parabolic PDEs. In particular, we study blow-up in nonlinear non-stationary convection-diffusion equations that, based on the results of the previous chapter, are discretised using an IMEX dG method. An a posteriori error estimator is then proposed for this discretisation of the problem. In order to produce a viable error estimator for this problem, there are three major obstacles that need to be overcome: the lack of symmetry of the problem, the nonlinear error equation and the error due to non-conformity. A posteriori error estimation for blow-up in symmetric problems has been considered in \cite{K09,K01} by Kyza and Makridakis, however, such results are not easily generalised to non-symmetric problems; the lack of symmetry in the problem can, however, be dealt with through the Gagliardo-Nirenberg inequality. The nonlinear error equation is dealt with in a similar way to in Chapter 4 {\bf --} through a continuation argument and the error due to non-conformity can be dealt with via localised bounds for the non-conforming part of the error \cite{DG01,KP03,KP07}. The proposed error estimator is utilised to approximate the blow-up time of the problem through an adaptive algorithm that is based on those given in Chapter 3 and Chapter 4. The adaptive algorithm is then applied to some test problems and the results are compared to those given in Chapter 4. It is worth noting that solution profiles close to the blow-up time can also be obtained through the rescaling algorithm of Berger and Kohn \cite{BK88,NZ14} or the MMPDE method \cite{BHR96,HMR08}. There is also work looking at the numerical approximation of blow-up in the  nonlinear Schr\"{o}dinger equation and its generalisations \cite{ADKM03,CS02,FI03,KMR11,TS92}. Other numerical methods for approximating blow-up in a variety of different nonlinear PDEs can be found in \cite{ABN09,DPMF05,DKKV98,FGR02,NB11}.

In Chapter 6, we consider an IMEX dG discretisation of a nonlinear interface problem that was introduced in \cite{CGJ13,CGJ14}, based on the works \cite{F08,KK58,P05,T53,Z02}, to model the mass transfer of solutes through semi-permiable membranes. There are a variety of a posteriori error estimators of both residual and recovery type in the literature for different discretisations of interface problems. In particular, for the conforming finite element method \cite{CZ09,CZ12}, the discontinuous Galerkin method \cite{CYZ11} and the finite volume method \cite{EPTW11,MJ14}. Based on these works, and the techniques used in previous chapters, we seek to derive a residual-based a posteriori error estimator for this discretisation of the model. The effectiveness of the error estimator is then tested through a series of numerical experiments utilising the adaptive algorithm that was developed in Chapter 3.

Finally, in Chapter 7 we summarise the results of this work and discuss ways in which this work could be extended.

\chapter{Preliminaries}

\begin{section}{Sobolev spaces}

Let $\omega \subset \mathbb{R}^2$ be a bounded Lipschitz domain with boundary $\partial \omega$. For $1 \leq p \leq +\infty$, we define the $L^p$ norms by
\begin{equation}
\begin{aligned}
\notag
\|v\|_{L^p(\omega)}&:=\left(\int_{\omega} \! |v|^p \, dx \right)^{1/p} \qquad && \text{for } 1\le p< +\infty, \\
\|v\|_{L^{p}(\omega)}&:=\esssup_{x \in \omega} |v(x)| & & \text{for } p = +\infty,
\end{aligned}
\end{equation} 
and the respective $L^p$ spaces by 
\begin{equation}
\notag
L^p(\omega) := \left \{u \mbox{ }\big |\mbox{ } ||u||_{L^p(\omega)} < \infty \right \}.
\end{equation} 
Note that $L^2(\omega)$ is a Hilbert space with an inner product given by 
\begin{equation}
\notag
\left(u,v\right)_{\omega} := \int_{\omega} \! uv \, dx.
\end{equation} 
When $\omega$ is the computational domain $\Omega$ (to be defined later) then because both the $L^2$ norm and $L^2$ inner product over $\Omega$ are used frequently in this thesis, the relevant subscripts are omitted. Given a multi-index $\alpha \in \mathbb{N}^2$, the weak derivative $D^{\alpha}$ of order $|\alpha|$ is given by
\begin{equation}
\notag
D^{\alpha} := \frac{\partial^{|\alpha|}}{\partial x_1^{\alpha_1} \partial x_2^{\alpha_2} }.
\end{equation} 

For $k \in \mathbb{N}$ and $1 \leq p \leq \infty$, the Sobolev space $W^{k,p}(\omega)$ is given by
\begin{equation}
\notag
W^{k,p}(\omega) := \left\{u \in L^p(\omega) \mbox{ } \big| \mbox{ } D^{\alpha}u \in L^p(\omega),  \mbox{ } |\alpha| \leq  k \right \}.
\end{equation} 
The spaces $W^{k,p}(\omega)$ are equipped with the norms
\begin{equation}
\begin{aligned}
\notag
\|v\|_{W^{k,p}(\omega)} &:=\left(\sum_{|\alpha| \leq k} ||D^{\alpha}v||^p_{L^p(\omega)}\right)^{1/p} \qquad && \text{for } 1\le p< +\infty, \\
\|v\|_{W^{k,p}(\omega)}&:=\sum_{|\alpha| \leq k} ||D^{\alpha}v||_{L^{\infty}(\omega)}& & \text{for }  p = +\infty.
\end{aligned}
\end{equation} 
The space $W^{k,2}(\omega)$ together with the standard inner product is a Hilbert space which we shall denote by $H^k(\omega) := W^{k,2}(\omega)$. Fractional Sobolev spaces ($H^{1/2}(\omega)$ in particular) are also of significant use when trying to make sense of boundary values, in the sense of traces, in Sobolev spaces; we refer to \cite{A03} for details. Whenever boundary values are used in this thesis, they are to be understood in the sense of traces. We define the space $H^1_D(\omega)$, which is the prototypical PDE solution space, by
\begin{equation}
\begin{aligned}
\notag
H^1_D(\omega) := \big\{u \in H^1(\omega) \mbox{ } \big| \mbox{ } u |_{\Gamma_D} = 0 \big\},
\end{aligned}
\end{equation}
where $\Gamma_D$ is some subset of $\partial\omega$ with positive one-dimensional Hausdorff measure; if $\Gamma_D = \partial\omega$, we denote this space by $H^1_0(\omega)$.
 Finally, we let $C^k(\omega)$  denote the space of all functions $u$ for which $D^{\alpha}u$ is continuous for all multi-indices $\alpha$ with $|\alpha| \leq k$.

For $T>0$, the spaces $L^p(0,T;X)$ (where $X$ is a real Banach space with norm $\|\cdot\|_X$) consist of all measurable functions $v: [0,T]\to X$ for which
\begin{equation}
\begin{aligned}
\notag
\|v\|_{L^p(0,T;X)}&:=\left(\int_0^T \! \|v(t)\|_{X}^p \, dt\right)^{1/p}<\infty \qquad && \text{for } 1\le p< +\infty, \\
\|v\|_{L^{p}(0,T;X)}&:=\esssup_{0\le t\le T}\|v(t)\|_{X}<\infty& & \text{for }  p = +\infty.
\end{aligned}
\end{equation} 
We also define $H^1(0,T;X):=\big\{u\in L^2(0,T;X) \mbox{ } \big |  \mbox{ } u_t\in  L^2(0,T;X)\big\}$. 
Finally, we denote by $C(0,T;X)$ and $C^{0,1}(0,T;X)$, respectively, the spaces of continuous and Lipschitz continuous functions $v:[0,T] \rightarrow X$ such that
\begin{equation}
\begin{aligned}
\notag
\displaystyle ||v||_{C(0,T;X)} & :=\max_{0 \leq t \leq T}{||v(t)||_X} < \infty \mbox{,} \\ ||v||_{C^{0,1}(0,T;X)}& :=\max \left\{  ||v||_{C(0,T;X)}, \bigg|\bigg|\frac{\partial{v}}{\partial{t}}\bigg|\bigg|_{L^{\infty}(0,T;X)} \right\} < \infty \mbox{.}
\end{aligned}
\end{equation}
\end{section}

\begin{section}{Stationary convection-diffusion equation}\label{prelim}

Let the \emph{computational domain} $\Omega \subset \mathbb{R}^2$ be a polygon with boundary $\partial\Omega$, this assumption will be used throughout the rest of this thesis. We consider the model problem of finding $u:\Omega\to\mathbb{R}$ such that
\begin{equation}\label{ellipticmodel_strong}
\begin{aligned}
 - \varepsilon\Delta{u}+ {\bf a} \cdot \nabla{u}+bu &= f  \qquad  && \text{in } \Omega, \\ u &=0 \qquad && \mbox{on } \partial\Omega \mbox{.}
\end{aligned}
\end{equation}
The variable functions are collectively referred to as the \emph{data} of the problem. Problem \eqref{ellipticmodel_strong} is the prototypical convection-diffusion equation. If the convection is constant then the scale of the solution to  \eqref{ellipticmodel_strong} can be characterised through the ratio of convection to diffusion as described by the P\'{e}clet number
\begin{equation}
\begin{aligned}
\notag
Pe := \frac{|{\bf a}||\Omega|}{\varepsilon}.
\end{aligned}
\end{equation}
When $Pe \gg 1$, \eqref{ellipticmodel_strong} is advection dominated and can exhibit some or all of the following features (see \cite{J05,RST08} for a detailed analysis):

\begin{itemize}
\item The presence of \emph{ordinary layers}, spatial areas containing steep gradients of the solution $u$ of width $\mathcal{O}(\varepsilon)$ that usually occur near the outflow boundary as \emph{boundary layers}.

\item The presence of \emph{parabolic layers}, spatial areas containing moderate gradients of the solution $u$ of width $\mathcal{O}\left(\sqrt{\varepsilon} \right)$ that typically occur on the inflow boundary as boundary layers or as \emph{interior layers}.

\end{itemize}
These complex spatial features can render the numerical approximation of the solution difficult as discussed in the next section.

In order to analyse \eqref{ellipticmodel_strong} and its parabolic counterparts, we must make some assumptions on the data. We assume that: $0<\varepsilon \leq1$, $f \in L^2(\Omega)$, ${\bf a} \in \left[W^{1, \infty}(\Omega)\right]^2$ and $b \in L^{\infty}(\Omega)$. Furthermore, we require some additional assumptions in order to state standard coercivity and continuity results (see, e.g., \cite{SZ09,V05}). To that end, we assume that there are constants $\beta \geq 0$ and $c_* \geq 0$ such that
\begin{equation}
\label{coefficientconditions}
b - \frac{1}{2}\nabla \cdot {\bf a} \geq \beta \quad \text{a.e. in }  \Omega,\qquad ||b-\nabla\cdot {\bf a}  | |_{L^{\infty}(\Omega)} \leq c_*\beta\mbox{.}
\end{equation}
The weak form of \eqref{ellipticmodel_strong} reads: find $u \in H^1_0(\Omega)$ such that
\begin{equation}\label{ellipticmodel_weak}
B\left(u,v \right) =\left(f,v \right) \qquad \forall v \in H^1_0(\Omega),
\end{equation}
where
\begin{equation}
B(u,v)=\int_{\Omega} \! \left(\varepsilon \nabla{u} \cdot \nabla{v}+{\bf a}\cdot \nabla{u}v + buv \right) \, dx.
\end{equation}

\end{section}

\begin{section}{Discontinuous Galerkin method}\label{dg_section}

A \emph{finite element method} is a numerical technique for finding approximate solutions to the weak formulation of PDEs characterised by the use of a subdivision of $\Omega$ referred to as the \emph{mesh} or \emph{triangulation}. The mesh is a collection of \emph{elements} with $K$ denoting a generic element. A finite element space is then constructed over the mesh and the weak form of the PDE is discretised; different choices of finite element space and different discretisations give rise to different finite element methods. The \emph{discretisation parameters} are quantities related to convergence of the method, specifically, the diameters of elements in the mesh (and the lengths of time steps in parabolic problems).

When it comes to the finite element approximation of \eqref{ellipticmodel_weak}, the presence of layers introduces a certain amount of difficulty. In particular, the standard conforming finite element method performs poorly if an insufficient number of elements are placed in the vicinity of the layers resulting in unphysical oscillations. This issue can be solved with layer-adapted meshes such as Shishkin meshes (see \cite{K10} for an overview of this subject) but these special meshes require a priori knowledge of where the layer will occur. These problems led to the development of stabilised finite element methods for convection-diffusion equations \cite{RST08} the most popular of which are the streamline upwind Petrov-Galerkin (SUPG) method \cite{BH82} and the discontinuous Galerkin (dG) method. In this work, we focus upon a dG discretisation of \eqref{ellipticmodel_weak} taken from \cite{HSS02} which is based upon a classical interior penalty discretisation of the diffusive term originally introduced in \cite{AD82,B77a, NZ71} and an upwind discretisation of the transport term first discussed in  \cite{LS74,RD73}. 

In order to state the dG discretisation of \eqref{ellipticmodel_weak}, we need some additional notation. The mesh $\zeta$ is assumed to be constructed via affine mappings \mbox{ }\mbox{ }\mbox{ } $F_{K}:\hat{K}\to K$ with non-singular Jacobian where $\hat{K}$ is the reference triangle or the reference square. The mesh is allowed to contain a uniformly fixed number of regular hanging nodes per edge. We define the finite element space
\begin{equation}\label{eq:FEspace}
V_h\equiv V_h(\zeta) := \big\{v \in L^2(\Omega) \mbox{ }\big | \mbox{ }v|_K\circ F_K \in \mathcal{P}^p\big(\hat{K} \big), \mbox{ } K \in \zeta \big\},
\end{equation}
 where $\mathcal{P}^p(K)$ is the space of polynomials of total degree $p$ if $\hat{K}$ is the reference triangle, or the space of polynomials of degree $p$ in each variable if $\hat{K}$ is the reference square. Let $\mathcal{E}(\zeta)$ denote the set of all edges in the mesh $\zeta$ and $\mathcal{E}^{int}(\zeta)$ the set of all interior edges. We also denote the diameter of an element $K \in \zeta$ by $h_K$ and the length of an edge $E \in \mathcal{E}(\zeta)$ by $h_E$. The outward unit normal to the boundary of an element $K$ is denoted by ${\bf n}_K$. We assume that the mesh $\zeta$ is \emph{shape-regular}, that is, there exists $C>0$ such that for all $K \in \zeta$ we have
\begin{equation}
\begin{aligned}
\notag
\frac{h_K}{d_K} \leq C,
 \end{aligned}
\end{equation}
where $d_K$ denotes the diameter of the largest ball that can be completely contained in $K$.

In what follows, it will be useful to associate \emph{patches} with each element $K \in \zeta$. Specifically, we have the \emph{(elemental) patch} $\tilde{K}$ which is the union of all elements that ``neighbour" $K$ and the \emph{edge patch} $\tilde{K}_E$ which is the union of all edges that intersect the boundary of $K$. Formally, these are defined, respectively, by

\begin{equation}
\begin{aligned}
\notag
\tilde{K} & := \left\{\bigcup K',\mbox{ } K' \in \zeta \mbox{ } \Big | \mbox{ } \partial K \cap \partial K' \neq \emptyset \right \}, \\
\tilde{K}_E & := \left\{\bigcup E,\mbox{ }  E \in \mathcal{E}(\zeta) \mbox{ }\Big | \mbox{ } \partial K \cap \bar{E} \neq \emptyset \right\}.
 \end{aligned}
\end{equation}
We also associate an \emph{(elemental) patch} $\tilde{E}$ with each edge $E \in \mathcal{E}(\zeta)$ which is the union of all elements whose boundary intersects $\bar{E}$, viz.,
\begin{equation}
\begin{aligned}
\notag
\tilde{E} & := \left\{\bigcup K,\mbox{ }  K \in \zeta \mbox{ }\Big | \mbox{ } \bar{E} \cap \partial K \neq \emptyset \right\}.
 \end{aligned}
\end{equation}

Given an edge $E\in \mathcal{E}^{int}(\zeta)$ shared by two elements $K$ and $K'$, a vector field ${\bf v}\in \left[H^{1/2}(\Omega) \right]^2$ and a scalar field $v\in H^{1/2}(\Omega)$, we define jumps $[\cdot]$ and averages $ \{ \cdot \}$ of ${\bf v}$ and $v$ across $E$ by
\begin{equation}
\begin{aligned}
\notag
 \{{\bf v}\} & := \frac{1}{2}({\bf v}|_{\bar{K}}+{\bf v}|_{\bar{K}'}), \qquad &
   [{\bf v}]  & := {\bf v}|_{\bar{K}} \cdot {\bf n}_K+ {\bf v}|_{\bar{K}'} \cdot {\bf n}_{K'},  \\
  \{v\} & := \frac{1}{2}( v|_{\bar{K}}+ v|_{\bar{K}'}), \qquad &
   [v]   & :=  v|_{\bar{K}}  {\bf n}_K+  v|_{\bar{K}'}  {\bf n}_{K'}.
 \end{aligned}
\end{equation}
If $E\subset \partial\Omega$, we set $\{{\bf v}\}:={\bf v}$, $[{\bf v}]:={\bf v} \cdot {\bf n}$, $\{v\}:= v$ and $[ v]:= v {\bf n}$, with ${\bf n}$ denoting the outward unit normal to the boundary $\partial\Omega$. 

We define the inflow and outflow parts of the boundary $\partial\Omega$, respectively, by
\begin{equation}
\notag
 \partial\Omega_{in} := \{x \in \partial\Omega \mbox{ }|\mbox{ } {\bf a}(x) \cdot {\bf n}(x) < 0 \}, \qquad \partial\Omega_{out} := \{x \in \partial\Omega \mbox{ } | \mbox{ } {\bf a}(x) \cdot {\bf n}(x) \geq 0 \}.
\end{equation}
Similarly, the inflow and outflow parts of an element $K$ are defined as
\begin{equation}
\notag
\partial K_{in} := \{x \in \partial K \mbox{ } | \mbox { } {\bf a}(x) \cdot {\bf n}_K(x) < 0 \}, \qquad \partial K_{out} := \{x \in \partial K \mbox{ } | \mbox{ } {\bf a}(x) \cdot {\bf n}_K(x) \geq 0 \}.
\end{equation}

With all the above notation at hand, the dG approximation to \eqref{ellipticmodel_weak} reads as follows: find $u_h \in V_h$ such that
\begin{equation}\label{dgelliptic}
B(u_h,v_h)+K_h(u_h,v_h)=(f,v_h) \qquad  \forall v_h \in V_h \mbox{,}
\end{equation} 
where
\begin{equation}
\begin{aligned}
\label{bilinearformB}
B(u_h,v_h) & := \sum_{K \in \zeta}{\int_{K} \! 
(\varepsilon \nabla u_h  -{\bf{a}}u_h ) \cdot \nabla v_h+(b- \nabla \cdot {\bf a} )u_hv_h\, dx}
\\&+\sum_{E \in \mathcal{E}(\zeta)}\frac{ \gamma \varepsilon}{h_E}\int_{E} \! [u_h] \cdot [v_h] \, ds 
+ \sum_{K \in \zeta}\int_{\partial{K}_{out}} \! u_h[{\bf a}v_h] \, ds, \\ 
K_h(u_h,v_h) & := -\sum_{E \in \mathcal{E}(\zeta)}\int_{E} \! \{\varepsilon \nabla u_h \} \cdot [v_h] + \{\varepsilon \nabla v_h \} \cdot [u_h] \, ds.
\end{aligned}
\end{equation}
The \emph{penalty parameter}, $\gamma$, is set to $\gamma = 2p^2$ in light of \cite{HSS02} so that the operator $B+K_h$ is coercive on $V_h$ (see below).

We note that the bilinear form $K_h$ is not well-defined for arguments in $H^1_0(\Omega)$, but the bilinear form $B$ is and is equal to that appearing in \eqref{ellipticmodel_weak}. To analyse the dG discretisation, we introduce the quantities
\begin{equation}
\begin{aligned}
\notag
|||u||| & := \left(\sum_{K \in \zeta} \Big(\varepsilon||\nabla{u}||^2_{L^2(K)}+\beta||u||^2_{L^2(K)} \Big)+\sum_{E \in \mathcal{E}(\zeta)} \bigg({\frac{\gamma \varepsilon}{h_E} + \beta h_E \bigg )||[u]||^2_{L^2(E)}}\right)^{1/2}, \\ 
|u|_{A} & := \left(\Bigg(\sup_{v \in H^1_0(\Omega) \setminus \{0\}}{\frac{\int_{\Omega} \! {{\bf a}u \cdot \nabla{v}} \, dx}{|||v|||}}\Bigg)^2 + \sum_{E \in \mathcal{E}(\zeta)}{\frac{h_E}{\varepsilon}||[{\bf a}u]||^2_{L^2(E)}} \right)^{1/2}.
\end{aligned}
\end{equation}
These quantities define norms on $H^1_0(\Omega) + V_h$. In the literature, $|||\cdot|||$ is referred to as the \emph{energy norm} while the quantity $|\cdot|_A$ is referred to as a \emph{dual norm}. It is easy to see that the bilinear form $B$ is \emph{coercive} on $H^1_0(\Omega)$, viz.,
\begin{equation}\label{coercivity}
B(v,v) \geq |||v|||^2,
\end{equation}
for all  $v \in H^1_0(\Omega)$, and is \emph{continuous} in the following sense
\begin{equation}\label{continuity}
B(u,v) \lesssim (|||u|||+|u|_{A})|||v|||,
\end{equation}
for all $u \in H^1_0(\Omega) + V_h$ and  $v \in H^1_0(\Omega)$. Moreover, the discrete bilinear form is \emph{coercive} for $v_h \in V_h$ with respect to the energy norm, viz.,
\begin{equation}
\label{FEcoercive}
B(v_h,v_h) + K_h(v_h,v_h) \gtrsim |||v_h|||^2.
\end{equation}
The symbols $\lesssim$ and $\gtrsim$ used above and throughout the rest of the thesis are used to describe inequalities that are true up to an unspecified positive constant that is independent of the data, the discretisation parameters, the exact solution and the dG solution.

\end{section}

\begin{section}{Error bounds for the stationary problem}\label{ell_bounds_sec}

Let $u$ be the exact solution of a PDE and $u_h$ be some finite element approximation; an \emph{a posteriori error estimator}, $\eta$, is an approximation of the error, $e := u - u_h$, in a certain norm $\displaystyle||\cdot||$ such that $||e|| \approx \eta$. The error estimator must be computable and thus is allowed to depend upon the data, the discretisation parameters and the finite element solution $u_h$ but not the unknown solution $u$. In order to discuss how good $\eta$ is at approximating $||e||$, it is useful to introduce the notion of \emph{reliability} and \emph{efficiency} of an estimator. An a posteriori estimator, $\eta$, is said to be \emph{reliable} if there is $C>0$, independent of the exact solution $u$, such that
\begin{equation}
\begin{aligned}
\label{reliability}
||e|| & \leq C\eta,
\end{aligned}
\end{equation} 
while $\eta$ is said to be \emph{efficient} if there is $c>0$, independent of the exact solution $u$, such that
\begin{equation}
\begin{aligned}
\label{efficiency}
c\eta & \leq ||e||.
\end{aligned}
\end{equation} 
The constants appearing in \eqref{reliability}-\eqref{efficiency} are often impossible to calculate explicitly which leads us to the useful notion of the \emph{effectivity index}. If $u$ is known for particular data, $||e||$ may be calculated explicitly for different realisations of $u_h$. Thus, we can compute the \emph{effectivity index} - the ratio of $\eta$ to $||e||$:
\begin{equation}
\begin{aligned}
\notag
\text{effectivity index} & := \frac{\eta}{||e||}.
\end{aligned}
\end{equation} 

The effectivity index naturally leads to the notion of \emph{robustness}. An error estimator, $\eta$, is said to be \emph{robust} (with respect to $||\cdot||$) if the constants in \eqref{reliability}-\eqref{efficiency} are always independent of the data, the discretisation parameters and the finite element solution $u_h$. There is also the weaker notion of \emph{asymptotic robustness}: $\eta$ is said to be \emph{asymptotically robust} if it is robust once the discretisation parameters are sufficiently small. If $\eta$ is asymptotically robust then it successfully reproduces the convergence rate of $||e||$ with respect to the discretisation parameters.

\begin{figure}
\centering
\includegraphics[scale=0.25]{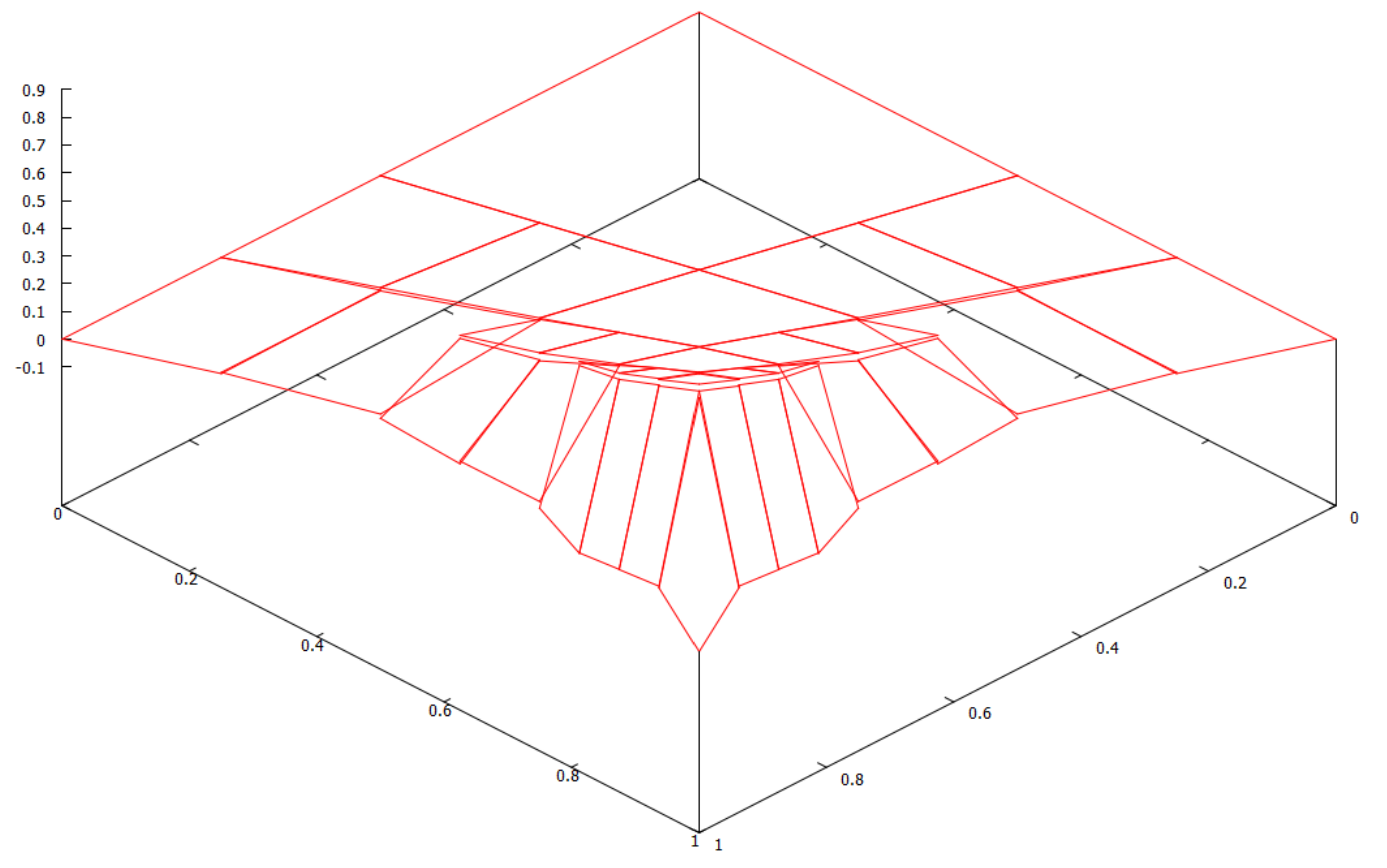}\includegraphics[scale=0.25]{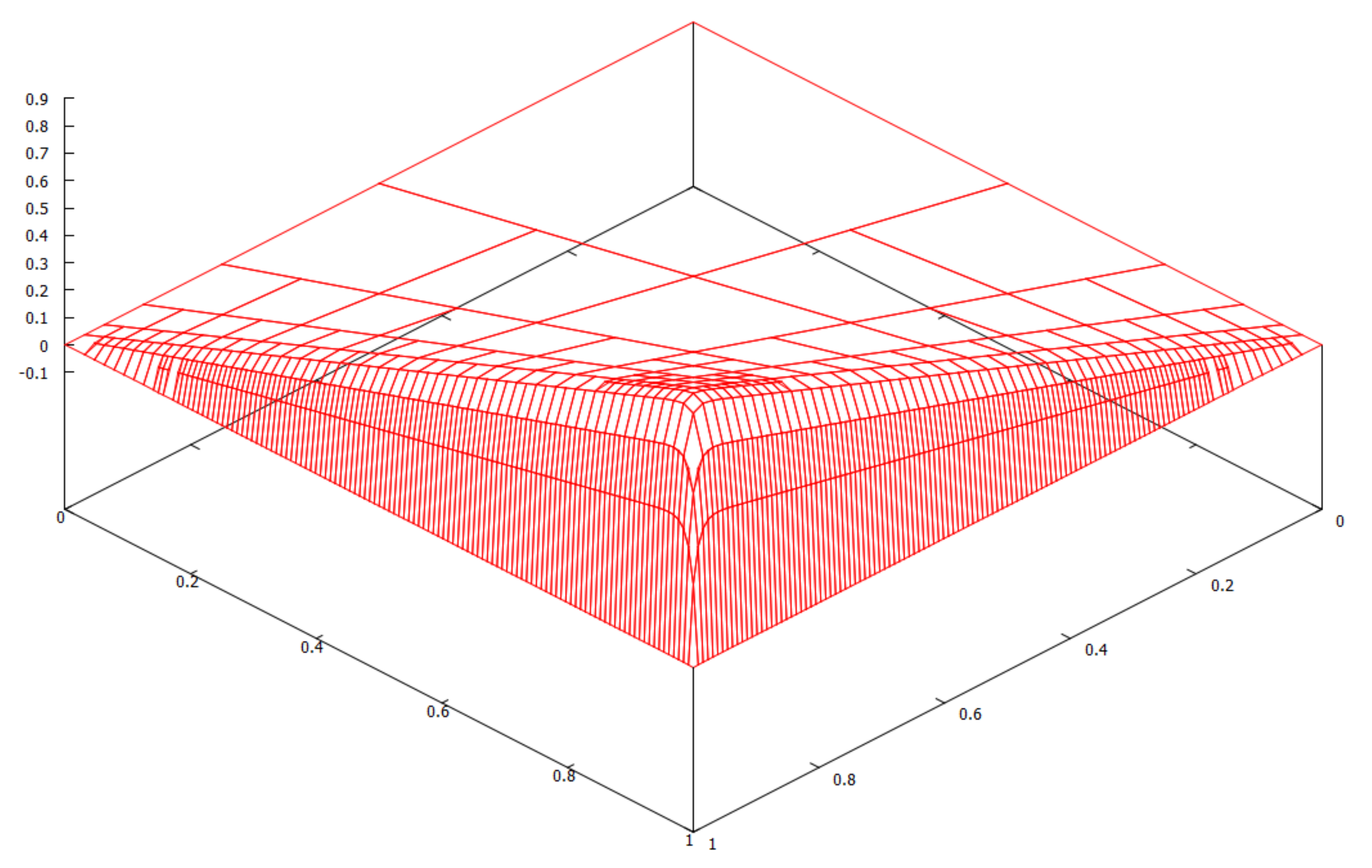}
\caption{Numerical approximation of a boundary layer in the pre-asymptotic regime (left) and the asymptotic regime (right).}
\end{figure}

Regarding the a posteriori error estimation of \eqref{ellipticmodel_weak}, many estimators exist for stabilised finite element schemes \cite{ABR05, APS05, BC04, E10, GHH07, JN13,  K03,S08, SZ09, V98,V05, SZ11} and the primary issue is robustness with respect to the small parameter $\varepsilon$. Before the layers have been covered by a sufficient number of elements, the \emph{pre-asymptotic regime}, the error in the energy norm is relatively constant. When a sufficient number of elements have been placed within the layers, the \emph{asymptotic regime}, the error in the energy norm starts to display the correct convergence rate with respect to the discretisation parameters. In order for an error estimator to be robust with respect to $\varepsilon$ in the energy norm, it must accurately capture the behaviour of the energy norm in both the pre-asymptotic and asymptotic regimes. Standard, classical estimators in the literature significantly overestimate the error in the energy norm in the pre-asymptotic regime to allow the layers to be detected and refined but this means that they are not robust with respect to $\varepsilon$ in the energy norm until the layers have been sufficiently resolved. One way out of this difficulty is to add an additional term, the dual norm, to the energy norm to account for the pre-asymptotic regime \cite{S05,S08}; robustness is then recovered in all regimes for the full norm \cite{SZ09,V05}. The question of whether a robust error estimator exists for the energy norm in all regimes is still open, however, recent results in this direction seem promising \cite{GHM14}.

An a posteriori estimator for the stationary problem, inspired by \cite{SZ09}, will be utilised in our analysis. More specifically, we have the following result whose proof is completely analogous to  Theorem 3.2 in \cite{SZ09} and is therefore omitted for brevity.

\begin{theorem}\label{elliptic_apost}For $f \in L^2(\Omega)$, let $u^s \in H^1_0(\Omega)$ be such that
\[
B(u^s,v)=(f,v)\qquad \forall v\in H^1_0(\Omega),
\]
and consider $u^s_h\in V_h$ such that
\[
B(u^s_h,v_h)+K_h(u^s_h,v_h)=(f,v_h) \qquad\forall v_h\in V_h.
\]
Then the following a posteriori bound holds for any $0 \neq v \in H_0^1(\Omega)$:
\begin{equation}\label{apost_stationary}
\begin{aligned}
\notag
\frac{B(u^s-u_h^s,v)}{|||v|||} & \lesssim  \left ( \sum_{K\in \zeta}\frac{h^2_K}{\varepsilon}||f+ \varepsilon \Delta u^s_h-{\bf a} \cdot \nabla u^s_h - bu_h^s||_{L^2(K)}^2 + \sum_{E\in \mathcal{E}(\zeta)} \frac{\gamma\varepsilon}{h_E}|| [ u^s_h] ||_{L^2(E)}^2 \right. \\  & \left.  +\sum_{E\in \mathcal{E}^{int}(\zeta)} \frac{h_E}{\varepsilon}|| [{\bf a} u^s_h] ||_{L^2(E)}^2
+\sum_{E\in \mathcal{E}^{int}(\zeta)} \varepsilon h_E|| [\nabla u^s_h] ||_{L^2(E)}^2 \right)^{1/2} .
\end{aligned}
\end{equation}
\end{theorem}
  
\end{section}

\begin{section}{Finite element approximation results}

Error bounds for the approximation of functions in $H^1_D(\Omega)$ have been constructed in the finite element literature and will be utilised below.

\begin{theorem}
\label{hodor63}
Given $u \in H^1_D(\Omega)$, there is a finite element quasi-interpolant $I_X u \in H^1_D(\Omega) \cap V_h$ such that
\begin{equation}
\begin{aligned}
\notag
h^{-1}_K||u-I_Xu||_{L^2(K)} & \lesssim ||\nabla u||_{L^2(\tilde{K})} &  \qquad & \forall K \in \zeta, \\ 
h^{-\frac{1}{2}}_E||u-I_Xu||_{L^2(E)} & \lesssim ||\nabla u||_{L^2(\tilde{E})} & & \forall E \in \mathcal{E}(\zeta).
\end{aligned}
\end{equation}
\end{theorem}
\begin{proof}
See \cite{V98b}.
\end{proof}

A useful tool in a posteriori error estimation involving dG finite element spaces is the approximation of functions in the dG finite element space, $V_h$, by functions in the conforming finite element space, $H^1_D(\Omega) \cap V_h$. The next theorem gives bounds on such approximation errors.

\begin{theorem} \label{ncbounds} Given $u_h \in V_h$, there exists a decomposition $u_h = u_{h,c} + u_{h,d}$ with $u_{h,c} \in H^1_D(\Omega) \cap V_h$ and $u_{h,d} \in V_h$ such that the following bounds hold for each element $K \in \zeta$:
\begin{equation}
\begin{aligned}
\notag
||\nabla u_{h,d}||^2_{L^2(K)} & \lesssim \sum_{E \subset \tilde{K}_E \setminus \partial\Omega} h_E^{-1}||[u_h]||^2_{L^2(E)} + \sum_{E\subset \tilde{K}_E \cap \Gamma_D} h_E^{-1}||u_h||^2_{L^2(E)}, \\
||u_{h,d}||^2_{L^2(K)}  & \lesssim \sum_{E \subset \tilde{K}_E
 \setminus \partial\Omega} h_E||[u_h]||^2_{L^2(E)} + \sum_{E \subset \tilde{K}_E \cap \Gamma_D} h_E||u_h||^2_{L^2(E)}, \\
||u_{h,d}||_{L^{\infty}(K)} & \lesssim ||[u_h]||_{L^{\infty}(\tilde{K}_E \setminus \partial\Omega)} + ||u_h||_{L^{\infty}(\tilde{K}_E \cap \Gamma_D)}.
\end{aligned}
\end{equation}
\end{theorem}
\begin{proof}
The proof is based upon a weighted averaging of the dG solution around the nodes; see \cite{KP03,KP07} for the first two estimates and \cite{DG01} for the final estimate.
\end{proof}

Given a function $u \in L^2(\Omega)$, the $L^2$ projection of $u$ onto $V_h$, denoted by $I_{h}u \in V_h$, is the unique solution of 
\begin{equation}
\label{L2project}
(u-I_hu,v_h)=0 \qquad \forall v_h \in V_h.
\end{equation}
\begin{remark}
For the dG finite element space, $I_h u$ can also be constructed by considering \eqref{L2project} elementwise instead of globally and then piecing together the resultant local functions.
\end{remark}

The final theorem in this section gives bounds on the approximation error involved in the $L^2$ projection.
\begin{theorem}
\label{projectionbounds}
Given a function $u \in H^1_D(\Omega)$ and its $L^2$ projection $I_h u \in V_h$, the following bound holds for any element $K \in \zeta$:
\begin{equation}
\notag
h^{-1}_K||u-I_hu||_{L^2(K)} \lesssim ||\nabla u||_{L^2(\tilde{K})}.
\end{equation}
\end{theorem}
\begin{proof} We use the definition of the $L^2$ projection along with the Cauchy-Schwarz inequality to conclude that for any $v_h \in V_h$:
\begin{equation}
\begin{aligned}
\notag
||u-I_hu||^2_{L^2(K)} & = (u-I_hu,u-I_hu)_K \\ 
& = (u-I_hu,u)_K \\ 
&= (u - I_hu, u-v_h)_K \\
&\leq ||u - I_hu||_{L^2(K)}|| u-v_h||_{L^2(K)}.
\end{aligned}
\end{equation}
Therefore,
\begin{equation}
\begin{aligned}
\notag
||u-I_hu||_{L^2(K)} &\leq || u-v_h||_{L^2(K)}.
\end{aligned}
\end{equation}
All that remains is to choose $v_h$ to give the bound presented in the theorem. In particular, the finite element interpolant in Theorem \ref{hodor63} suffices.
\end{proof}

\end{section}

\begin{section}{Useful inequalities}

In this section, we introduce general theorems that will be used throughout the rest of this thesis.
\begin{theorem}[Young's inequality]
Given $a,b \in \mathbb{R}$, then for any $\delta>0$ we have
\begin{equation}
\begin{aligned}
\notag
ab \leq \frac{a^2\delta}{2} + \frac{b^2}{2\delta}.
\end{aligned}
\end{equation}
\end{theorem}
\begin{proof} Follows by expanding the inequality $\displaystyle \big(a\delta^{1/2}-b\delta^{-1/2} \big)^2 \geq 0$.
\end{proof}
\begin{theorem}[Power mean inequality]
Given $a,b \geq0$, then for any $\mu\geq 0$ we have
\begin{equation}
\begin{aligned}
\notag
\big(a+b\big)^{\mu} \leq \max\big\{1,2^{\mu-1}\big\}\big(a^{\mu}+b^{\mu}\big).
\end{aligned}
\end{equation}
\end{theorem}
\begin{proof} Follows from Jensen's inequality.
\end{proof}
\begin{theorem}[Multidimensional integration by parts]
Given $u \in H^1(\Omega)$ and $ \displaystyle {\bf v} \in H(div,\Omega) := \Big\{{\bf v} \in \big [ L^2(\Omega) \big]^2 \text { } \Big | \text{ }  \nabla \cdot {\bf v} \in L^2(\Omega) \Big\}$  we have
\begin{equation}
\notag
\int_{\Omega} \! {\bf v} \cdot {\nabla u} \, dx + \int_{\Omega} \! u \nabla \cdot {\bf v} \, dx = \int_{\partial \Omega} \! (u {\bf v}) \cdot {\bf n} \, ds.
\end{equation}
\end{theorem}
\begin{proof}
Follows from the divergence theorem.
\end{proof}
\begin{theorem}[Poincar\'{e}-Friedrichs inequality]
For $u \in H^1_D(\Omega)$, we have the following bound:
\begin{equation}
\begin{aligned}
\notag
||u||^2 \lesssim ||\nabla u||^2.
\end{aligned}
\end{equation}
\end{theorem}
\begin{proof} See \cite{ZF05}.
\end{proof}
\begin{theorem}[Gagliardo-Nirenberg inequality]
For any $u \in H^1_D(\Omega)$ and $\mu \geq 0$, we have the following bound:
\begin{equation}
\begin{aligned}
\notag
||u||^{2 + \mu}_{L^{2+\mu}(\Omega)} & \lesssim ||u||^2||\nabla u||^{\mu}. \\
\end{aligned}
\end{equation}
\end{theorem}
\begin{proof} See \cite{AA08} specifically or \cite{N59} for a larger class of inequalities.
\end{proof}
\begin{theorem}[Trace inequality]
\label{scaledtrace}
Given $u \in H^1(\Omega)$, the following bound holds for any $\delta \in (0,1)$:
\begin{equation}
\begin{aligned}
\notag
||u||^2_{L^2(\partial\Omega)} & \lesssim \delta||\nabla u||^2 + \big(1+\delta^{-1} \big)||u||^2.
\end{aligned}
\end{equation}
\end{theorem}
\begin{proof}
See Theorem 1.5.1.10. in \cite{PG01}.
\end{proof}
\begin{theorem}[Inverse estimate] Given $v_h \in V_h$, the following bound holds for any $K \in \zeta$:
\begin{equation}
\begin{aligned}
\notag
||\nabla v_h||^2_{L^2(\partial K)} \lesssim h^{-1}_K||\nabla v_h||^2_{L^2(K)}.
\end{aligned}
\end{equation}
\end{theorem}
\begin{proof}
See \cite{WH03}.
\end{proof}
The final two theorems in this section relate to a specific group of integral inequalities and provide an upper bound on the functions that satisfy such \mbox{ } inequalities. The first theorem is (classical) Gronwall's inequality while the second theorem is a variant of the first theorem that makes use of a lower order term.
\begin{theorem}[Gronwall's inequality]
\label{Gronwall1}
Let $T>0$ and suppose that \mbox{ }\mbox{ }\mbox{ }\mbox{ } $c_0, c_1 \in L^1(0,T)$ and $u \in W^{1,1}(0,T)$. If for almost every $t \in (0,T]$ we have

\begin{equation}
\begin{aligned}
\notag
u'(t) \leq c_0(t) + c_1(t)u(t),
\end{aligned}
\end{equation}
then
\begin{equation}
\begin{aligned}
\notag
u(T) \leq G(0,T)u(0) + \int_0^T \! G(s,T)c_0(s) \, ds,
\end{aligned}
\end{equation}
where $\displaystyle G(s,t) := \exp\bigg(\int_s^t \! c_1(\xi) \, d \xi \bigg)$. Additionally, if $c_0$ and $c_1$ are non-negative a.e. then

\vspace{0.2em}

\begin{equation}
\begin{aligned}
\notag
u(T) \leq G(0,T)\bigg(u(0) +  \int_0^T \! c_0(s) \, ds \bigg).
\end{aligned}
\end{equation}
\end{theorem}

\begin{proof}
Multiplying the inequality by $G^{-1}(0,t)$ yields
\begin{equation}
\begin{aligned}
\notag
G^{-1}(0,t)u'(t) \leq G^{-1}(0,t)c_0(t) + G^{-1}(0,t)c_1(t)u(t).
\end{aligned}
\end{equation}
The product rule and the fundamental theorem of calculus imply that
\begin{equation}
\begin{aligned}
\notag
\frac{d}{dt}\left(G^{-1}(0,t)u(t) \right) = G^{-1}(0,t)u'(t) - G^{-1}(0,t)c_1(t)u(t).
\end{aligned}
\end{equation}
Thus,
\begin{equation}
\begin{aligned}
\notag
\frac{d}{dt}\big(G^{-1}(0,t)u(t) \big) \leq G^{-1}(0,t)c_0(t).
\end{aligned}
\end{equation}
The primary result then follows by integrating over $[0,T]$ and noting that \mbox{ } $G(s,T) = G(0,T)G^{-1}(0,s)$.
\end{proof}
\begin{theorem}
\label{Gronwall}
Let $T>0$ and suppose that $c_0$ is a constant, $c_1$ and $c_2$ are non-negative $L^1$ functions and that $u$ is a non-negative $W^{1,1}$ function that satisfies
\begin{equation}
\begin{aligned}
\notag
u^2(T) \leq c^2_0 + \int_0^T \! c_1(s)u(s) \, ds + \int_0^T \! c_2(s)u^2(s) \, ds,
\end{aligned}
\end{equation}
then
\begin{equation}
\begin{aligned}
\notag
u(T) \leq \bigg (|c_0|+\frac{1}{2}\int_0^T \! c_1(s) \, ds  \bigg ) \exp \bigg (\frac{1}{2}\int_0^T \! c_2(s) \, ds \bigg ).
\end{aligned}
\end{equation}
\end{theorem}

\medskip
\begin{proof}
See Theorem 21 in \cite{D01}.
\end{proof}
\end{section}

\chapter{A posteriori error estimation and adaptivity for non-stationary convection-diffusion problems}
\begin{section}{Non-stationary convection-diffusion equation}\label{prelim}

For $T>0$, we consider the model problem of finding $u:\Omega\times(0,T]\to\mathbb{R}$ such that
\begin{equation}\label{model_strong}
\begin{aligned}
\frac{\partial{u}}{\partial{t}} - \varepsilon\Delta{u}+ {\bf a} \cdot \nabla{u}+bu &= f  \qquad & & \text{in } \Omega\times(0,T],
\\ u &=0 \mbox{ } &&\text{on }  \partial\Omega\times (0,T], 
\\ u(\cdot,0) &=u_0 \mbox{ } &&\text{in } {\Omega}.
\end{aligned}
\end{equation}
The model problem \eqref{model_strong} can display the same spatial features as \eqref{ellipticmodel_strong}, namely, boundary and interior layers; the addition of the time domain has the potential to make things more complicated, however. Indeed, for suitable data the solution may exhibit moving boundary or interior layers. Additionally, the temporal behaviour of \eqref{model_strong} can also be strongly influenced by $\varepsilon$ \cite{RST08}.

The standard weak formulation (see, e.g., \cite{E98}) of \eqref{model_strong} reads: find $u\in L^2 \big(0,T;H^1_0(\Omega) \big)\cap H^1 \big(0,T;L^2(\Omega) \big)$ such that for almost every $t \in (0,T]$ we have
\begin{equation}
\label{model_weak}
\bigg(\frac{\partial{u}}{\partial{t}},v\bigg)+B(t;u,v)=(f,v) \qquad \forall v \in H^1_0(\Omega),
\end{equation}
where $B(t;\cdot, \cdot)$ denotes the bilinear form \eqref{bilinearformB} with the data evaluated at the time $t$. We also make some assumptions on the data: $u_0 \in H^1_0(\Omega)$, $0 < \varepsilon \leq 1$, ${\bf a} \in \left[C(0,T;W^{1, \infty}(\Omega)) \right ]^2$, $b \in C (0,T;L^{\infty}(\Omega))$ and $f \in C(0,T;L^2(\Omega))$.

In order to ensure coercivity and continuity of $B(t;\cdot,\cdot)$ for any $t \in [0,T]$, we must extend \eqref{coefficientconditions}. To that end, we assume that there are constants $\beta \geq 0$ and $c_* \geq 0$ such that
\begin{equation}
b - \frac{1}{2}\nabla \cdot {\bf a} \geq \beta \quad \text{a.e. in }  \Omega \times [0,T],\qquad ||b-\nabla\cdot {\bf a} ||_{C(0,T;L^{\infty}(\Omega))} \leq c_*\beta\mbox{.}
\end{equation}

\end{section}

\begin{section}{Space-time discretisation}

The semi-discrete discontinuous Galerkin approximation to \eqref{model_weak} then reads as follows. For $t=0$, set $ u_h(0) \in V_h$ to be some projection of $u_0$ onto $V_h$. Then, seek  $u_h \in C^{0,1}(0,T;V_h)$ such that for almost every $t \in (0,T]$ we have
\begin{equation}\label{dg_sd}
\bigg(\frac{\partial{u_h}}{\partial{t}},v_h\bigg)+B(t;u_h,v_h)+K_h(u_h,v_h)=(f,v_h) \qquad  \forall v_h \in V_h.
\end{equation} 
We shall also consider a full discretisation of problem \eqref{model_weak} by using a backward Euler method to approximate the time derivative. 

To this end, consider a subdivision of $[0,T]$ into time intervals of lengths \mbox{ } \mbox{ }\mbox{ } $\tau_1$, ..., $\tau_n$ such that $\displaystyle \sum_{j=1}^n{\tau_j}=T$ for some $n \geq 1$ then set $t^0 := 0$ and $\displaystyle t^k := \sum_{j=1}^{k}\tau_{j}$. Denote an initial triangulation by $\zeta^0$ and then associate a triangulation $\zeta^k$ to each time step $k>0$ which is assumed to have been obtained from $\zeta^{k-1}$ by locally refining and coarsening $\zeta^{k-1}$. This restriction upon mesh change is made to avoid altering the mesh too much between time steps in an attempt to prevent degradation of the finite element solution, {\em cf.}~\cite{BKM13,D82}. To each mesh $\zeta^{k}$, we assign the finite element space $V_h^k := V_h\big(\zeta^k\big)$ given by~\eqref{eq:FEspace}. We also set $f^k:=f\big(.,t^k\big)$, ${\bf a}^k:={\bf a}\big(.,t^k \big)$ and $b^k:=b\big(.,t^k \big)$ for brevity. 

The fully-discrete dG method then reads as follows. Set $u_h^0$ to be a projection of $u_0$ onto $V_h^0$. For $k=0$, ..., $n-1$, find $u_h^{k+1} \in V_h^{k+1}$ such that
\begin{equation}
\label{dg_fd}
\bigg(\frac{u_h^{k+1}-u_h^k}{\tau_{k+1}},v_h^{k+1}\bigg)+B\big(t^{k+1};u_h^{k+1},v_h^{k+1} \big)+K_h \big(u_h^{k+1},v_h^{k+1} \big)=\big(f^{k+1},v_h^{k+1} \big),
\end{equation}
for all $v^{k+1}_h \in V_h^{k+1}$. We shall take $u_h^0$ to be the orthogonal $L^2$ projection of $u_0$ onto $V_h^0$ although other projections onto $V_h^0$ can also be used.

\end{section}

\begin{section}{Error bounds for the non-stationary problem}

Here, we will present a posteriori error bounds for both the semi-discrete and fully-discrete schemes which can be found in \cite{CGM01}; these results are extended by the authors in \cite{KU14} to convection-diffusion problems with nonlinear reaction term. Other a posteriori error estimators for different space-time discretisations of \eqref{model_weak} can be found in  \cite{DEV13, EP05, HS01, PP09, S06,V05b}.

To devise our error bounds, we use the elliptic reconstruction approach originally introduced by Makridakis and Nochetto for the conforming finite element method \cite{MN03} and extended to dG methods in \cite{GL10,GLV11}; this effectively allows decomposition of the error into separate parabolic and elliptic parts and can be viewed as the a posteriori counterpart to Wheeler's elliptic projection \cite{WF73} from a priori analysis. Elliptic reconstruction allows us to bound the elliptic part of the error with any error estimator for the stationary problem  \eqref{ellipticmodel_weak} that currently exists in the literature; we shall use the bound in Theorem \ref{elliptic_apost} taken from \cite{SZ09}. Given that we already have a reasonable spatial estimator to use, it is clear that the main challenge in the a posteriori estimation of \eqref{model_weak} is thus related to the parabolic part of the error. In particular, the primary challenge is obtaining an error estimator that is either robust or at least asymptotically robust with respect to $\varepsilon$ in the $L^2 (H^1) + L^{\infty}(L^2)$ type norm 

\begin{equation}
\begin{aligned}
\notag
||u||_* := \left(||u||^2_{L^{\infty}(0,T;L^2(\Omega))}+\int_0^T \! |||u|||^2 \, dt\right)^{1/2}.
\end{aligned}
\end{equation}
This requires careful treatment of the temporal part of the residual and we introduce a novel way of dealing with this difficulty.

 \begin{subsection}{An a posteriori bound for the semi-discrete method}\label{sd_apost_sec}
 To highlight the main ideas, we begin by deriving an a posteriori bound for the semi-discrete method.

 \begin{definition}
\label{reconstrdef1}
For almost every $t\in(0,T]$, we define the elliptic reconstruction $w \in H^1_0(\Omega)$ to be the (unique) solution of the problem
\begin{equation}\label{reconstruction_sd}
\notag
B(t;w,v)=\bigg(f-\frac{\partial{u_h}}{\partial{t}},v \bigg) \qquad \forall v \in H^1_0(\Omega).
\end{equation}
\end{definition}

\begin{remark}\label{reconstruction_remark}
The dG discretisation of the above equation is to find a function $w_h \in C^{0,1}(0,T,V_h)$ such that for almost every $t \in (0,T]$ we have
\[
 B(t;w_h,v_h) + K_h(w_h,v_h)=\bigg(f-\frac{\partial{u_h}}{\partial{t}},v_h \bigg) \qquad \forall v_h \in V_h.
 \]
In conjunction with \eqref{FEcoercive} and \eqref{dg_sd}, this implies that $w_h=u_h$. Therefore, $B(t;w-u_h,v)$ can be estimated using Theorem \ref{elliptic_apost}. 
\end{remark}
 \begin{theorem}
\label{nonconforming_bound}
The dG solution, $u_h$, admits a decomposition into a conforming part $u_{h,c} \in H^1_0(\Omega) \cap V_h$ and a non-conforming part $u_{h,d} \in V_h$ with $u_h=u_{h,c}+u_{h,d}$ such that
\begin{equation}
\begin{aligned}
\notag
(|||u_{h,d}|||+|u_{h,d}|_A)^2 & \lesssim \sum_{E \in \mathcal{E}(\zeta)}{\bigg(\frac{\gamma\varepsilon}{h_E} + \beta{h_E}\bigg)||[u_h]||^2_{L^2(E)}} + \sum_{E \in \mathcal{E}(\zeta)} \frac{h_E}{\varepsilon}||[{\bf a}u_h]||^2_{L^2(E)},
\\
\bigg | \bigg | \frac{\partial u_{h,d}}{\partial t} \bigg | \bigg |^2 & \lesssim \sum_{E \in \mathcal{E}(\zeta)}{h_E\bigg|\bigg|\bigg[\frac{\partial{u_h}}{\partial{t}}\bigg]\bigg|\bigg|^2_{L^2(E)}}.
\end{aligned}
\end{equation} 
\end{theorem}
\begin{proof}
The first estimate follows from the definition of the norms, Theorem \ref{ncbounds} and the Cauchy-Schwarz inequality, viz.,
\begin{equation}
\begin{aligned}
\notag
(|||u_{h,d}|||+|u_{h,d}|_A)^2 & \lesssim \sum_{K \in \zeta} \Big( \varepsilon ||\nabla u_{h,d}||^2_{L^2(K)} + \beta ||u_{h,d}||^2_{L^2(K)} + \varepsilon^{-1}||{\bf a}u_{h,d}||^2_{L^2(K)} \Big ) \\
& + \sum_{E \in \mathcal{E}(\zeta)}{\bigg(\frac{\gamma\varepsilon}{h_E} + \beta{h_E}\bigg)||[u_h]||^2_{L^2(E)}} + \sum_{E \in \mathcal{E}(\zeta)} \frac{h_E}{\varepsilon}||[{\bf a}u_h]||^2_{L^2(E)} \\
& \lesssim \sum_{E \in \mathcal{E}(\zeta)}{\bigg(\frac{\gamma\varepsilon}{h_E} + \beta{h_E}\bigg)||[u_h]||^2_{L^2(E)}} + \sum_{E \in \mathcal{E}(\zeta)} \frac{h_E}{\varepsilon}||[{\bf a}u_h]||^2_{L^2(E)}.
\end{aligned}
\end{equation} 
The second estimate follows directly from Theorem \ref{ncbounds}.
\end{proof}
The error, $e := u-u_h$, is decomposed into a parabolic part $\rho$ and an elliptic part $\epsilon$, viz.,
 \begin{equation}
\notag
  e=\rho+\epsilon,\ \text{ with }\  \rho := u-w \  \text{ and }\ \epsilon := w-u_h.
 \end{equation}We further define $e_c:=u-u_{h,c}$ and $\epsilon_c := w - u_{h,c}$ to help facilitate the error analysis. We are now ready to state our a posteriori error estimator, $\eta$, given by
\begin{equation}
\notag
{\eta}:= \left [ ||e(0)||^2+\int_0^T \!{\eta}_{S_1}^2\, dt+\min\left\{\left(\int_0^T \! {\eta}_{S_2} \, dt \right)^2,{\alpha}^2_T\int_0^T \! {\eta}_{S_2}^2 \, dt\right\}+ \max_{0 \leq t \leq T} {\eta}^2_{S_3} \right ]^{1/2},
\end{equation}
where $\alpha_{T} := \min \left\{\varepsilon^{-1/2},\beta^{-1/2} \right \} $ with
\begin{equation}
\begin{aligned}
\notag
{\eta}_{S_1}^2 & := \sum_{K \in \zeta}\frac{h^2_K}{\varepsilon} \bigg|\bigg|f-\frac{\partial{u_h}}{\partial{t}}+\varepsilon\Delta{u_h}-{\bf a}\cdot \nabla{u_h}-bu_h \bigg|\bigg|^2_{L^2(K)}
+\sum_{E \in \mathcal{E}^{int}(\zeta)}{{\varepsilon}h_E||[\nabla{u_h}]||^2_{L^2(E)}}\\
&+\sum_{E \in \mathcal{E}(\zeta)}{\bigg(\frac{ \gamma \varepsilon}{h_E} + \beta{h_E}\bigg)||[u_h]||^2_{L^2(E)}} + \sum_{E \in \mathcal{E}(\zeta)} \frac{h_E}{\varepsilon} ||[{\bf a}u_h]||^2_{L^2(E)},
\\
 {\eta}_{S_2}^2 & :=  \sum_{E \in \mathcal{E}(\zeta)}{h_E\bigg|\bigg|\bigg[\frac{\partial{u_h}}{\partial{t}}\bigg]\bigg|\bigg|^2_{L^2(E)}},
\\
 {\eta}_{S_3}^2 & := \sum_{E \in \mathcal{E}(\zeta)}{h_E||[u_h]||^2_{L^2(E)}}.
\end{aligned}
\end{equation}
\medskip
\begin{theorem}\label{apost_sd_thm}
 The error the semi-discrete dG method~\eqref{dg_sd} satisfies the bound
\begin{equation}
\notag
||e||_* \lesssim {\eta}.
\end{equation}
\end{theorem}

\begin{proof}  We know that the exact solution satisfies
\begin{equation}
\bigg(\frac{\partial{u}}{\partial{t}},v \bigg)+B(t;u,v)=(f,v) \qquad \forall v \in H^1_0(\Omega).
\end{equation}
Using Definition \ref{reconstrdef1} we obtain 
\begin{equation}
\bigg(\frac{\partial{e}}{\partial{t}},v \bigg)+B(t;\rho,v)=0 \qquad \forall v \in H^1_0(\Omega).
\end{equation}
Setting $v=e_c$ in the above equation gives
\begin{equation}
\bigg(\frac{\partial{e_c}}{\partial{t}},e_c \bigg)+B(t;e_c,e_c)=\bigg(\frac{\partial{u_{h,d}}}{\partial{t}},e_c \bigg)+B(t;\epsilon,e_c)+B(t;u_{h.d},e_c).
\end{equation}
Using the elliptic reconstruction property of $\epsilon$ along with Theorem \ref{elliptic_apost} yields
\begin{equation}
B(t;\epsilon,e_c) \lesssim \eta_{S_1}|||e_c|||,
\end{equation}
while using the continuity of the bilinear form $B$  with Theorem \ref{nonconforming_bound} gives 
\begin{equation}
B(t;u_{h.d},e_c) \lesssim (|||u_{h,d}||| + |u_{h,d}|_A)|||e_c||| \lesssim \eta_{S_1}|||e_c|||.
\end{equation}
Combining these results and using the coercivity of the bilinear form $B$, the Cauchy-Schwarz inequality, Young's inequality and Theorem \ref{nonconforming_bound} gives
\begin{equation}\label{first_sd_bound}
 \frac{d}{dt}\big(||e_c||^2 \big)+|||e_c|||^2 \lesssim \eta^2_{S_1}+\eta_{S_2}||e_c||.
\end{equation}
Let $T_0\in[0,T]$ be such that $ E_c:=||e_c(T_0)||=||e_c||_{L^{\infty}(0,T;L^2(\Omega))}$ then integrating \eqref{first_sd_bound} over $[0,T_0$] and using Young's inequality gives
\begin{equation}
\label{qq1}
E^2_c \lesssim ||e_c(0)||^2+\int_0^T \! \eta^2_{S_1} \, dt+\Bigg ( \int_0^T \! \eta_{S_2} \, dt \Bigg )^2.
\end{equation}
Going back to \eqref{first_sd_bound} and integrating over $[0,T]$ yields
\begin{equation}
\label{qq2}
\int_0^T \! |||e_c|||^2 \, dt \lesssim ||e_c(0)||^2+\int_0^T \! \eta^2_{S_1} \, dt+E_c \Bigg (\int_0^T \! \eta_{S_2} \, dt \Bigg ).
\end{equation}
Adding \eqref{qq1} and \eqref{qq2} then using Young's inequality we obtain
\begin{equation}
\label{pp1}
||e_c||^2_* \, \lesssim ||e_c(0)||^2+\int_0^T \! \eta^2_{S_1} \, dt+\Bigg (\int_0^T \! \eta_{S_2} \, dt \Bigg )^2.
\end{equation}
Going back to \eqref{first_sd_bound}, integrating over $[0,T_0]$ and $[0,T]$ then summing the results yields
\begin{equation}
\label{second_sd_bound}
||e_c||^2_* \, \lesssim ||e_c(0)||^2+\int_0^T \! \eta^2_{S_1} \, dt+\int_0^T \! \eta_{S_2}||e_c|| \, dt.
\end{equation}
Observing that $||e_c||$ is contained in $|||e_c|||$ up to the coefficient $\beta$, then using the Cauchy-Schwarz inequality and Young's inequality gives
\begin{equation}
\label{pp2}
||e_c||^2_* \, \lesssim ||e_c(0)||^2+\int_0^T \! \eta^2_{S_1} \, dt+\beta^{-1}\int_0^T \! \eta^2_{S_2} \, dt.
\end{equation}
Going back to \eqref{second_sd_bound}, we can instead use the Poincar\'{e}-Friedrichs inequality to bound $||e_c||$ by $||\nabla e_c||$ then observe that $||\nabla e_c||$ is also part of $|||e_c|||$ up to the coefficient $\varepsilon$, thus using the Cauchy-Schwarz inequality and Young's inequality gives
\begin{equation}
\label{pp3}
||e_c||^2_* \, \lesssim ||e_c(0)||^2+\int_0^T \! \eta^2_{S_1} \, dt+\varepsilon^{-1}\int_0^T \! \eta^2_{S_2} \, dt.
\end{equation}
Putting \eqref{pp1}, \eqref{pp2} and \eqref{pp3} together we obtain
\begin{equation}
||e_c||^2_* \, \lesssim ||e_c(0)||^2+\int_0^T \! \eta^2_{S_1} \, dt + \min\left\{\left(\int_0^T \! {\eta}_{S_2} \, dt \right)^2,{\alpha}^2_T\int_0^T \! {\eta}_{S_2}^2 \, dt\right\}.
\end{equation}
Obviously from the triangle inequality we have
\begin{equation}
||e||_*^2 \lesssim ||e_c||^2_* + ||u_{h,d}||^2_*.
\end{equation}
Thus, all we need to do to complete the proof is to bound $||u_{h,d}||^2_*$ and $||e_c(0)||^2$. To bound $\displaystyle ||u_{h,d}||^2_*$, we use the definition of the norm along with Theorem \ref{ncbounds} and Theorem \ref{nonconforming_bound}, viz.,
\begin{equation}
||u_{h,d}||_*^2 \lesssim \int_0^T \! |||u_{h,d}|||^2 \, dt + \max_{0 \leq t \leq T}||u_{h,d}||^2 \lesssim \int_0^T \! \eta_{S_1} \, dt + \max_{0 \leq t \leq T} \eta^2_{S_3}.
\end{equation}
Finally, $||e_c(0)||^2$ is bounded using the triangle inequality and Theorem \ref{ncbounds}, viz.,
\begin{equation}
||e_c(0)||^2 \lesssim ||e(0)||^2 + \max_{0 \leq t \leq T}||u_{h,d}||^2 \lesssim ||e(0)||^2 + \max_{0 \leq t \leq T} \eta^2_{S_3}.
\end{equation}
These bounds complete the proof.
\end{proof}

\end{subsection}

\begin{subsection}{An a posteriori bound for the fully-discrete method}\label{fd_apost_sec}
We now continue by applying to the fully-discrete setting the general framework presented in the previous subsection.

\begin{definition}\label{ellipticreconstr2} We define $A^k \in V_h^k$ to be the unique solution of the elliptic problem
\begin{equation}\label{ell_rec_fd}
\notag
B\big(t^{k};u_h^k,v_h^k \big)+K_h \big(u_h^k,v_h^k \big)=\big(A^k,v_h^k \big) \qquad
\forall v_h^k \in V_h^k.
\end{equation}
\end{definition}
\begin{remark}
\label{Akplusone}
For $k \geq 1$ we obtain from \eqref{dg_fd} that $\displaystyle A^{k+1} = I_h^{k+1}f^{k+1} - \frac{u_h^{k+1}-I_h^{k+1}u_h^k}{\tau_{k+1}}$ where $I_h^{k+1}$ is the $L^2$ projection operator onto $V_h^{k+1}$.
\end{remark} 

\begin{definition}\label{ellipticreconstr2} We define the elliptic reconstruction $w^{k} \in H^1_0(\Omega)$ to be the unique solution of the elliptic problem
\begin{equation}
\notag
B\big(t^{k};w^{k},v \big)=\big(A^k,v \big) \qquad
\forall v \in H^1_0(\Omega).
\end{equation}
\end{definition}
\begin{remark}
The dG discretisation of the equation in Definition \ref{ellipticreconstr2} is to find $w_h^k \in V_h^k$ such that
\begin{equation}
\notag
B\big(t^{k};w_h^k,v_h^k \big)+K_h \big(w_h^k,v_h^k \big)=\big(A^k,v_h^k \big) \qquad \forall v_h^k \in V_h^k.
\end{equation}
Using the definition of $A^k$, we obtain the equality
\begin{equation}
\notag
B\big(t^{k};w_h^k,v_h^k \big)+K_h \big(w_h^k,v_h^k \big)=B\big(t^{k};u_h^k,v_h^k \big)+K_h \big(u_h^k,v_h^k \big) \qquad \forall v_h^k \in V_h^k.
\end{equation}
Therefore $w_h^k = u_h^k$ and so $\displaystyle B\big(t^k;w^k - u_h^k,v\big)$ can be estimated using Theorem \ref{elliptic_apost}.
\end{remark}

At each time step $k$, we decompose the dG solution $u_h^k$ into a conforming part $u_{h,c}^k \in H^1_0(\Omega) \cap V_h^k$ and a non-conforming part $u_{h,d}^k \in V_h^k$ such that $u_h^k = u_{h,c}^k + u_{h,d}^k$. Given $t \in \big(t^k,t^{k+1} \big]$, we (re)define $u_h(t)$ to be the linear interpolant with respect to $t$ of the values $u_h^k$ and $u_h^{k+1}$, viz.,
\begin{equation}
\notag
u_h(t):=l_k(t)u_h^k+l_{k+1}(t)u_h^{k+1},
\end{equation}
where $\{l_k, l_{k+1}\}$ denotes the standard linear Lagrange interpolation basis defined on the interval $\big[t^k,t^{k+1}\big]$.
 We define $u_{h,c}(t)$ and $u_{h,d}(t)$ analogously. We can then decompose the error $e:=u-u_h=e_c-u_{h,d}$ where $e_c:=u-u_{h,c}$. It will also be useful to define $\epsilon^{k} := w^{k}-u_h^{k}$. 

\begin{remark}
Aside from the numerical method itself and the bilinear forms appearing in Definition \ref{ellipticreconstr2} then given $t \in \big(t^k,t^{k+1} \big]$ all bilinear forms appearing in the error analysis are assumed to take place over the union triangulation $\zeta^k \cup \zeta^{k+1}$. The norms $|||\cdot|||$ and $|\cdot|_A$ (and thus by extension $||\cdot||_*$) are all evaluated over the union triangulation.
\end{remark}

\begin{lemma}\label{error_eq_simple_fd}
For almost any $t \in \big(t^k,t^{k+1}\big]$ we have 
\begin{equation}\label{error_relation_fd}
\notag
\bigg(\frac{\partial{e}}{\partial{t}},v\bigg)+B\big(t;e,v \big)=\big(f-f^{k+1},v\big)+\bigg(f^{k+1}-\frac{\partial{u_h}}{\partial{t}},v\bigg)-B\big(t;u_h,v \big) \qquad \forall v \in H^1_0(\Omega).
\end{equation} 
\end{lemma}
\begin{proof}
This follows from \eqref{model_weak}. 
\end{proof}

\noindent Before proving the a posteriori bounds for the fully-discrete method, we introduce the error estimators. We begin by defining the \emph{spatial estimator}, $\eta_S$, by
\begin{equation}
\begin{aligned}
\notag
\eta^2_S & := ||e(0)||^2+\frac{1}{3}\sum_{j=0}^{n-1}\tau_{j+1}\big(\eta_{S_1,j}^2+\eta_{S_1,j+1}^2 \big) + \sum_{j=0}^{n-1} \tau_{j+1}\eta^2_{S_2,j+1}+ \max_{0 \leq j \leq n} \eta^2_{S_3,j}
\\&+\min\left\{\left(\sum_{j=0}^{n-1}\tau_{j+1} \eta_{S_4,j+1} \right)^2, {\alpha}_T^2 \sum_{j=0}^{n-1} \tau_{j+1} \eta_{S_4,j+1}^2 \right\},
\end{aligned}
\end{equation}
where
\begin{equation}
\begin{aligned}
\notag
 \eta_{S_1,j}^2 & :=  \sum_{K \in\zeta^{j}}{\frac{h^2_K}{\varepsilon}\big|\big|A^j+\varepsilon\Delta u_h^{j}}-{\bf a}^{j} \cdot \nabla u_h^{j}-b^{j}u_h^{j}\big|\big|_{L^2(K)}^2+\sum_{E \in \mathcal{E}(\zeta^{j})}\frac{h_E }{\varepsilon}\big|\big|\big[{\bf a}^j u_h^j \big]\big|\big|^2_{L^2(E)}\\&+ \sum_{E \in \mathcal{E}(\zeta^{j})}\bigg(\frac{\gamma\varepsilon}{h_E}+\beta h_E\bigg)\big|\big|\big[u_h^j \big]\big|\big|^2_{L^2(E)}   + \sum_{E \in \mathcal{E}^{int}(\zeta^{j})}{\varepsilon}h_E \big|\big|\big[\nabla u_h^{j}\big]\big|\big|^2_{L^2(E)},
\\
\eta_{S_2,j+1}^2 & := \sum_{K \in \zeta^j \cup \zeta^{j+1}} \frac{h^2_K}{\varepsilon}\bigg|\bigg|f^{j+1}-I_h^{j+1}f^{j+1}+\frac{u_h^j - I_h^{j+1}u_h^j}{\tau_{j+1}}\bigg|\bigg|^2_{L^2(K)},
\\
 \eta_{S_3,j}^2 & := \sum_{E \in \mathcal{E}(\zeta^j)}h_E\big|\big|\big[u_h^j\big]\big|\big|^2_{L^2(E)},
\\
 \eta_{S_4,j+1}^2 & := \sum_{E \in \mathcal{E}(\zeta^j \cup \zeta^{j+1})}h_E\bigg|\bigg|\bigg[\frac{u_h^{j+1}-u_h^j}{\tau_{j+1}}\bigg]\bigg|\bigg|^2_{L^2(E)}.
\end{aligned}
\end{equation}
The \emph{time (or temporal) estimator}, $\eta_T$, is given by 
\begin{equation}
\notag
\eta^2_T := \sum_{j=0}^{n-1} \int_{t^j}^{t^{j+1}} \! \eta_{T_1,j+1}^2 \, dt+ \min\left\{\left(\sum_{j=0}^{n-1} \int_{t^j}^{t^{j+1}} \! \eta_{T_2,j+1} \, dt \right)^2, {\alpha}_T^2 \sum_{j=0}^{n-1} \int_{t^j}^{t^{j+1}} \! \eta_{T_2,j+1}^2\, dt \right\},
\end{equation}
where
\begin{equation}
\begin{aligned}
\notag
 \eta_{T_1,j+1}^2 & :=  \sum_{K \in \zeta^j \cup \zeta^{j+1}} {\varepsilon}^{-1}\big|\big|l_{j+1}\big({\bf a}^{j+1}-{\bf a} \big)u_h^{j+1}+l_j \big({\bf a}^j-{\bf a} \big)u_h^j\big|\big|^2_{L^2(K)},
\\
 \eta_{T_2,j+1}^2 & :=   \sum_{K \in \zeta^j \cup \zeta^{j+1}}\big|\big|f-f^{j+1}+l_j\big(A^{j+1}-A^j \big)+l_{j}\big(b^{j}-b-\nabla \cdot {\bf a}^{j}+\nabla \cdot {\bf a} \big)u_h^{j}\\&+l_{j+1}\big(b^{j+1}-b-\nabla \cdot {\bf a}^{j+1}+\nabla \cdot {\bf a}\big)u_h^{j+1}\big|\big|^2_{L^2(K)}.
\end{aligned}
\end{equation}

\begin{theorem}\label{apost_fd_thm}
 The error of the fully-discrete method \eqref{dg_fd} satisfies the bound
\begin{equation}
\notag
||e||_* \lesssim \sqrt{\eta^2_S+\eta^2_T}.
\end{equation}
\end{theorem}
\begin{proof} From Lemma \ref{error_eq_simple_fd} and Definition \ref{ellipticreconstr2} we have
\begin{equation}
\begin{aligned}
\bigg(\frac{\partial{e}}{\partial{t}},v \bigg)+B\big(t;e,v \big) &=\big(f-f^{k+1},v \big)+B\big(t^{k+1};u_h^{k+1},v\big)-B\big(t;u_h,v\big)\\&+B\big(t^{k+1};\epsilon^{k+1},v\big)+\bigg(f^{k+1}-\frac{\partial{u_h}}{\partial{t}}-A^{k+1},v \bigg),
\end{aligned}
\end{equation}
which upon adding and subtracting  $\big(l_k\big(A^{k+1}-A^k \big),v \big)$ and using Remark \ref{Akplusone} gives
\begin{equation}
\begin{aligned}
& \bigg(\frac{\partial{e}}{\partial{t}},v \bigg)+B\big(t;e,v\big)=\big(f-f^{k+1}+l_k\big(A^{k+1}-A^k\big),v\big)-B\big(t;u_h,v \big)\\&+B\big(t^{k+1};u_h^{k+1},v \big)
+B\big(t^{k+1};\epsilon^{k+1},v\big)  -l_k B\big(t^{k+1};w^{k+1},v \big)\\&+l_kB\big(t^k,w^k,v \big)+\bigg(f^{k+1}-\frac{\partial{u_h}}{\partial{t}}-A^{k+1},v\bigg).
\end{aligned}
\end{equation}
Finally, we add and subtract $l_k B\big(t^{k+1};u_h^{k+1},v\big)$ and $l_k B\big(t^k;u_h^k,v\big)$ to obtain the primary error equation, viz.,
\begin{equation}
\begin{aligned}
\label{equation1}
&\bigg(\frac{\partial{e}}{\partial{t}},v \bigg)+B\big(t;e,v \big) =\big(f-f^{k+1}+l_k\big(A^{k+1}-A^k\big),v\big)-B\big(t;u_h,v \big)\\&+l_{k+1}B\big(t^{k+1};u_h^{k+1},v\big)+l_kB\big(t^k;u_h^k,v\big)+l_{k+1}B\big(t^{k+1};\epsilon^{k+1},v\big)\\&+l_kB\big(t^k;\epsilon^k,v\big)+\bigg(f^{k+1}-\frac{\partial{u_h}}{\partial{t}}-A^{k+1},v \bigg).
\end{aligned}
\end{equation}
By combining terms, using the definition of the bilinear form $B$ and the \mbox{ }\mbox{ }\mbox{ } Cauchy-Schwarz inequality; the first four terms give rise to the time estimator:
\begin{equation}
\begin{aligned}
&\big(f-f^{k+1}+l_k\big(A^{k+1}-A^k\big),v\big)+l_{k+1}B\big(t^{k+1};u_h^{k+1},v\big)+l_kB\big(t^k;u_h^k,v\big) \\ &-B\big(t;u_h,v \big) \leq \eta_{T_1,k+1}|||v|||+\eta_{T_2,k+1}||v||.
\end{aligned}
\end{equation}
The final term can be rewritten using Remark \ref{Akplusone} then bounded using Theorem \ref{projectionbounds} and the Cauchy-Schwarz inequality, viz.,
\begin{equation}
\begin{aligned}
\bigg(f^{k+1}-\frac{\partial{u_h}}{\partial{t}}-A^{k+1},v \bigg) & =\bigg(f^{k+1}-\frac{\partial{u_h}}{\partial{t}}-A^{k+1},v-I_h^{k+1}v \bigg) \\ &\lesssim \eta_{S_2,k+1}|||v|||.
\end{aligned}
\end{equation}
For the remaining terms, we use the elliptic reconstruction property together with Theorem \ref{elliptic_apost} to conclude that
\begin{equation}
\begin{aligned}
l_{k+1}B\big(t^{k+1};\epsilon^{k+1},v \big)+l_kB\big(t^k;\epsilon^k,v \big) \lesssim ( l_{k+1}\eta_{S_1,k+1} + l_k \eta_{S_1,k})|||v|||.
\end{aligned}
\end{equation}
Combining the above, setting $v=e_c$ and using \eqref{coercivity}, \eqref{continuity}, the Cauchy-Schwarz inequality and Young's inequality yields
\begin{equation}
\begin{aligned}
 \frac{d}{dt}\big(||e_c||^2 \big)+|||e_c|||^2  & \lesssim l^2_{k+1} \eta_{S_1,k+1}^2+l_k^2\eta_{S_1,k}^2+ \eta^2_{S_2,k+1}+|||u_{h,d}|||^2+|u_{h,d}|_{A}^2 \\&+\eta_{S_4,k+1}||e_c||+\eta_{T_1,k+1}^2 + \eta_{T_2,k+1}||e_c||.
\end{aligned}
\end{equation}
The proof then follows from Theorem \ref{nonconforming_bound} and by employing a bounding strategy identical to that used in Theorem \ref{apost_sd_thm}.
 \end{proof}

\begin{remark}
The spatial estimator is expected to be asymptotically robust with respect to $\varepsilon$ as the predominant terms are the same as in the elliptic case. For the pre-asymptotic case, one would need to work in stronger norms to achieve theoretical robustness. We note, however, that in all the numerical experiments below, the adaptive algorithm, implementing the estimators presented here, was able to arrive to quasi-optimal space-time mesh modifications. The temporal error estimator is also expected to be asymptotically robust with respect to $\varepsilon$ as the temporal data approximation error terms are all order two in time and the only order one temporal term is a difference of derivatives (from Remark \ref{Akplusone}) which is anticipated to be independent of $\varepsilon$ in the asymptotic regime. 
\end{remark}

\begin{remark}
The use of elliptic reconstruction is not essential to the proof of Theorem \ref{apost_sd_thm} and Theorem \ref{apost_fd_thm}; it is possible to derive the residual based a posteriori bounds directly albeit at the cost of a lengthier calculation. The advantage of using elliptic reconstruction in the proof lies in the fact that the space estimator can be easily modified to accommodate non-residual based elliptic error estimators. This, in turn, may offer improvements in robustness with respect to the P\'eclet number cf.~\cite{S08}.
\end{remark}
 
\end{subsection}
\end{section}

\begin{section}{An adaptive algorithm}\label{adaptive_sec}

An \emph{adaptive algorithm} is a computational procedure that seeks to use an error estimator, $\eta$, to try and minimise the error in some norm by appropriately reducing the discretisation parameters. For elliptic problems, such adaptive algorithms are relatively straightforward: a finite element solution is calculated on an initial mesh and its estimator evaluated then the regions of the mesh where the estimator is largest are targeted for refinement by the adaptive algorithm and the finite element solution is recalculated on this new mesh; the algorithm continues in this fashion until the error estimator is below a given \emph{tolerance}. For parabolic problems, the design of adaptive algorithms is far more challenging because it is unclear how the spatial and temporal components of the error estimator should be utilised. 

The algorithms currently in the literature \cite{CF04,EJ91,EJ95a,EJ95b,EJ95c,EJL98,JNT90,M98,PWW,SS05} consist of an \emph{initial condition tolerance} to control refinement of the coarse input mesh, a \emph{spatial refinement tolerance} to control mesh refinement, a \emph{spatial coarsening tolerance} to control mesh coarsening and a \emph{temporal tolerance} to control the length of each time interval. Typically, these algorithms focus on the use of these individual tolerances to force the error estimator below a given \emph{global tolerance}. However, it is not necessarily clear that this is the correct choice. Indeed, proving that an adaptive algorithm will terminate with the total estimator below a tolerance is not the same as showing that it produces a quasi-optimal distribution of time steps and mesh parameters. 

We shall introduce a new adaptive algorithm, based on that given in \cite{CF04}, with a different emphasis on the use of tolerances and we will show numerically that our adaptive algorithm reduces the error estimator that it utilises at the optimal rate with respect to the mesh parameters and the total number of time steps. The pseudocode for our algorithm is given in Algorithm 3.1 and is based on using different parts of the a posteriori estimator from Theorem \ref{apost_fd_thm}  to drive space-time adaptivity. It is useful to provide heuristic justifications for the approaches taken in our adaptive algorithm and to compare our adaptive algorithm to similar algorithms already in the literature:

\begin{algorithm} \label{qqmore}
  \begin{algorithmic}[1]
     \State {\bf Input:} $\varepsilon$, ${\bf a}$, $b$, $f$, $u_0$, $T$, $\Omega$, $n$, $\zeta^0$, $\gamma$, ${\tt ttol}$, ${\tt stol^+}$, ${\tt stol^-}$.
     \State  Set $\tau_1$, ..., $\tau_n = T/n$.
     \State Calculate $u_h^0$.
     \State Calculate $u_h^1$ from $u_h^0$.
    \While {$\displaystyle {\hat{\eta}}^2_{T,1} > {\tt ttol} \text{ OR } \max_K \eta^2_{S_1,1} |_K > {\tt stol^+} $}
    \State Modify $\zeta^0$ by refining all elements such that $\eta^2_{S_1,1} |_K > {\tt stol^+}$ and coarsening all elements such that $\eta^2_{S_1,1} |_K < {\tt stol^-}$.
\If   {$\hat{\eta}^2_{T,1}> {\tt ttol}$}
    
 \State $n \leftarrow n+1$.

\State $\tau_{n} = \tau_{n-1}$, ..., $\tau_{3} = \tau_{2}$.

 \State $\tau_{2} = \tau_{1}/2$.

 \State $\tau_{1} \leftarrow \tau_{1}/2$.

    \EndIf

     \State Calculate $u_h^0$.
\State Calculate $u_h^1$ from $u_h^0$.
    \EndWhile

\State Set $j = 1$,  $\zeta^1 = \zeta^0$, $time = \tau_{1}$.
    \While {  $time <  T$}
    \State Calculate $u_h^{j+1}$ from $u_h^j$.
    \While
     {$\hat{\eta}^2_{T,j+1}> {\tt ttol} $}
\If {$\hat{\eta}^2_{T,j+1}> {\tt ttol}$}
     
\State $n \leftarrow n+1$.

\State $\tau_{n} = \tau_{n-1}$, ..., $\tau_{j+3} = \tau_{j+2}$.

 \State $\tau_{j+2} = \tau_{j+1}/2$.

 \State $\tau_{j+1} \leftarrow \tau_{j+1}/2$.

\EndIf

\State Calculate $u_h^{j+1}$ from $u_h^j$.

    \EndWhile

    \State Create $\zeta^{j+1}$ from $\zeta^j$ by refining all elements such that $\eta^2_{S_1,j+1} |_K > {\tt stol^+}$ and coarsening all elements such that $\eta^2_{S_1,j+1} |_K < {\tt stol^-}$.
    \State Calculate $u_h^{j+1}$ from $u_h^j$.
   \State $time \leftarrow time + \tau_{j+1}$.
    \State $j\leftarrow j+1$.
    \EndWhile
  \end{algorithmic}
  \caption{Space-time adaptivity}
\end{algorithm}

\begin{itemize}

\item As in \cite{CF04}, our adaptive algorithm uses the dominant term in the space estimator, $\eta_{S_1,j+1}$, to control mesh refinement. All elements on which $\eta^2_{S_1,j+1}$ is larger than the \emph{spatial refinement threshold} ${\tt stol^+}$ are targeted for \mbox{ } refinement by the adaptive algorithm.

\item Most algorithms in the literature conduct mesh coarsening through the term $\eta_{S_2,j+1}$ (or equivalent) which is often referred to as a \emph{mesh-change indicator} {\bf --} this is because such a term is non-zero only on elements that have been subject to coarsening. This approach, however, comes with a serious disadvantage {\bf --} $\eta_{S_2,j+1}$ is spatially one order higher than $\eta_{S_1,j+1}$ which means such algorithms tend to be too conservative with regards to mesh coarsening. By contrast, our adaptive algorithm uses the dominant term in the space estimator, $\eta_{S_1,j+1}$, to control mesh coarsening as well as mesh refinement. In particular, all elements on which $\eta^2_{S_1,j+1}$ is smaller than the \emph{spatial coarsening threshold} ${\tt stol^-}$ are flagged for coarsening by the adaptive algorithm.

\item The nature of the time estimator, $\eta_T$, makes it inconvenient to use as a temporal refinement indicator so we define $\hat{\eta}_{T,j+1}$ given by
\begin{equation}
\label{jol5}
\hat{\eta}_{T,j+1}^2 := \int_{t^j}^{t^{j+1}} \! \eta_{T_1,j+1}^2 \, dt+ \min\{\alpha_T,T\}\int_{t^j}^{t^{j+1}} \! \eta^2_{T_2,j+1} \, dt,
\end{equation}
the sum of which bounds $\eta^2_T$. Our temporal strategy consists of continually halfing the time step length until $\hat{\eta}_{T,j+1}^2$ is smaller than the \emph{temporal threshold} ${\tt ttol}$ which is in contrast to the approach taken in \cite{CF04} where they instead seek to force the integrand in \eqref{jol5} below a given tolerance. We will show numerically that our approach yields a quasi-optimal distribution of time steps.

\item The final major difference between our proposed algorithm and those in the literature lies in how the coarse input mesh is dealt with. Typically, mesh refinement is carried out on the coarse mesh until the \emph{initial condition estimator}, $||e(0)||$, is smaller than a given tolerance. Such a term is, however,  spatially one order higher than $\eta_{S_1,1}$ which means using it for configuration of the input mesh leads to mass mesh refinement during the first time step and such a large amount of mesh change during one time step can destabilise the numerical scheme. Therefore, our adaptive algorithm continues to refine the coarse input mesh until $\eta^2 _{S_1,1}$ is smaller than the \emph{spatial refinement threshold} ${\tt stol^+}$ on every element.

\end{itemize}

\begin{remark}
Mesh modification must be done very carefully to ensure that the numerical solution does not degrade \cite{BKM13,D82}. Specifically, the spatial refinement threshold ${\tt stol^+}$ needs to be chosen sufficiently small in comparison to the temporal threshold ${\tt ttol}$. The spatial coarsening threshold ${\tt stol^-}$ also needs to be chosen sufficiently small in comparison to the spatial refinement threshold ${\tt stol^+}$ in order to avoid unnecessary refine and coarsen loops.
\end{remark}

\begin{remark}
We stress that the algorithm does not necessarily produce a \mbox{ }\mbox{ } monotonically decreasing time step distribution from $0$ to $T$. Indeed, the algorithm starts with an initial equispaced subdivision of $[0,T]$ into $n$ time intervals, which is then, possibly, locally bisected based on ${\tt ttol}$. For instance, if the solution reaches a smoothly varying steady state, the algorithm will retain the original (coarse) time step length of $T/n$ during the final stages of the computation.
\end{remark}

\end{section}

\begin{section}{Numerical experiments}\label{numerics_sec}

We shall numerically investigate the presented a posteriori bounds and the \mbox{ } performance of the adaptive algorithm through an implementation based on the {\tt deal.II} finite element library \cite{BHK07}. All the numerical experiments have been performed using the  high performance computing facility ALICE at the University of Leicester. 

In order to discuss the numerical results, we need some additional definitions. We shall begin by extending our notion of the effectivity index. Let the maximum meshsize be given by $\displaystyle h := \max_{0 \leq k \leq n} \max_{K \in \zeta^k} h_K$ and the time largest step length be given by $\displaystyle \Delta t :=  \max_{1 \leq k \leq n} \tau_{k}$. We then define the \emph{spatial effectivity index} by
\begin{equation}
\begin{aligned}
\notag
\text{spatial effectivity index} & := \lim_{\Delta t \rightarrow 0} \frac{\eta}{||e||},
\end{aligned}
\end{equation} 
whereas the \emph{temporal effectivity index} is given by
\begin{equation}
\begin{aligned}
\notag
\text{temporal effectivity index} & := \lim_{h \rightarrow 0} \frac{\eta}{||e||}.
\end{aligned}
\end{equation} 
These notions give us a way of measuring the contribution of the spatial and temporal estimators to the constants in \eqref{reliability}-\eqref{efficiency}. In particular, we can assess whether specific parts of the estimator are robust or not. The spatial effectivity indices can be observed in practise by choosing a very small temporal threshold while the temporal effectivity indices can be observed through use of a high polynomial degree and/or sufficiently fine spatial mesh.

We also need a notion of the average number of degrees of freedom so we can discuss spatial convergence rates. If the total number of degrees of freedom on the union mesh $\zeta^k \cup \zeta^{k+1}$ is denoted by $\lambda_k$ then the weighted degrees of freedom of the problem is given by
\begin{equation}
\notag
\text{Weighted Average DoFs} := \frac{1}{T}\sum_{j=0}^{n-1} \tau_{j+1}\lambda_j.
\end{equation}
In all examples presented below, unless otherwise stated, we use polynomials of degree two and an initial $4\times4$ uniform quadrilateral mesh. We also set the spatial coarsening parameter to ${\tt stol^-} = 0.001*{\tt stol^+}$. Finally, unmarked lines in convergence plots represent the theoretically expected optimal rate of convergence for reference purposes.

\begin{subsection}{Example 1}

\noindent Let $\Omega = (0,1)^2$, ${\bf a}=(1,1)^T$, $b=0$, $u_0=0$, $T=10$ and select the function $f$ so that the exact solution to problem \eqref{model_weak} is given by
\begin{equation}
\notag
u(x,y,t)=\big(1-e^{-t}\big)\bigg(\frac{e^{(x-1)/\varepsilon}-1}{e^{-1 / \varepsilon}-1}+x-1 \bigg)\bigg(\frac{e^{(y-1)/ \varepsilon}-1}{e^{- 1/ \varepsilon}-1}+y-1 \bigg).
\end{equation}
The solution exhibits boundary layers at the outflow boundary of the domain of width $\mathcal{O}(\varepsilon)$ as well as a temporal boundary layer.

We begin by fixing a temporal threshold that produces enough time steps so that the temporal contribution to the error is very small in comparison to the spatial contribution. The spatial threshold is then gradually reduced to observe the spatial effectivity indices for this problem which are given in Figure \ref{ex1rates}.
Optimal rates of convergence are observed with respect to the weighted average degrees of freedom for both the estimator and the error but are omitted in this example. As shown, the effectivity indices are bounded asymptotically and remain between five and ten for the different values of $\varepsilon$; these are directly comparable to those observed in \cite{SZ09} for the stationary problem.

In order to study the temporal effectivity indices for this problem any boundary layers must be fully resolved so that the spatial error is dominated by the temporal one. To this end, we use a high polynomial degree and  specially constructed anisotropic meshes in order to ensure that the spatial error is sufficiently small. The temporal threshold is then reduced to observe the temporal effectivity indices of the problem which are given in Figure \ref{ex1rates}. Optimal order is observed in both the estimator and the error and the effectivity indices remain bounded between four and seven for all values of $\varepsilon$.

\begin{figure}[t]
\centering
\includegraphics[scale=0.48]{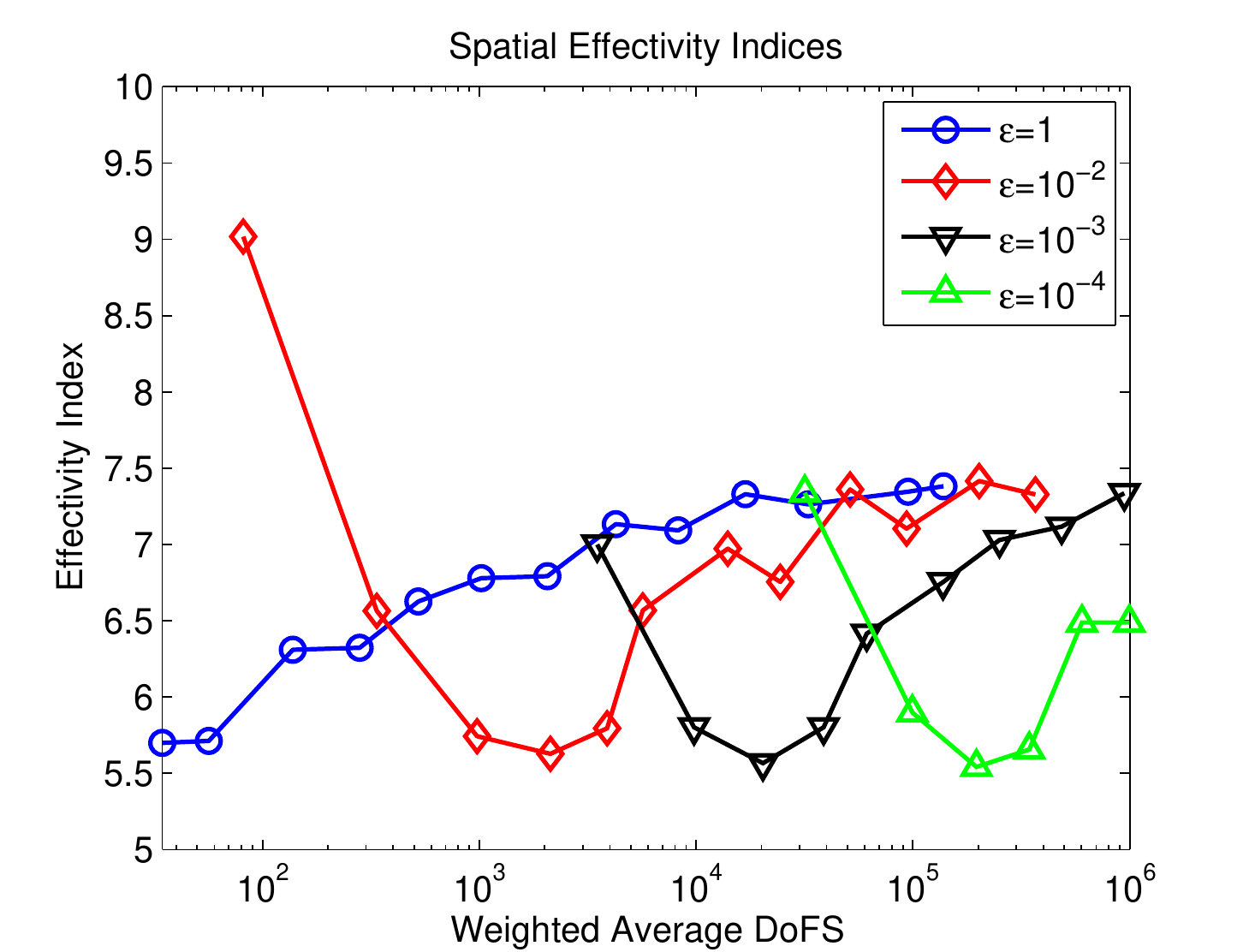} \includegraphics[scale=0.48]{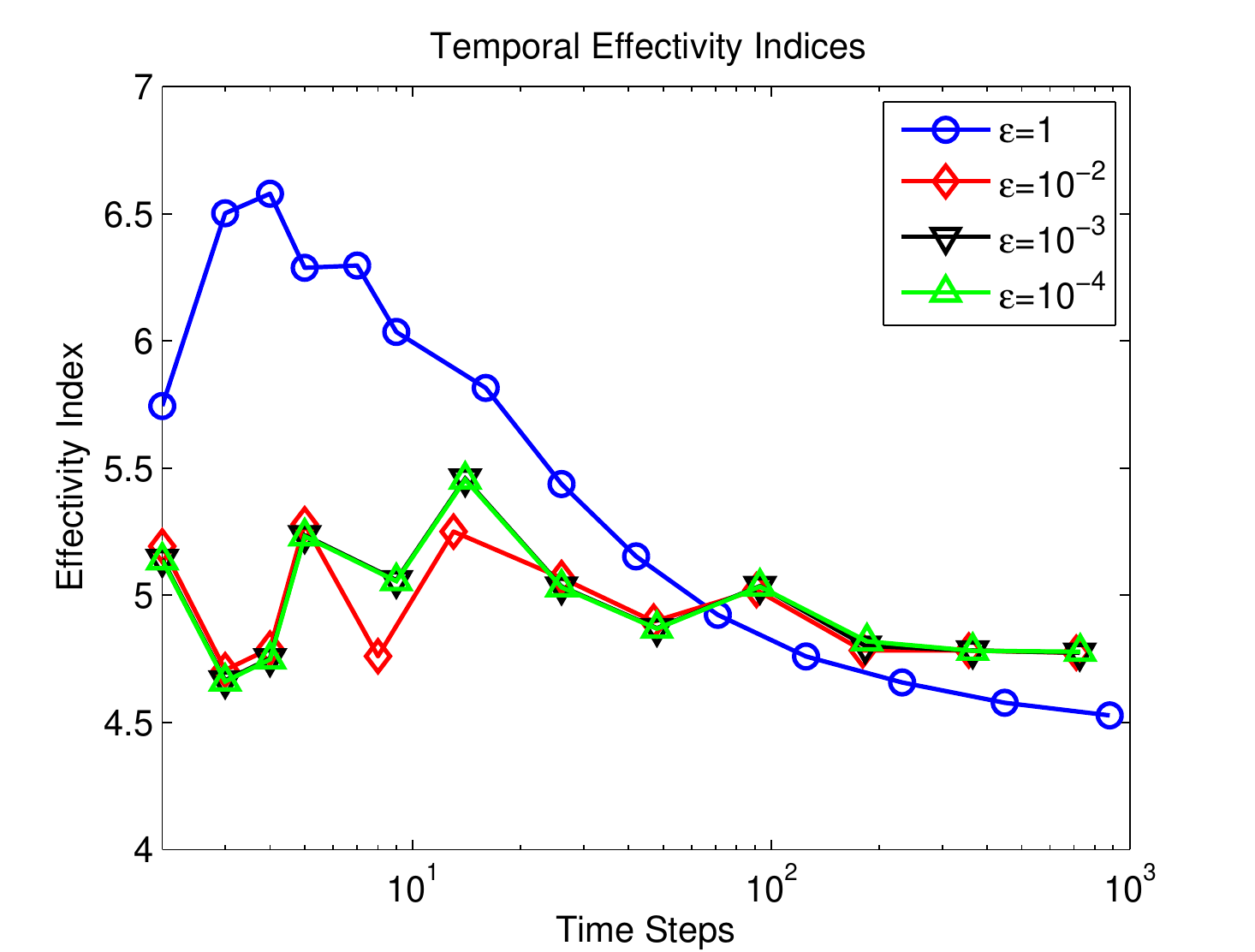}
\caption{Example 1: Spatial and temporal effectivity indices.}
\label{ex1rates}
\end{figure}
\begin{figure}
\centering
\includegraphics[scale=0.48]{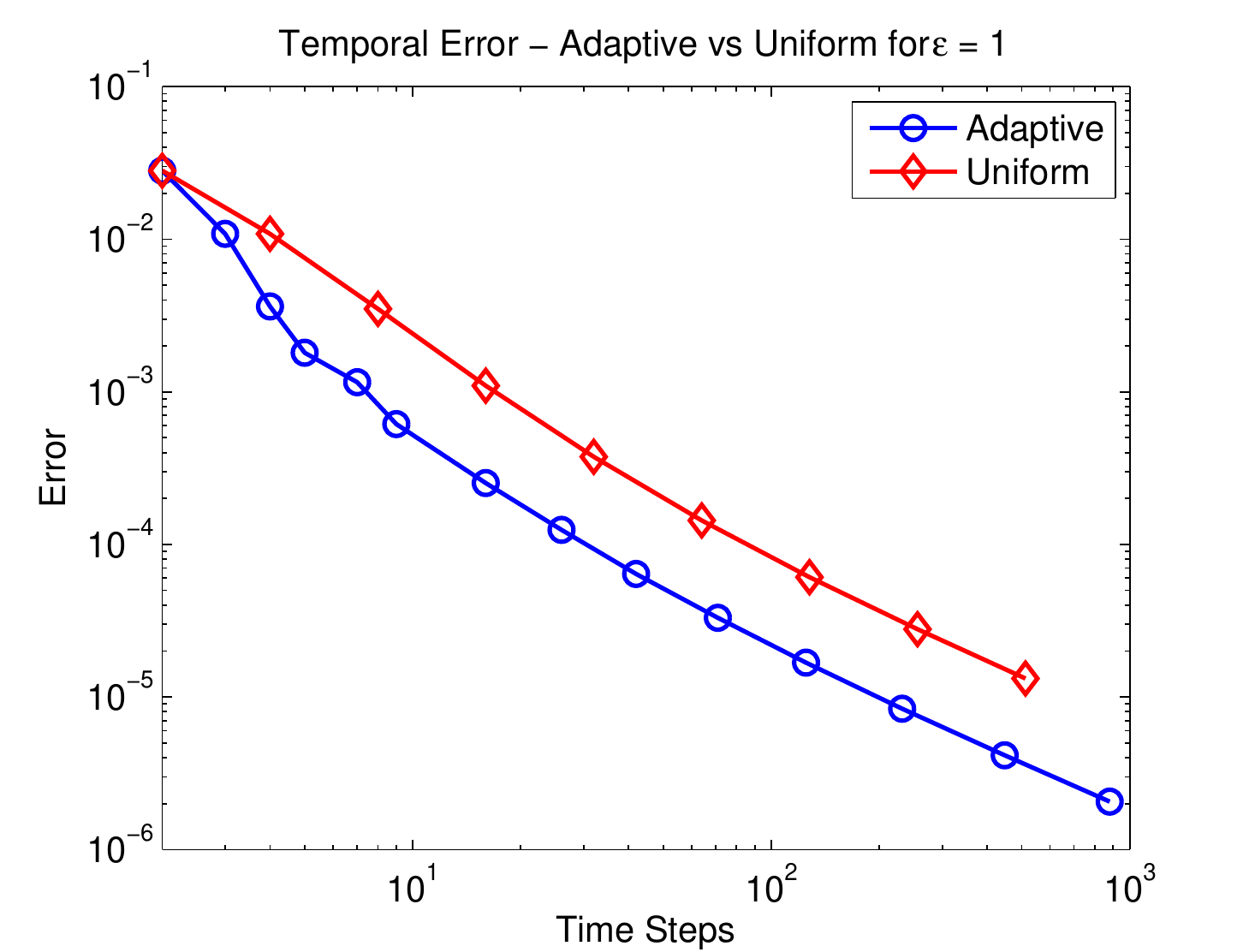} \includegraphics[scale=0.48]{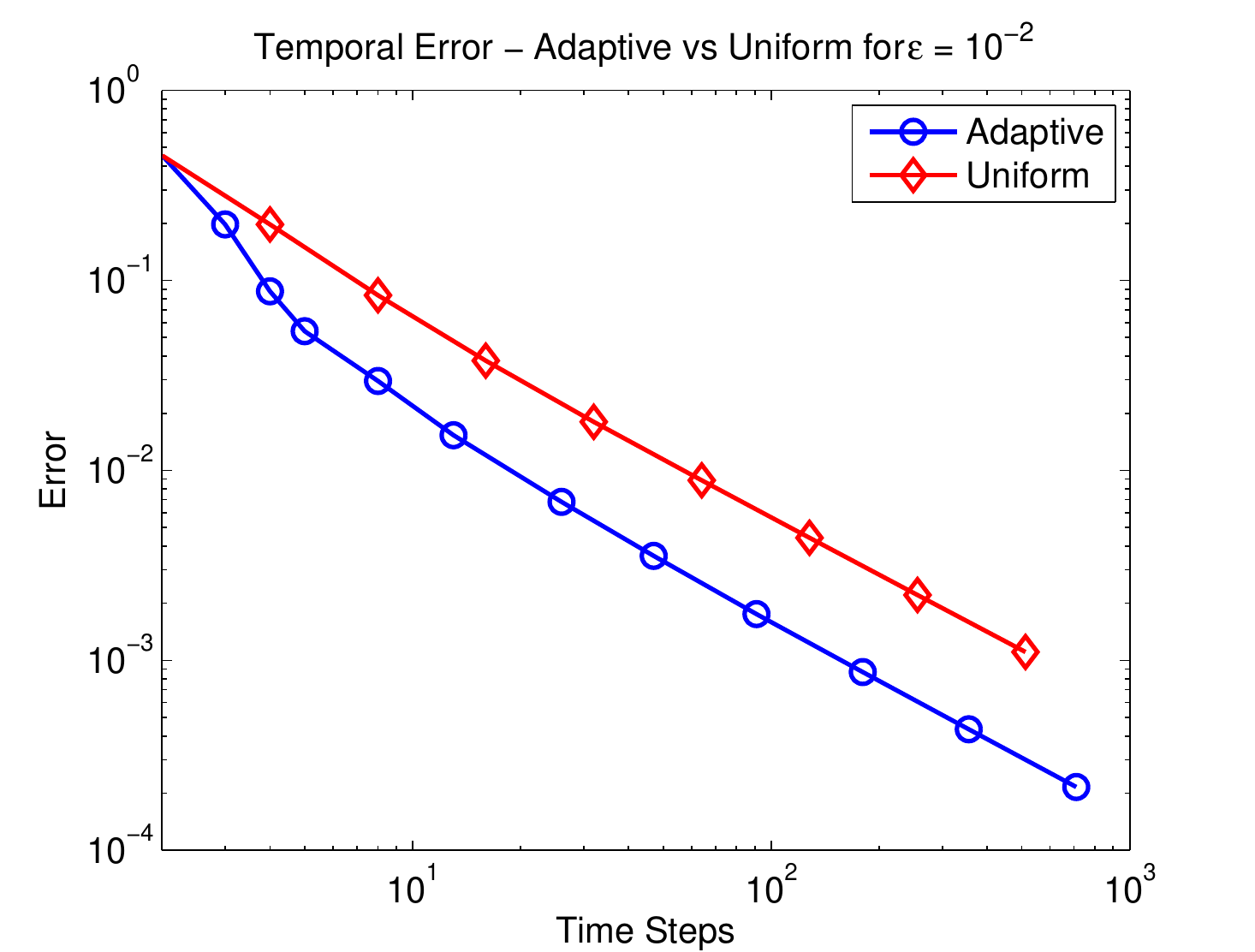}
\caption{Example 1: Temporal error comparison under adaptive and uniform time-stepping for $\varepsilon=1$ and $\varepsilon=10^{-2}$.}
\label{ex1errorcomparison}
\end{figure}

The presence of a temporal boundary layer in the solution motivates a \mbox{ } comparison between adaptive and uniform time-stepping. To this end, a sufficiently small spatial threshold is chosen so that the spatial contribution to the error is small and then the temporal threshold is decreased and the results are compared to just using uniform time-stepping. The results given in Figure \ref{ex1errorcomparison} show that the temporal strategy of the adaptive algorithm minimises the temporal portion of the error better than just using uniform time stepping.

\end{subsection}

\begin{subsection}{Example 2}

\begin{figure}[t]
\centering
\includegraphics[scale=0.48]{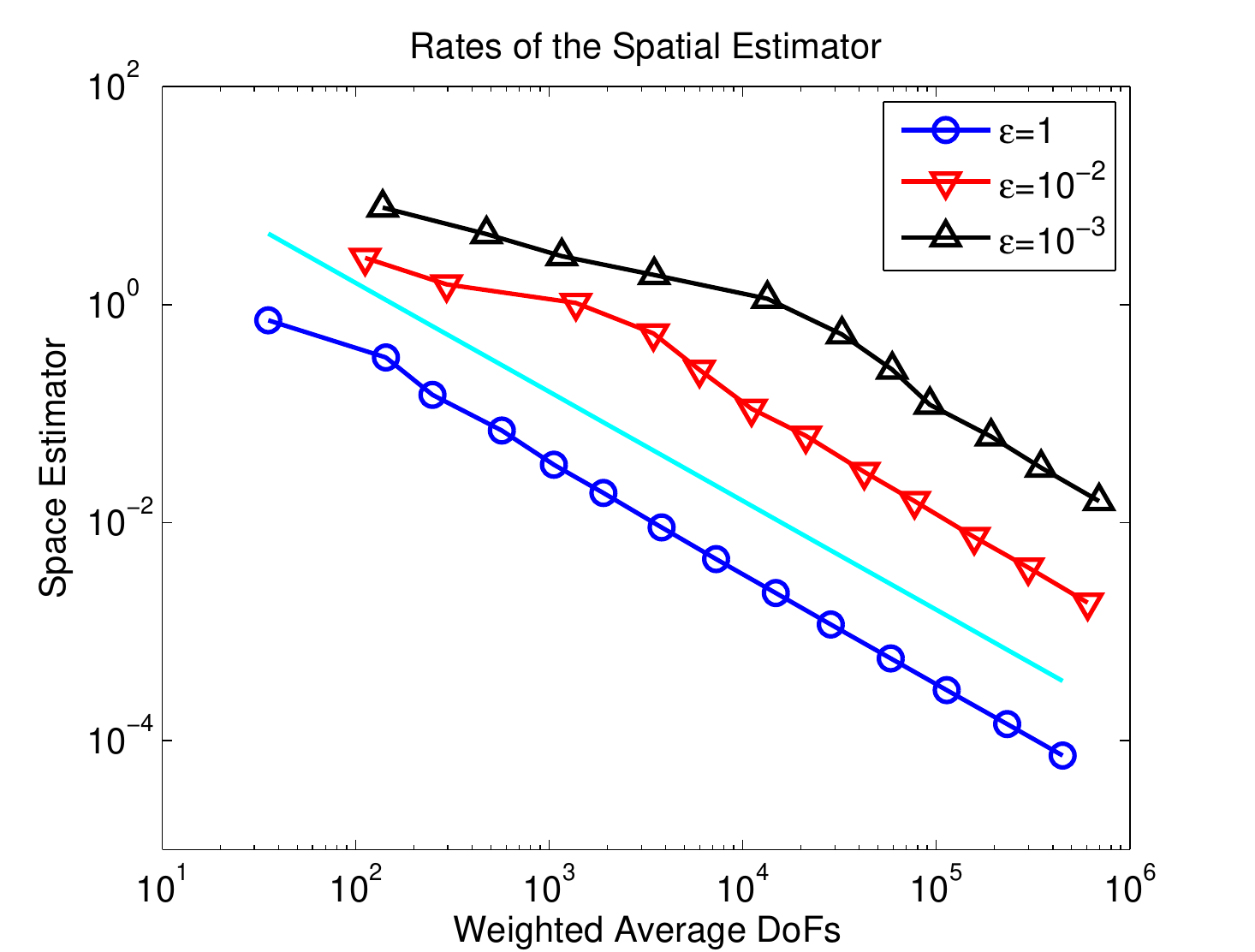} \includegraphics[scale=0.48]{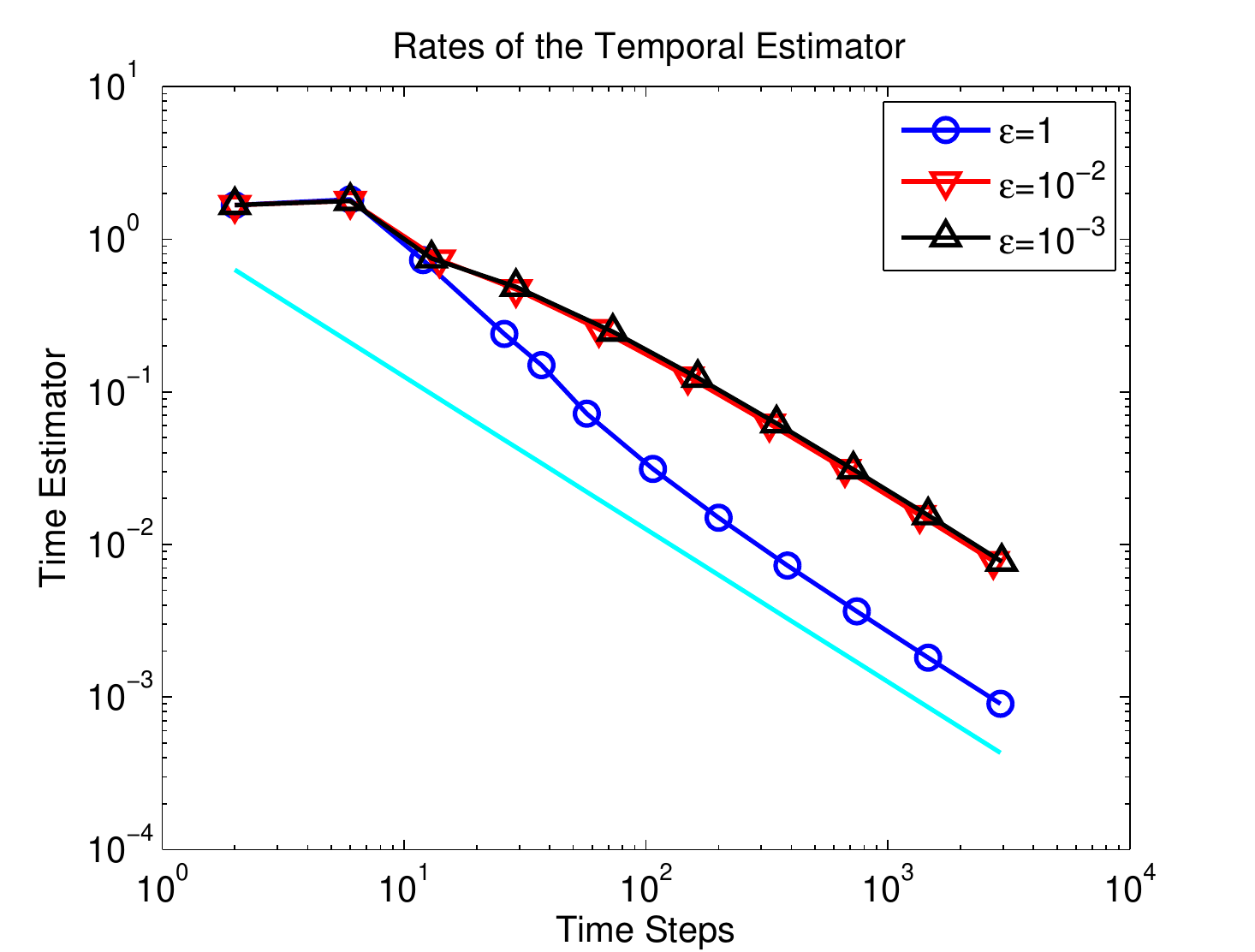}
\caption{Example 2: Spatial and temporal rates.}
\label{ex2rates}
\end{figure}
\begin{figure}
\centering
\includegraphics[scale=0.48]{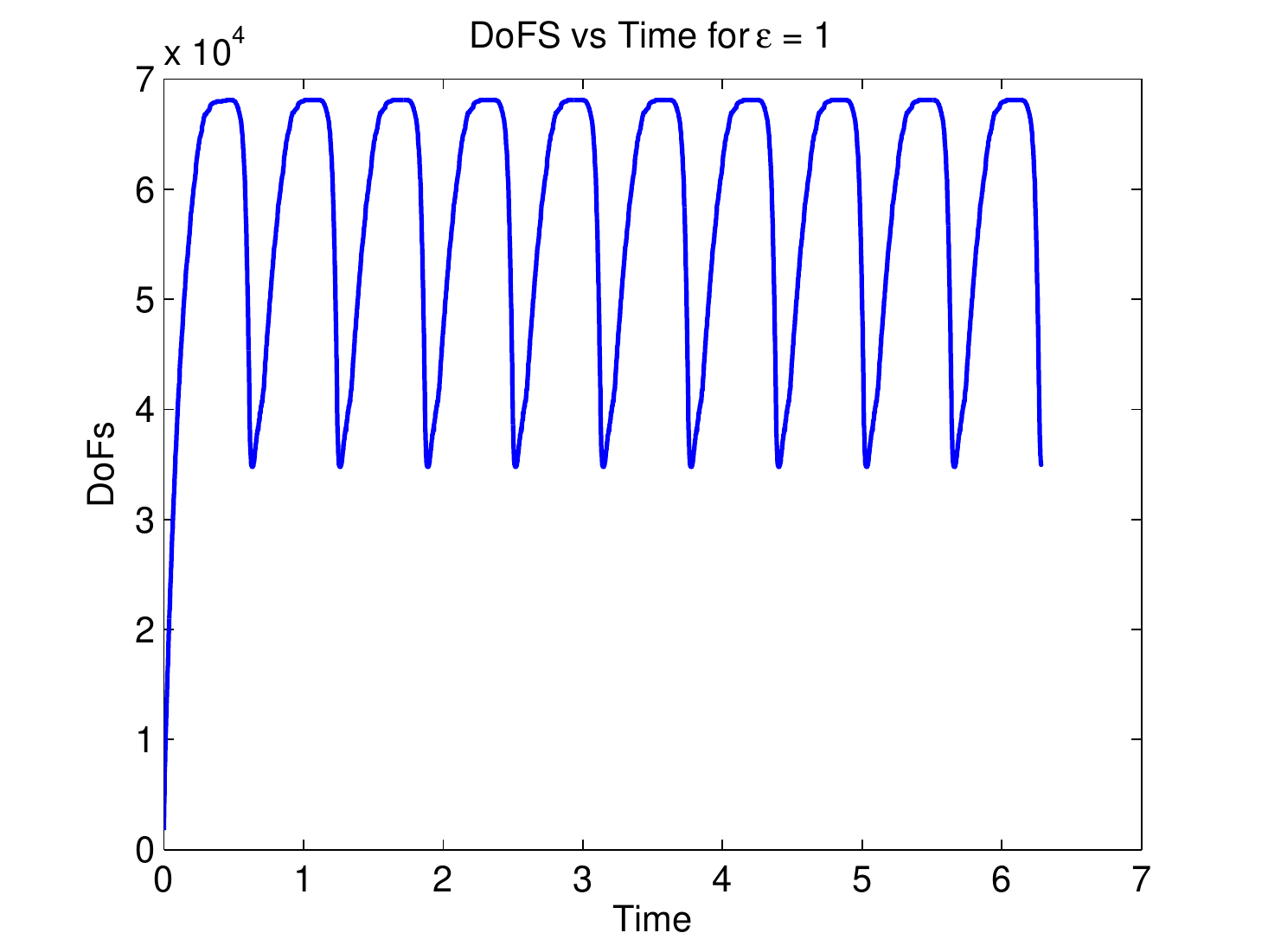} \includegraphics[scale=0.48]{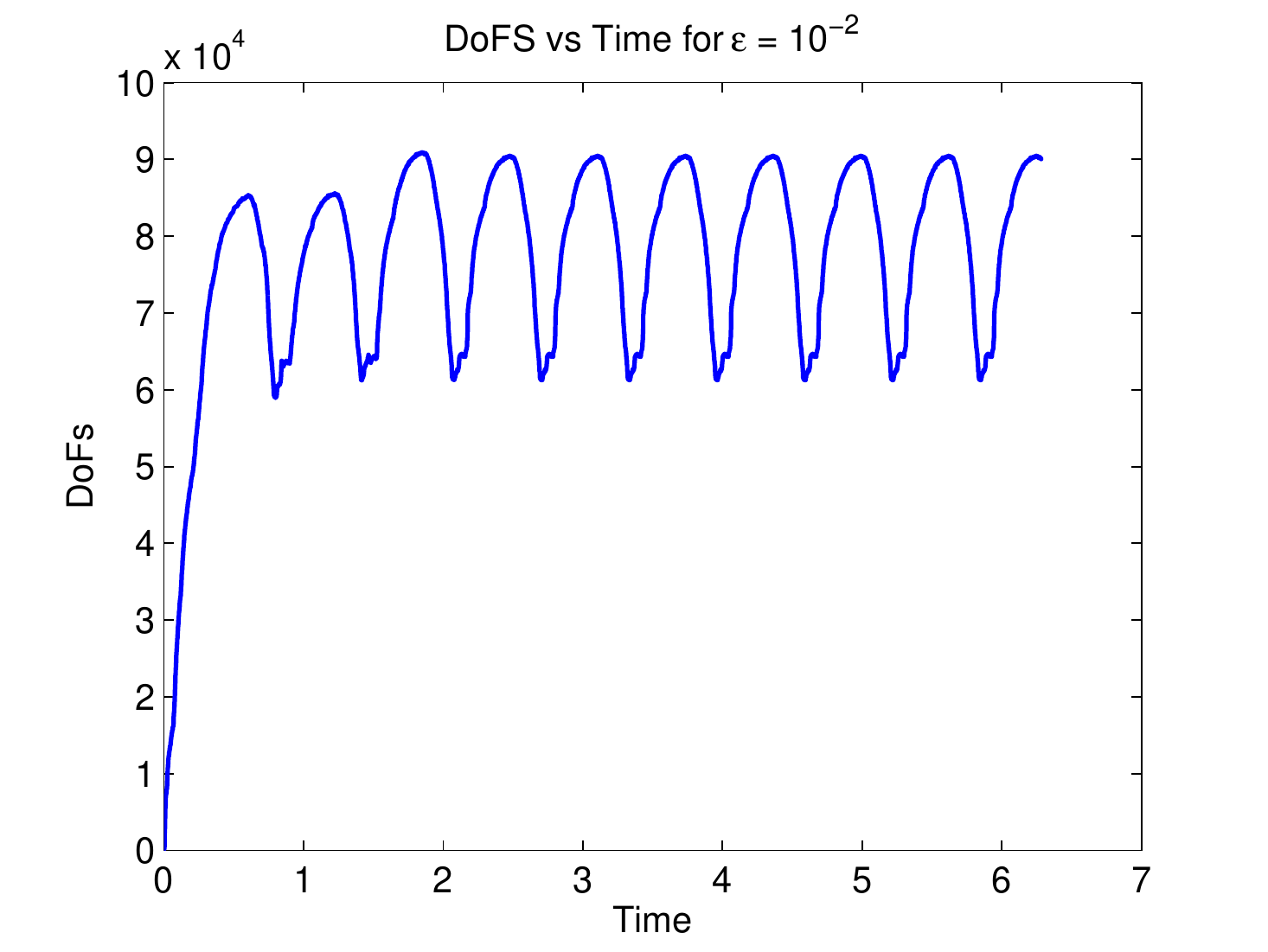}
\caption{Example 2: DoFS vs Time for $\varepsilon=1$ and $\varepsilon=10^{-2}$.}
\label{ex2dofs}
\end{figure}

We set $\Omega=(-1,1)^2$, ${\bf a}=(1,1)^T$, $b=1$, $f=\sin(5t)xy$, $u_0=0$ and $T = 2\pi$. The solution exhibits layers of width $\mathcal{O}(\varepsilon)$ in the proximity of the outflow boundary and is oscillatory in time. The sharpness of the boundary layers depend on time, thus making this a good test of the ability of the algorithm to add and remove degrees of freedom.

As in Example 1, we begin by fixing a temporal threshold while decreasing the spatial threshold to observe the rates of convergence for the space estimator. We then set a spatial threshold small enough to resolve any boundary layers, while reducing the temporal threshold to observe the rates of the time estimator. The results are displayed in Figure \ref{ex2rates}. Optimal rates of convergence are observed for both the space and time estimators.

\begin{figure}
\centering
\includegraphics[scale=0.48]{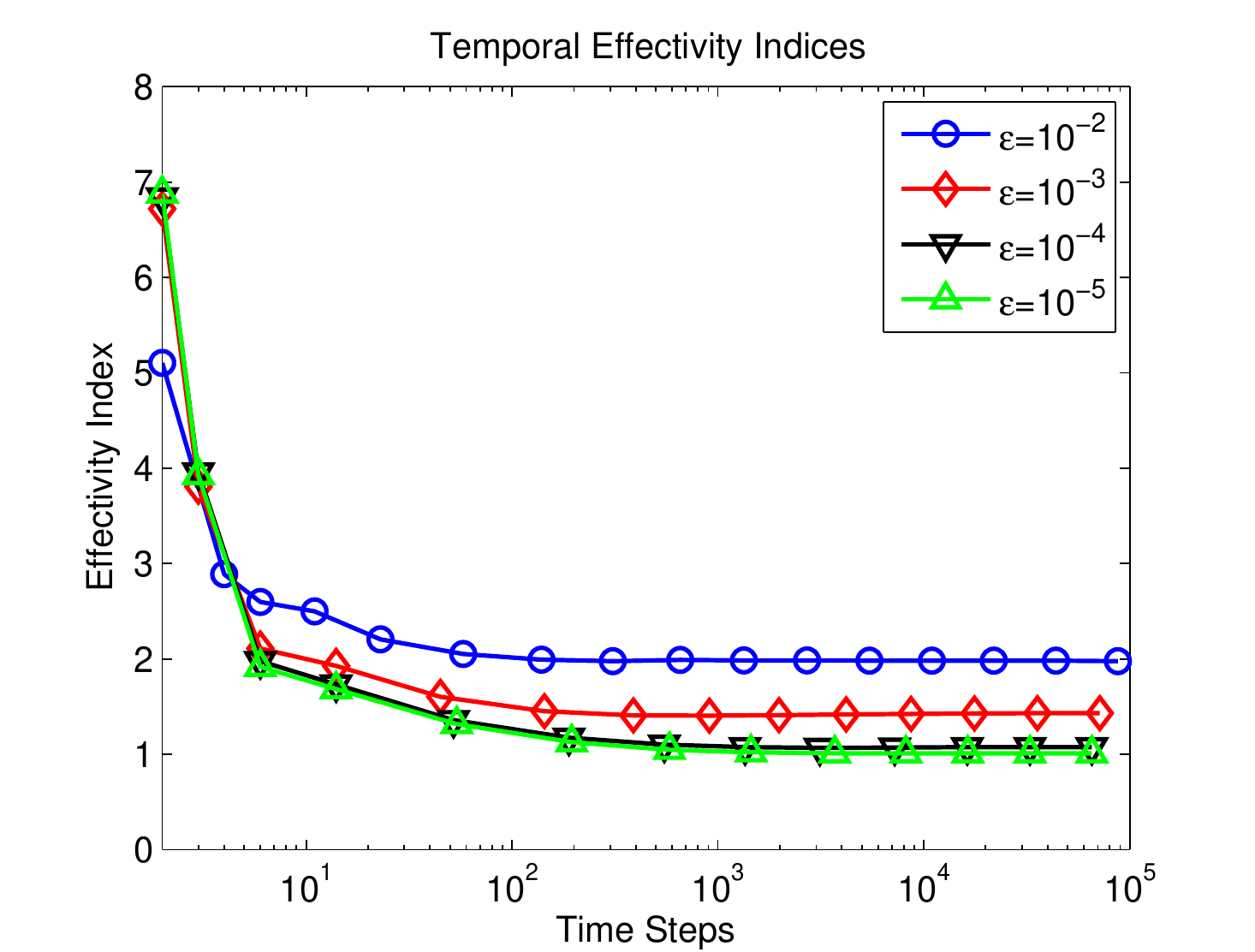} \includegraphics[scale=0.48]{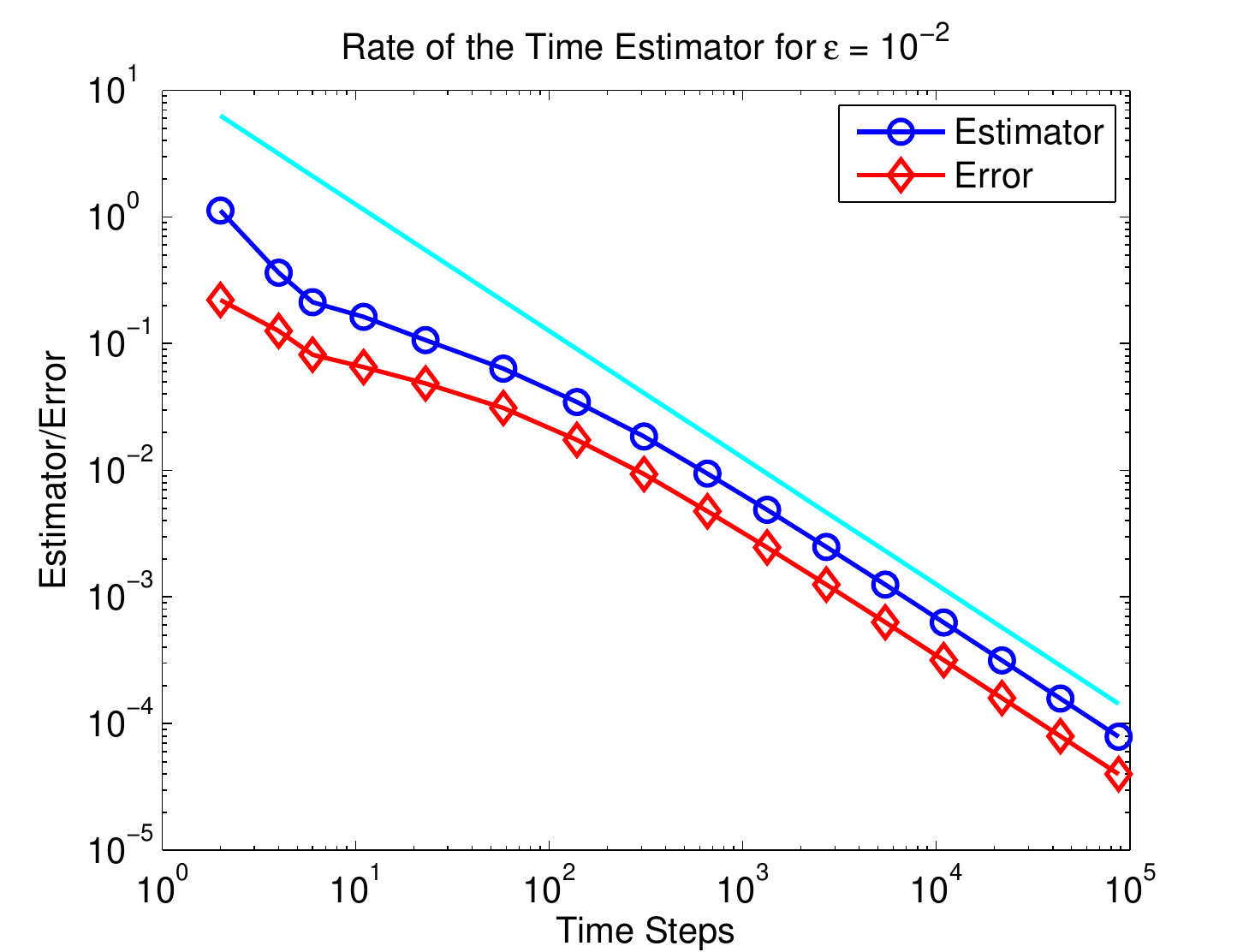}

\caption{Example 3: Temporal effectivity indices and the rate of the time estimator for $\varepsilon=10^{-2}$.}
\label{ex3rates}
\end{figure}

To assess the mesh change driven by the adaptive algorithm we also plot the individual degrees of freedom on each mesh against time for a given spatial and temporal threshold. The results are given in Figure \ref{ex2dofs}. We observe that the adaptive algorithm is adding and removing degrees of freedom at a rate that is in accordance with the oscillating nature of the solution driven by the sinusoidal forcing $f$. 

\end{subsection}

\begin{subsection}{Example 3}

 Let $\Omega=(-2,2)^2$,  $T = 2\pi$, ${\bf a}=(y,-x)^T$, $b=0$, $f=0$ and $u_0=e^{-64(x-0.5)^2}e^{-64y^2}$.
  The PDE convects the initial two dimensional Gaussian profile along the circular wind while diffusing it at a rate depending upon $\varepsilon$. In particular, provided the error at the boundary is sufficiently small, the exact solution to problem \eqref{model_weak} is given by
\begin{equation}
\notag
u(x,y,t)=\frac{1}{1+256\varepsilon t}\exp \bigg(-\frac{64(x-0.5\cos(t))^2}{1+256\varepsilon t} \bigg)\exp \bigg(-\frac{64(y+0.5\sin(t))^2}{1+256\varepsilon t} \bigg).
\end{equation}

To observe the temporal effectivity indices and temporal rates of the problem we first fix a spatial threshold so that the spatial contribution to the error is small and then reduce the temporal threshold; the results given in Figure \ref{ex3rates} show that the temporal effectivity indices are bounded and remain between one and eight for all values of $\varepsilon$ and that the optimal rate of convergence is achieved by both the error and the estimator. Some meshes at various time steps produced by the algorithm for $\varepsilon=10^{-5}$ are displayed in Figure \ref{ex3grids} and show that the adaptive algorithm is adding and removing degrees of freedom efficiently.

\begin{figure}[t]
\centering
\includegraphics[scale=0.56]{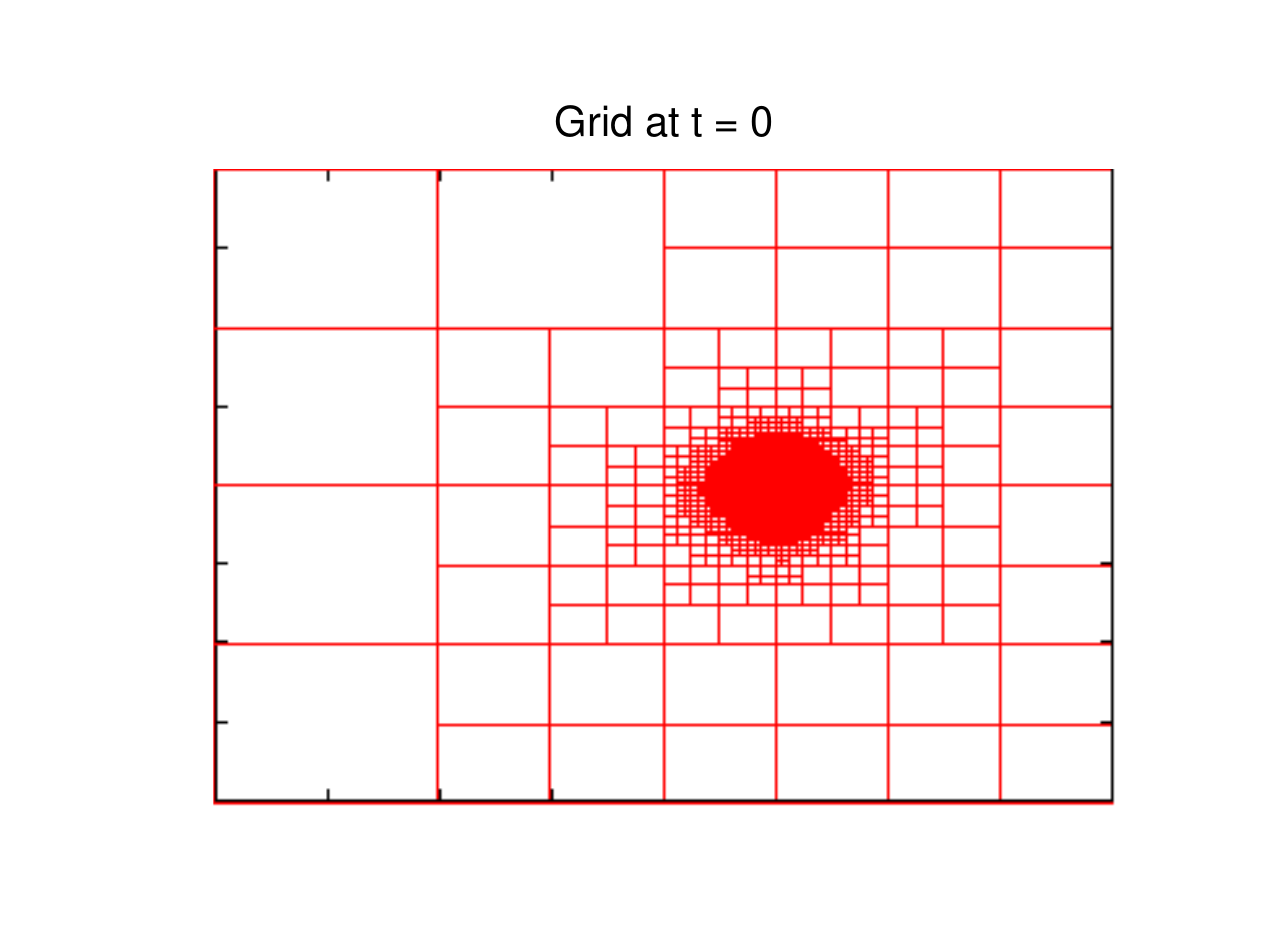} \includegraphics[scale=0.56]{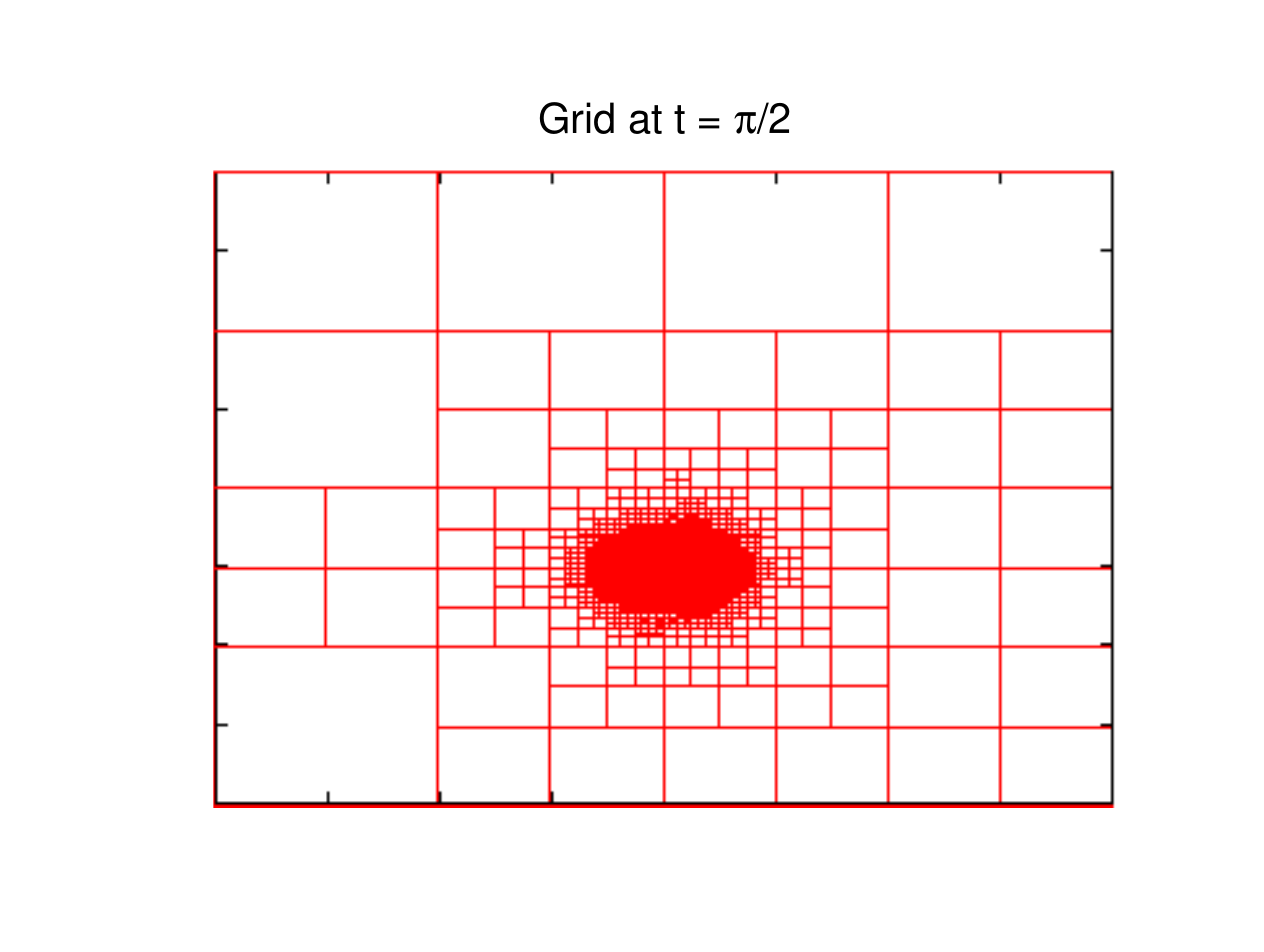}

\includegraphics[scale=0.56]{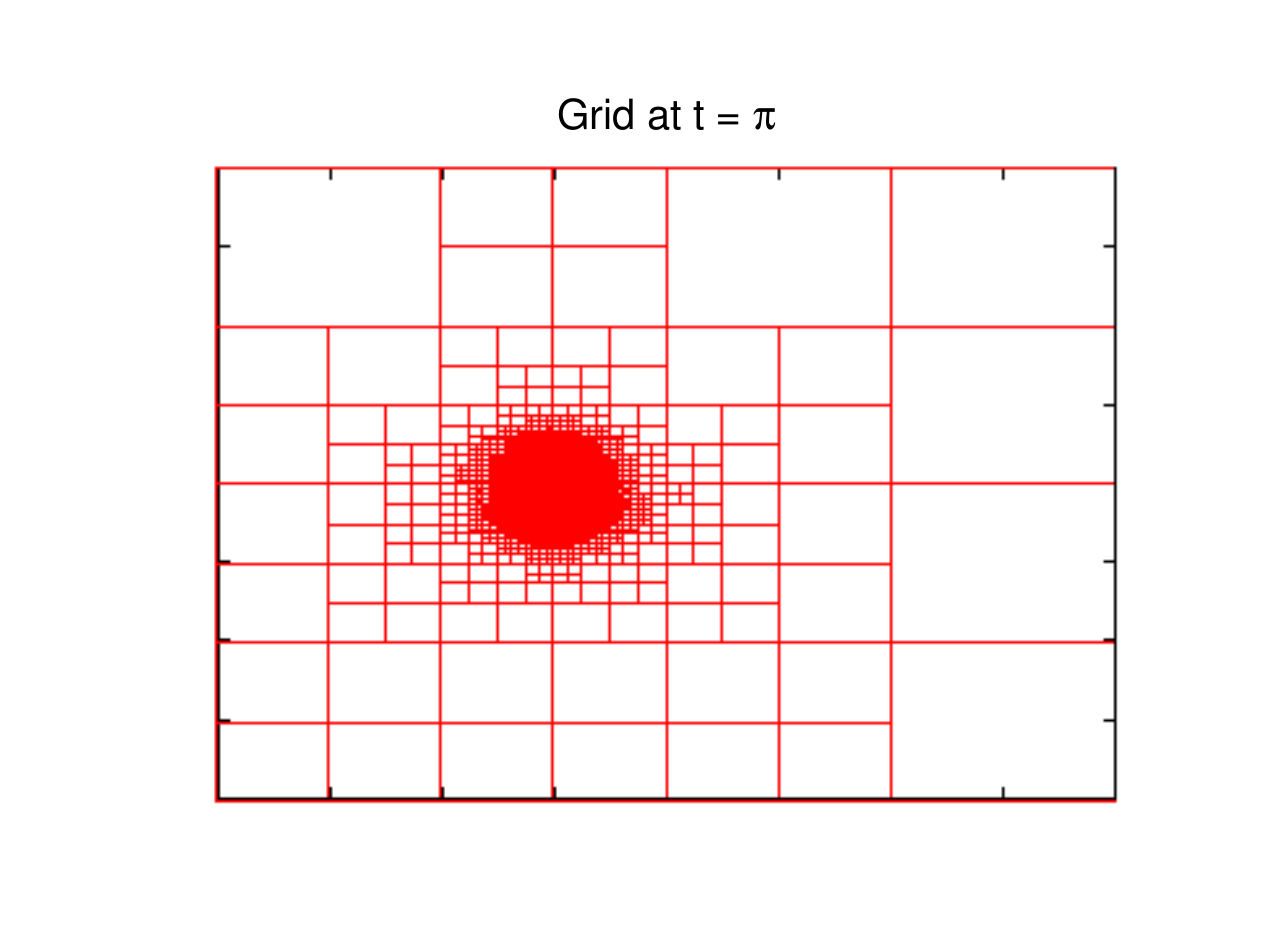} \includegraphics[scale=0.56]{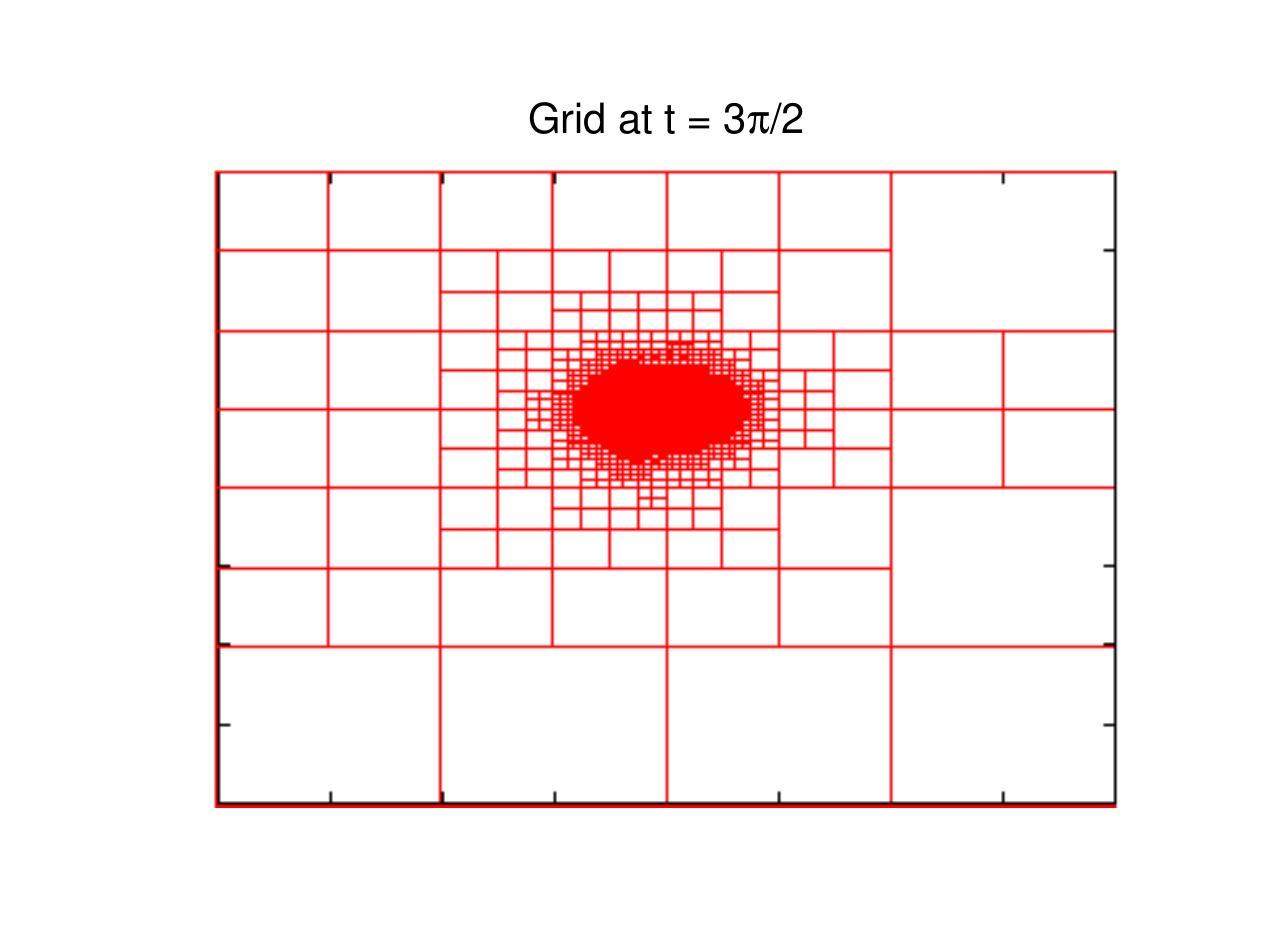}

\caption{Example 3: Grid snapshots.}
\label{ex3grids}
\end{figure}

\end{subsection}
\begin{subsection}{Example 4}

\begin{figure}
\centering
\includegraphics[scale=0.48]{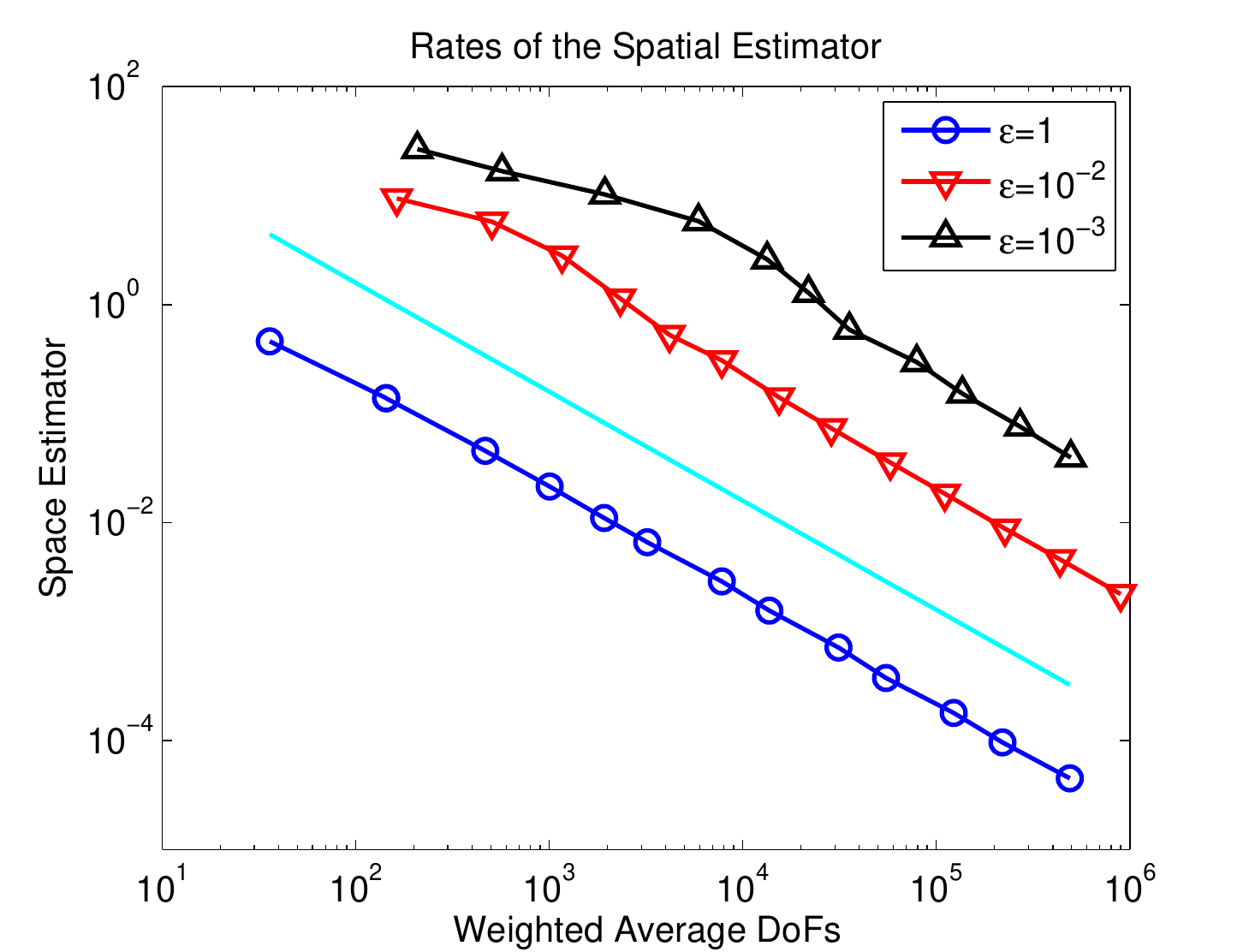} \includegraphics[scale=0.48]{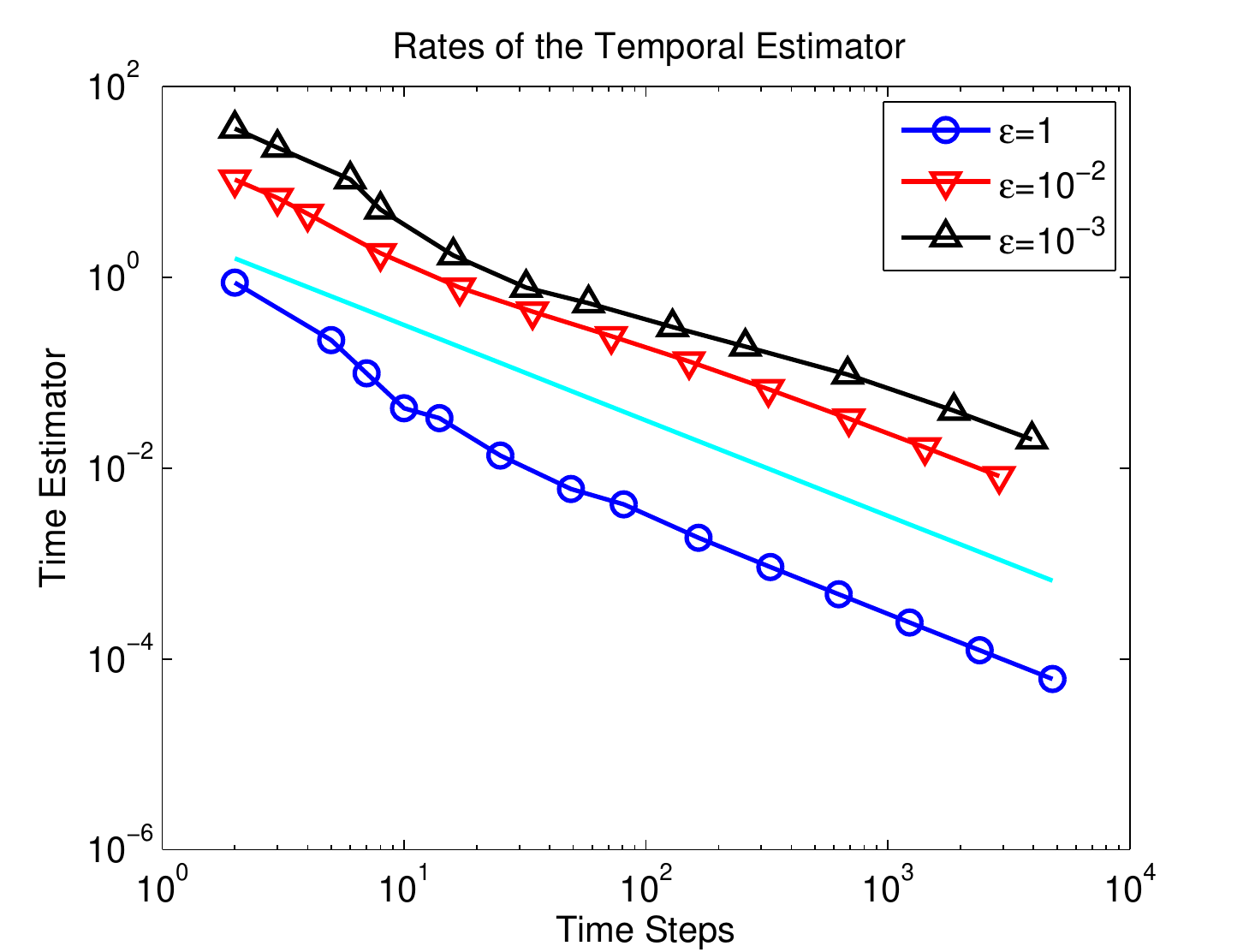}
\caption{Example 4: Spatial and temporal rates.}
\label{ex4rates}
\end{figure}
\begin{figure}
\centering
\includegraphics[scale=0.55]{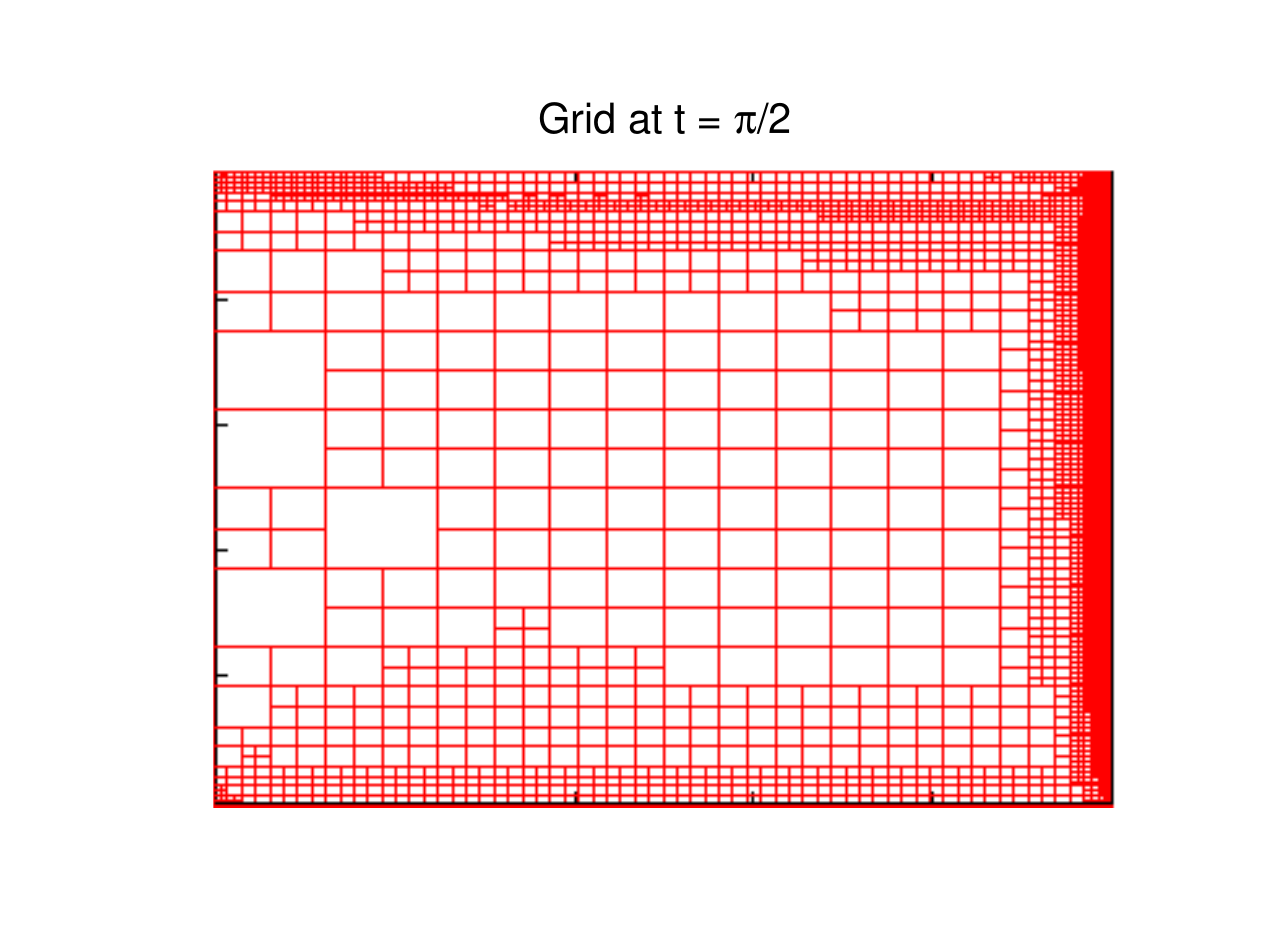} \includegraphics[scale=0.55]{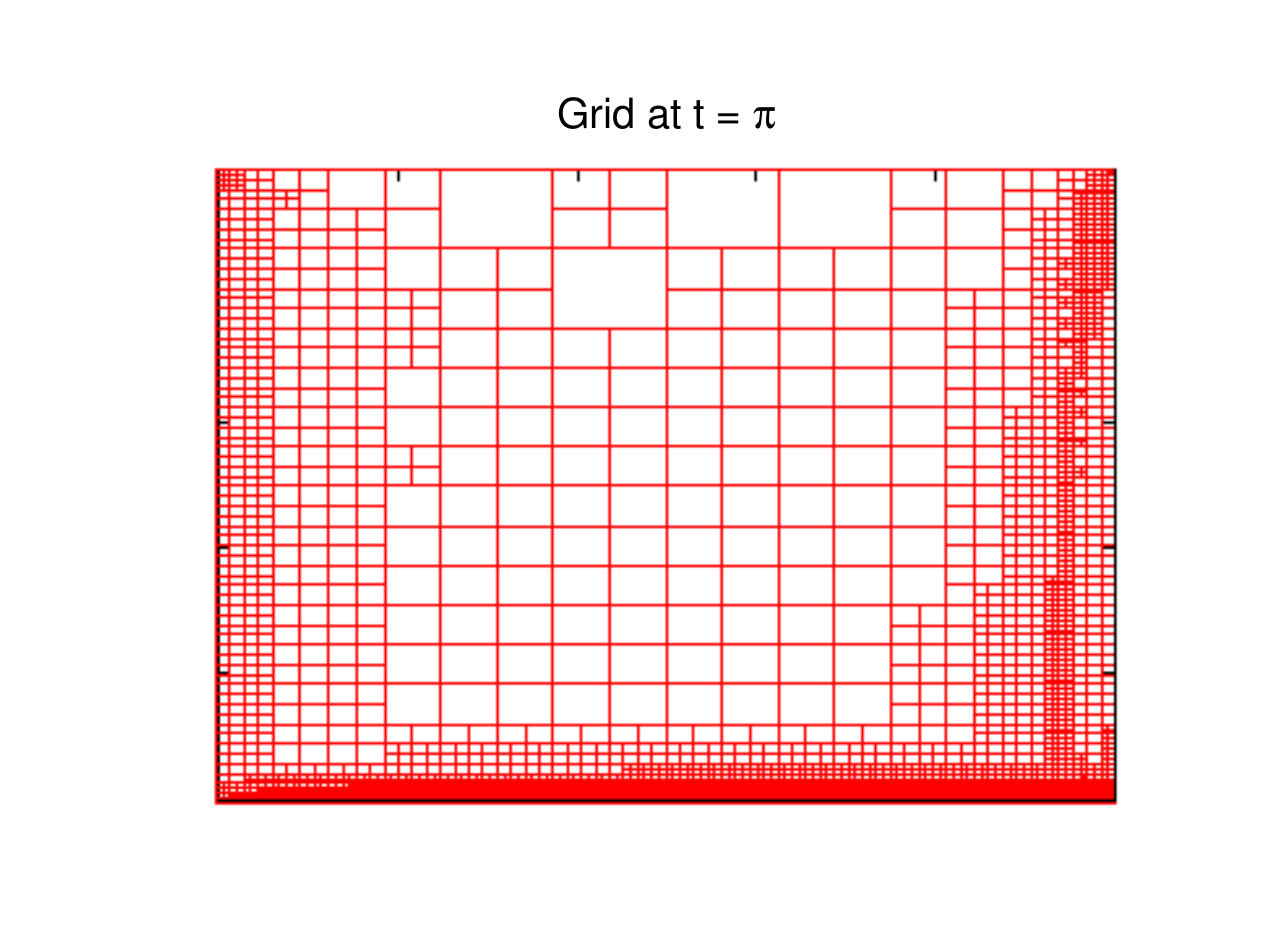}
\includegraphics[scale=0.55]{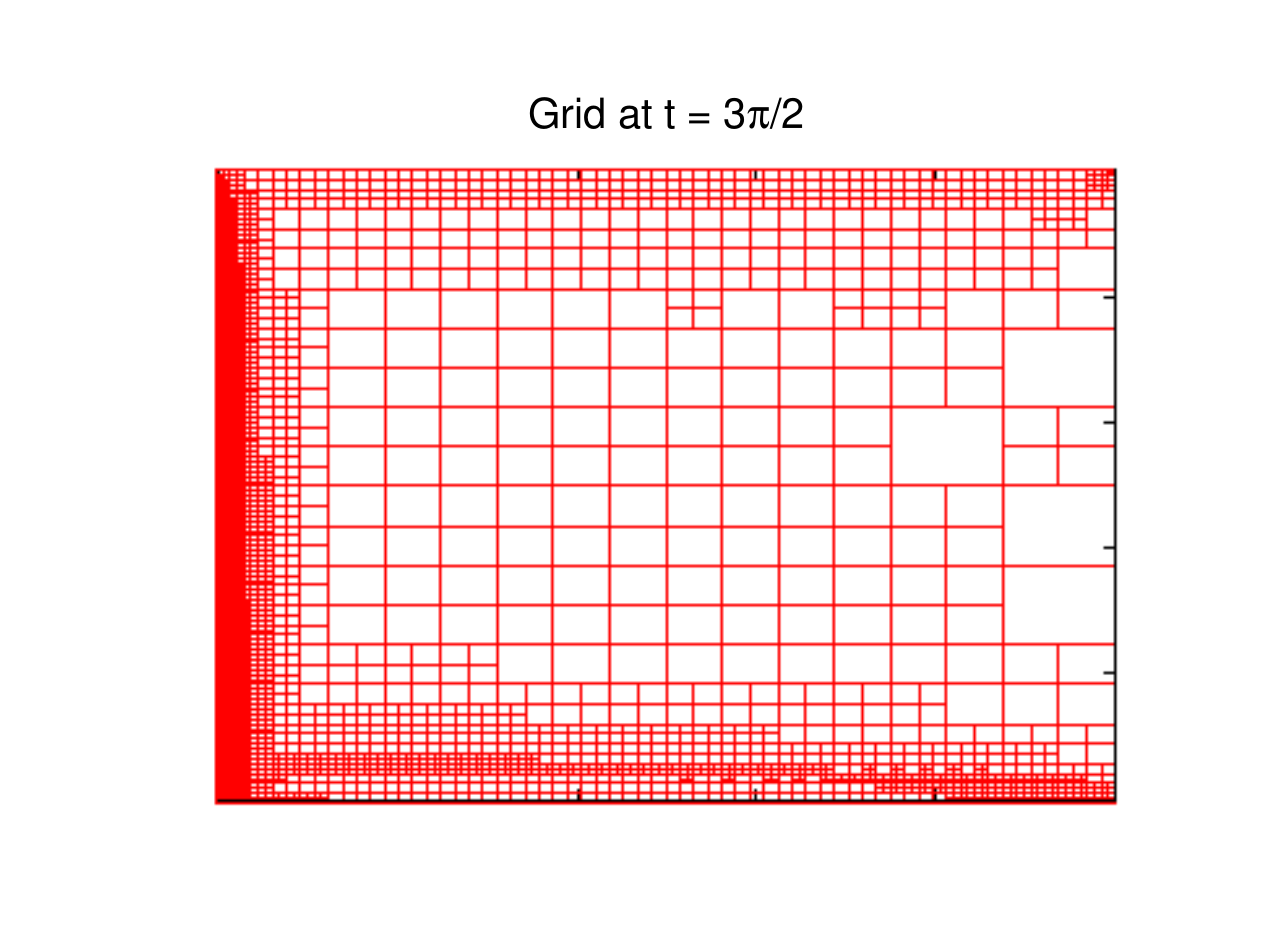} \includegraphics[scale=0.55]{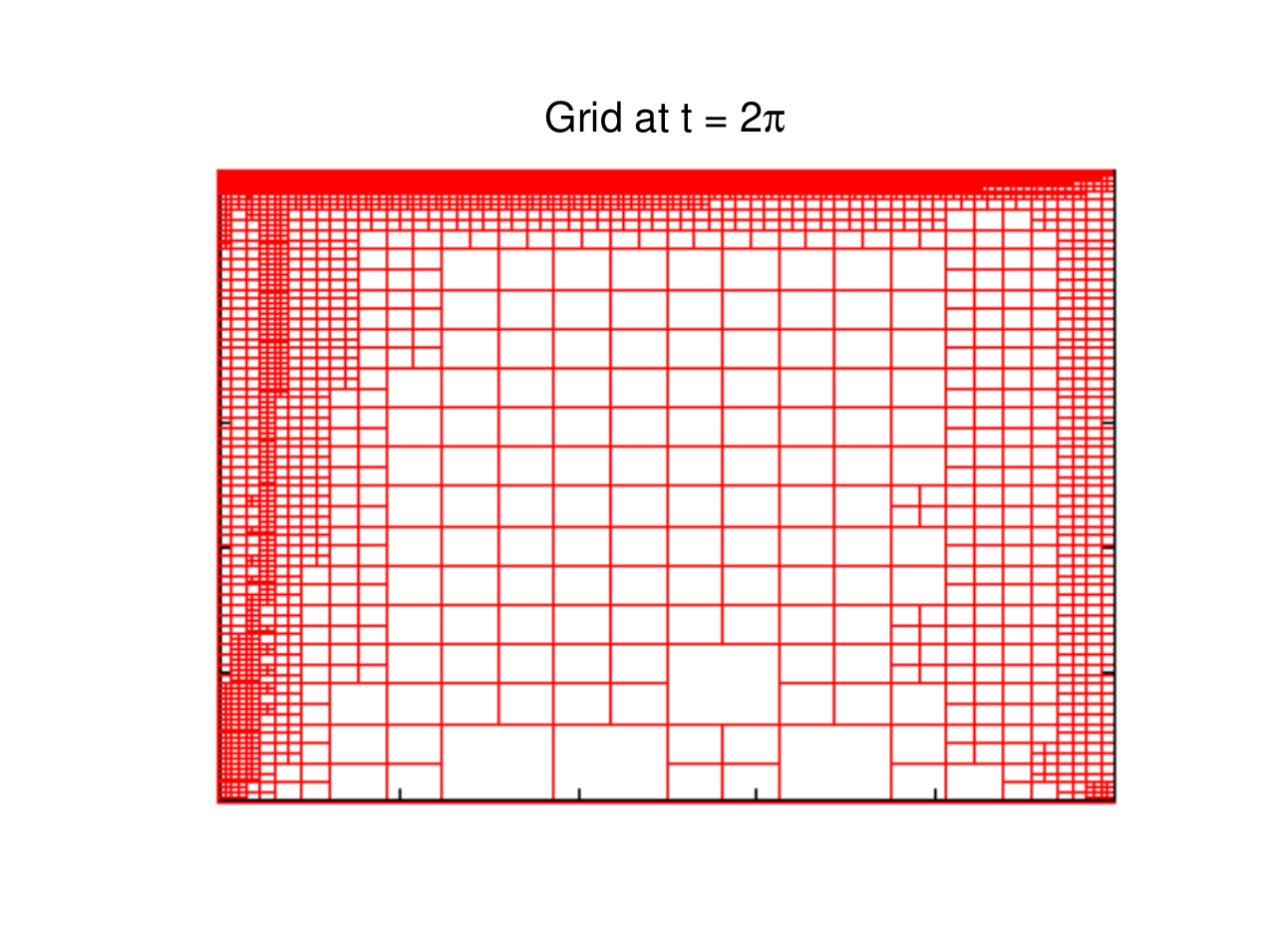}
\caption{Example 4: Grid snapshots.}
\label{ex4grids}
\end{figure}
Let $\Omega=(0,1)^2$, ${\bf a}=(\sin(t),\cos(t))^T$, $b=0$, $f=1$, $ u_0=0$ and $T = 2\pi$. The nature of the solution is rather uniform in time but spatially the solution possesses a moving boundary layer of width $\mathcal{O}(\varepsilon)$ that is driven by the changing nature of the inflow and outflow boundaries. Therefore, this example is well suited to testing the ability of the algorithm to adapt the grid to this moving boundary layer. Grids at various times are shown in  Figure \ref{ex4grids} for $\varepsilon=10^{-2}$.

 As in previous examples, we fix a small temporal threshold and then reduce the spatial threshold to observe the rates of the space estimator. Again, we also fix a spatial threshold small enough to ensure that all boundary layers are sufficiently resolved and then reduce the temporal threshold to observe the rates of the time estimator. These results are given in Figure \ref{ex4rates}. 

Optimal spatial and temporal rates of convergence are observed and the grids produced for $\varepsilon=10^{-2}$ clearly show that the adaptive algorithm is picking up the boundary layers as they move around the domain and that unneeded degrees of freedom are not retained.

\end{subsection}

\end{section}
\begin{section}{Conclusions}
An a posteriori error estimator for the discontinuous Galerkin spatial discretisation of a non-stationary linear convection-diffusion equation was derived. The numerical examples presented clearly indicate that the error estimator is practical and the respective space-time adaptive algorithm works well for the studied problems. As predicted, the spatial effecitivity indices are in an identical range to those observed in \cite{SZ09} and the spatial part of the error estimator appears to be asymptotically robust with respect to $\varepsilon$. Furthermore, the temporal effectivity indices of the studied problems are substantially smaller than those seen in \cite{GLV11} for the heat problem and may even be fully robust with respect to the norm $||\cdot||_*$.
\end{section}

\chapter{A posteriori error estimation \\and blow-up detection for nonlinear ODEs}

\begin{section}{Blow-up in nonlinear ODEs}
This section is devoted to the numerical approximation of the ODE
\begin{equation}
\begin{aligned}
\label{ODE}
\frac{du}{dt} & = f(u), \\
u(0) & = u_0,
\end{aligned}
\end{equation}
where $f$ is a Lipschitz continuous function and $u_0 > 0$. We say that \eqref{ODE} exhibits \emph{blow-up} if the solution $u$ has the property that
\begin{equation}
\begin{aligned}
\notag
\limsup_{t \to T^*} |u(t)| & = \infty,
\end{aligned}
\end{equation}
for some $T^* > 0$. The value $T^*$ is referred to as the \emph{blow-up time} and if $T^* < \infty$ we say the ODE exhibits \emph{finite time blow-up}. Throughout the rest of this chapter, it is assumed that \eqref{ODE} exhibits finite time blow-up. Regarding the numerical approximation of \eqref{ODE}, the following questions are of interest:

\begin{itemize}
\item Using a simple time stepping scheme to approximate \eqref{ODE}, can we construct a residual based a posteriori error estimator for the given method?

\item If the answer to the above is yes, is it possible to use the estimator to drive an adaptive algorithm that successfully converges to $T^*$?

\item If the adaptive algorithm does converge to $T^*$, is it possible to numerically quantify how quickly different time stepping schemes converge to $T^*$?
\end{itemize}

In attempting to answer the above questions, there are some papers in the literature of particular interest. In \cite{SF90}, the authors approximate \eqref{ODE} using a\mbox{ }\mbox{ }\mbox{ } $\theta$-method along with a temporal rescaling of the ODE; this modifies a distribution of uniform time steps so that they better match the blow-up behaviour of the numerical solution. Their numerical solution displays the same asymptotic \mbox{ }  behaviour as the exact solution. 

Similarly to \cite{SF90}, in \cite{H06} the authors also transform the ODE but through an arc length transformation. They use a forward Euler method to approximate the transformed equation and they show that their adaptive algorithm, which is based on their transformation plus a tolerance controlled ODE integrator, converges towards the blow-up time linearly with respect to the total number of time steps. However, in contrast to \cite{SF90}, they restrict themselves to the case $f(u) = u^{p}$, $p > 1$. 

Finally, in \cite{JW14} they prove existence results for numerical approximations to \eqref{ODE} provided that the time step lengths are sufficiently small and that the nonlinearity satisfies a polynomial growth condition. For the specific case \mbox{ } $f(u) = u^{p}$, $p > 1$, they show that a certain selection of time step lengths yields approach to the blow-up time.

\end{section}

\begin{section}{An a posteriori error estimator}

For simplicity, we restrict our attention to 
\begin{equation}
\begin{aligned}
\label{nonlinearform}
f(u) = \sum_{j=0}^{p} c_j u^j,
\end{aligned}
\end{equation}
where $p \geq 2$ is some positive integer and the coefficients satisfy $c_j \geq 0$ ($c_p > 0$) so that the problem is guarenteed to blow-up. In order to approximate \eqref{ODE}, we shall use a generic one-step scheme with right-hand side $f_h$ that approximates $f$. That is, we set $u_h^0 = u_0$ and for $k\geq 0$ with some time step length $\tau_{k+1}$,  we search for $u_h^{k+1}$ such that
\begin{equation}
\begin{aligned}
\label{ODEapprox}
\frac{u_h^{k+1} - u_h^k}{\tau_{k+1}} = f_h\big(u_h^k,u_h^{k+1}\big).
\end{aligned}
\end{equation}
We also recursively define our time $t^{k+1} := t^k + \tau_{k+1}$ with $t^0 := 0$. In order to discuss the error of different time stepping schemes, we need to describe $u_h$ on the interior of the intervals $\big[t^k,t^{k+1}\big]$. Thus, given $t \in \big(t^k,t^{k+1}\big]$, we define $u_h(t)$ to be the linear interpolant with respect to $t$ of the values $u_h^k$ and $u_h^{k+1}$, viz.,
\begin{equation}
\notag
u_h(t):=l_k(t)u_h^k+l_{k+1}(t)u_h^{k+1}.
\end{equation}
It is now possible to construct an error equation for \eqref{ODEapprox} by subtracting \eqref{ODEapprox} from \eqref{ODE}. Then, defining $e := u - u_h$, we obtain
\begin{equation}
\begin{aligned}
\label{ODEerror1}
\frac{de}{dt} = f\big(u\big)- f_h\big(u_h^k,u_h^{k+1}\big).
\end{aligned}
\end{equation}
Adding and subtracting $f(u_h)$ to the right of \eqref{ODEerror1} and defining the residual \mbox{ } $\eta_{k+1} := f \big(u_h \big) - f_h\big(u_h^k,u_h^{k+1}\big)$ we obtain the error equation
\begin{equation}
\begin{aligned}
\label{ODEerror2}
\frac{de}{dt} = \eta_{k+1} + f'(u_h)e + \sum_{j=2}^p \frac{f^{(j)}(u_h)}{j!}e^j,
\end{aligned}
\end{equation}
where $f^{(j)}$ denotes the order $j$ partial derivative of $f$ with respect to $u$. In order to derive a usable error estimator from \eqref{ODEerror2}, we make use of Gronwall's inequality. Application of Gronwall's inequality to \eqref{ODEerror2}  for $t \in \big[t^k,t^{k+1}\big]$ yields
\begin{equation}
\begin{aligned}
\label{ODEbound1}
|e(t)| \leq H_{k+1}(t)G_{k+1}\phi_{k+1},
\end{aligned}
\end{equation}
where
\begin{equation}
\begin{aligned}
\notag
H_{k+1}(t) & := \exp \left (\sum_{j=2}^p \int_{t^k}^{t} \! \bigg\vert\frac{f^{(j)}(u_h)}{j!} \bigg \vert |e|^{j-1} \, ds \right ), \\
G_{k+1} & := \exp \left (\int_{t^k}^{t^{k+1}} \! |f'(u_h)| \, ds \right ), \\
\phi_{k+1} & :=  \big|e \big(t^k \big) \big| + \int_{t^k}^{t^{k+1}} \! |\eta_{k+1}| \, ds.
\end{aligned}
\end{equation}
Note that this is not truly a posteriori yet due to the presence of $H_{k+1}$.  In order to amend this, we use a local continuation argument in the spirit of \cite{B05,GM14,KNS04}. To this end, define the set
\begin{equation}
\begin{aligned}
\notag
I_{k+1} := \bigg \{t \in \big[t^k,t^{k+1}\big] \mbox{ }\bigg | \mbox{ } \max_{s \in [t^k,t]} |e(s)| \leq \delta_{k+1}G_{k+1}\phi_{k+1}\bigg \},
\end{aligned}
\end{equation}
where $\delta_{k+1} > 1$ is a parameter to be chosen. Obviously $t^k \in I_{k+1}$ so $I_{k+1}$ is non-empty and bounded. We denote the maximal value of $t$ that belongs to $I_{k+1}$ by $t^*$ and we assume that $t^* < t^{k+1}$. From \eqref{ODEbound1}, we have that
\begin{equation}
\begin{aligned}
\max_{s \in [t^k,t^*]} |e(s)| \leq H_{k+1}(t^*)G_{k+1}\phi_{k+1}.
\end{aligned}
\end{equation}
By the definition of the set $I_{k+1}$, we have
\begin{equation}
\begin{aligned}
\label{Hbound}
H_{k+1}(t^*) & \leq \exp \left (\sum_{j=2}^p \left (\max_{s \in [t^k,t^*]} |e(s)| \right)^{j-1} \int_{t^k}^{t^{k+1}} \! \bigg \vert \frac{f^{(j)}(u_h)}{j!} \bigg \vert \, ds \right ) 
\\ & \leq \exp\left ( \sum_{j=2}^p \delta_{k+1}^{j-1}G_{k+1}^{j-1}\phi_{k+1}^{j-1}\int_{t^k}^{t^{k+1}} \! \bigg \vert \frac{f^{(j)}(u_h)}{j!}\bigg \vert \, ds \right ).
\end{aligned}
\end{equation}
Therefore,
\begin{equation}
\begin{aligned}
\label{ODEbound2}
\max_{s \in [t^k,t^*]} |e(s)| \leq G_{k+1}\phi_{k+1}\exp\left ( \sum_{j=2}^p \delta_{k+1}^{j-1}G_{k+1}^{j-1}\phi_{k+1}^{j-1}\int_{t^k}^{t^{k+1}} \! \bigg \vert \frac{f^{(j)}(u_h)}{j!}\bigg \vert \, ds \right ).
\end{aligned}
\end{equation}
Now, suppose that the upper bound in \eqref{ODEbound2} is bounded strictly from above by the upper bound of the set $I_{k+1}$, viz.,
\begin{equation}
\begin{aligned}
G_{k+1}\phi_{k+1}\exp\left ( \sum_{j=2}^p \delta_{k+1}^{j-1}G_{k+1}^{j-1}\phi_{k+1}^{j-1}\int_{t^k}^{t^{k+1}} \! \bigg \vert \frac{f^{(j)}(u_h)}{j!}\bigg \vert \, ds \right ) < \delta_{k+1}G_{k+1}\phi_{k+1},
\end{aligned}
\end{equation}
or equivalently,
\begin{equation}
\begin{aligned}
\label{deltainequality}
\exp\left ( \sum_{j=2}^p \delta_{k+1}^{j-1}G_{k+1}^{j-1}\phi_{k+1}^{j-1}\int_{t^k}^{t^{k+1}} \! \bigg \vert \frac{f^{(j)}(u_h)}{j!}\bigg \vert \, ds \right ) < \delta_{k+1},
\end{aligned}
\end{equation}
then $t^*$ cannot be the maximal value of $t$ that belongs to $I_{k+1}$ because we just showed $\displaystyle \max_{s \in [t^k,t^*]} |e(s)|$ satisfies a bound strictly less than that assumed in the set $I_{k+1}$ {\bf--} a contradiction. Therefore, provided \eqref{deltainequality} is satisfied, $I_{k+1} = \big[t^k,t^{k+1}\big]$ and we have our desired error bound once we select $\delta_{k+1}$. Given that we wish to construct the best bound possible, we seek to minimise \eqref{deltainequality}. Taking the limit we can, in fact, just select $\delta_{k+1}$ to be the minimiser of
\begin{equation}
\begin{aligned}
\label{deltaequation}
\sum_{j=2}^p \delta_{k+1}^{j-1}G_{k+1}^{j-1}\phi_{k+1}^{j-1}\int_{t^k}^{t^{k+1}} \! \bigg \vert \frac{f^{(j)}(u_h)}{j!} \bigg \vert \, ds - \log(\delta_{k+1}) = 0, \qquad \delta_{k+1} > 1.
\end{aligned}
\end{equation}
Therefore, providing the solution to \eqref{deltaequation} exists, we have the following error bound
\begin{equation}
\begin{aligned}
\label{ODEbound3}
\big |e\big(t^{k+1}\big)\big| \leq \max_{t \in [t^k,t^{k+1}]} |e(t)| \leq \delta_{k+1}G_{k+1}\phi_{k+1}.
\end{aligned}
\end{equation}

\begin{remark}
The term $\phi_{k+1}$ can be redefined with $\big|e\big(t^k\big)\big|$ estimated using the error estimator from the previous time step without any loss of generality to the argument presented giving us a recursive procedure for estimating the error. 
\end{remark}

\begin{remark}
In practice, the solution to \eqref{deltaequation} is approximated using a Newton method.
\end{remark}

A natural question that arises is whether or not \eqref{deltaequation} can be satisfied practically close to the blow-up time. With the aid of the next lemma, we state a precise condition on the time step lengths $\tau_{k+1}$ which indeed ensures that \eqref{deltaequation} has a root $\delta_{k+1}>1.$
\begin{lemma}
If $\displaystyle \sum_{j=1}^pj C_je^j \le 1$ then $\displaystyle s(x) = \sum_{j=1}^p C_j x^j-\log(x)$ with $C_j> 0$, $j=1$, ..., $p$, $p\in\mathbb{N}$ has a root in $(1,+\infty).$
\end{lemma}
\begin{proof}
See Lemma 2.2 in \cite{CGKM15}.
\end{proof}

The above lemma gives a sufficient condition on when \eqref{deltaequation} can be satisfied. In particular, condition \eqref{deltaequation} can always be made to be satisfied provided that the time step length $\tau_{k+1}$ is chosen such that 
$$\sum_{j=2}^p\frac{j-1}{j!}(G_{k+1}\phi_{k+1} e)^{j-1}\int_{t^{k}}^{t^{k+1}} \! {\left|f^{(j)}(u_h) \right|} \, ds \le 1.$$

It is useful to discuss in heuristic terms what $\delta_{k+1}$, $G_{k+1}$ and $\phi_{k+1}$ in \eqref{ODEbound3} represent. Obviously $\phi_{k+1}$ is an approximation to the error on each time interval $\big[t^k,t^{k+1} \big]$, but what about $G_{k+1}$ and $\delta_{k+1}$? Clearly both $G_{k+1}$ and $\delta_{k+1}$ are accumulation factors that represent the contribution of blow-up to the error estimator in some way. To gain some insight on these multiplicative terms, consider $f(u)=u^p$, $p > 1$. Through separation of variables, the solution to \eqref{ODE} is given by 
\begin{equation}
\begin{aligned}
\notag
u(t) = \big(u_0^{1-p} + (1-p)t \big)^{\frac{1}{1-p}}.
\end{aligned}
\end{equation}
Now, suppose that $u_h \approx u$ and consider $G_{k+1}(u)$ given by 
\begin{equation}
\begin{aligned}
\notag
G_{k+1}(u) = \exp \left (\int_{t^k}^{t^{k+1}} \! f'(u) \, ds \right ) =  \exp \left (\int_{t^k}^{t^{k+1}} \! pu^{p-1} \, ds \right ).  
\end{aligned}
\end{equation}
Substitution of the exact solution yields
\begin{equation}
\begin{aligned}
\notag
G_{k+1}(u) & = \exp \left (\int_{t^k}^{t^{k+1}} \! p \big (u_0^{1-p}+(1-p)t \big )^{-1} \, ds \right ) \\
& = \exp \Bigg(\frac{p}{1-p}\log \Bigg (\frac{u_0^{1-p}+(1-p)t^{k+1}}{u_0^{1-p}+(1-p)t^k}  \Bigg ) \Bigg ) \\
& = \Bigg (\frac{u_0^{1-p}+(1-p)t^{k+1}}{u_0^{1-p}+(1-p)t^k}  \Bigg )^{\frac{p}{1-p}} \\
& = \frac{u^p \big(t^{k+1} \big)}{u^p \big(t^k \big)}.
\end{aligned}
\end{equation}
So for $f(u)=u^p$, $G_{k+1}(u)$ measures the blow-up rate of the exact solution on the interval $\big[t^k,t^{k+1} \big]$ and thus $\delta_{k+1}G_{k+1}$ can be viewed as the blow-up rate of the numerical solution. As we performed a Taylor expansion in our error analysis, we infer that $G_{k+1}$ is the linearised numerical blow-up rate on the interval $\big[t^k,t^{k+1} \big]$ and $\delta_{k+1}$ is the higher order part of the numerical blow-up rate on the interval $\big[t^k,t^{k+1} \big]$. With these notions, another way of viewing \eqref{deltainequality} is that the numerical solution ceases to be valid once the (approximate) higher order terms from the Taylor expansion start to become dominant in the error estimator.

\end{section}

\begin{section}{Adaptivity and convergence towards the blow-up time}

Using our knowledge of algorithms utilising a posteriori error estimators for linear problems, we propose Algorithm 4.1 for advancing towards the blow-up time. The basic idea behind the algorithm is to half the time step length and recompute the solution until the residual is below a given input threshold ${\tt tol}$. The algorithm then advances by using the previous (now fixed) time step length as a reference to compute the next approximation. The algorithm continues in this way until \eqref{deltaequation} no longer has a solution; the algorithm then terminates and outputs the total number of time steps $N$ and the final time $T$.

\begin{algorithm} \label{ODEalgorithm1}
  \begin{algorithmic}[1]
     \State {\bf Input:} $f$, $f_h$, $u_0$, $\tau_1$, ${\tt tol}$.
     \State Calculate $u_h^1$ from $u^0_h$.
     \While {$\displaystyle \int_{t^0}^{t^{1}} \! |\eta_{1}| \, ds > {\tt tol}$}
     \State $\tau_1 \leftarrow \tau_1/2$.
     \State Calculate $u_h^1$ from $u^0_h$.
\EndWhile
\State Calculate $\delta_{1}$.
\State Set $k = 0$.
     \While {$\delta_{k+1}$ exists}
     \State $k \leftarrow k+1$.
     \State $\tau_{k+1} = \tau_k$.
     \State Calculate $u_h^{k+1}$ from $u_h^k$.
     \While {$\displaystyle \int_{t^k}^{t^{k+1}} \! |\eta_{k+1}| \, ds > {\tt tol}$}
     \State $\tau_{k+1} \leftarrow \tau_{k+1}/2$.
     \State Calculate $u_h^{k+1}$ from $u_h^k$.
\EndWhile
\State Calculate $\delta_{k+1}$.
\EndWhile
\State {\bf Output:} $k$, $t^{k}$.
  \end{algorithmic}
  \caption{ODE Algorithm 1}
\end{algorithm}

Assuming that the adaptive algorithm outputs successfully, we wish to observe the order with which the adaptive algorithm approaches the blow-up time. To this end, we define the $\lambda$ function
\begin{equation}
\begin{aligned}
\notag
\lambda({\tt tol},N) := \left|T^* - T({\tt tol},N) \right|,
\end{aligned}
\end{equation}
where $T^*$ is the blow-up time of problem \eqref{ODE}.
It is conjectured that
\begin{equation}
\begin{aligned}
\notag
\lambda({\tt tol},N) \propto N^{-r},
\end{aligned}
\end{equation}
where $r$ is the order with which the adaptive algorithm approaches the blow-up time. An educated guess would be that $r$ is the same as the order of the method that we choose to use.

 In order to gain some insight on how $\lambda$ converges, we apply Algorithm 4.1 to problem \eqref{ODE} with $f(u)=u^p$ for $p=2,3$ and $u(0)=1$ under the following time stepping schemes

\begin{equation}
\begin{aligned}
\notag
\mbox{Explicit Euler} \qquad f_h \big(u_h^k,u_h^{k+1}\big)& =f\big(u_h^k\big), \\ 
\mbox{Implicit Euler} \qquad f_h\big(u_h^k,u_h^{k+1}\big)& =f\big(u_h^{k+1}\big), \\ 
\mbox{Improved Euler} \qquad f_h\big(u_h^k,u_h^{k+1}\big)& =\frac{1}{2}\big(f\big(u_h^k\big)+f\big(u_h^k+\tau_{k+1}f\big(u_h^k \big)\big)\big).
\end{aligned}
\end{equation}
The approximate rates of $\lambda$ under Algorithm 4.1 are given in Table \ref{data1}.

\begin{table}[ht]
\caption{ODE Algorithm 1 Results}
\centering
\begin{tabular}{c c c c}
\hline\hline
Method & p = 2 & p = 3 \\
\hline 
Implicit Euler & $r \approx 0.66$ & $r \approx 0.79$ \\
Explicit Euler & $r \approx 1.35$ & $r \approx 1.60$ \\
Improved Euler & $r \approx 1.2$ & $r \approx 1.48$ \\
\hline
\end{tabular}
\label{data1}
\end{table}

The first question that arises is why is the explicit Euler method significantly better than the  implicit Euler method? The answer lies in the way in which we have derived the error estimator. The numerical solution from the explicit Euler method always underestimates the true solution $u$ \cite{SF90}; this means $\delta_{k+1}$ is correcting for the fact that $G_{k+1}$ is underestimating the true blow-up rate {\bf--} our error bound is very tight and this explains the high convergence rate of $\lambda$. For the implicit Euler method, $G_{k+1}$ overestimates the true blow-up rate \cite{SF90} meaning we obtain nothing ``extra" from the error analysis in the way that we do for the explicit Euler method.

The second question is why is improved Euler worse than explicit Euler? Indeed, one would expect a faster approach to the blow-up time with a higher order method. The reason for this lies in a fault with the proposed adaptive algorithm. Indeed, the threshold approach taken to reducing our time step length is good for linear problems. However, we have neglected the presence of $G_{k+1}$ in our error estimator; this factor tells us that each successive interval matters less to the error estimator than previous intervals meaning we need to increase ${\tt tol}$ on each interval. Thus, we propose Algorithm 4.2.

\begin{algorithm} \label{ODEalgorithm2}
  \begin{algorithmic}[1]
     \State {\bf Input:} $f$, $f_h$, $u_0$, $\tau_1$, ${\tt tol}$.
     \State Calculate $u_h^1$ from $u^0_h$.
     \While {$\displaystyle \int_{t^0}^{t^{1}} \! |\eta_{1}| \, ds > {\tt tol}$}
     \State $\tau_1 \leftarrow \tau_1/2$.
     \State Calculate $u_h^1$ from $u^0_h$.
\EndWhile
\State Calculate $\delta_{1}$.
\State ${\tt tol} = G_1*{\tt tol}.$
\State Set $k=0$.
     \While {$\delta_{k+1}$ exists}
     \State $k \leftarrow k+1$.
     \State $\tau_{k+1} = \tau_k$.
     \State Calculate $u_h^{k+1}$ from $u_h^k$.
     \While {$\displaystyle \int_{t^k}^{t^{k+1}} \! |\eta_{k+1}| \, ds > {\tt tol}$}
     \State $\tau_{k+1} \leftarrow \tau_{k+1}/2$.
     \State Calculate $u_h^{k+1}$  from $u_h^k$.
\EndWhile
\State Calculate $\delta_{k+1}$.
\State ${\tt tol} = G_{k+1}*{\tt tol}.$
\EndWhile
\State {\bf Output:} $k$, $t^{k}$.
  \end{algorithmic}
  \caption{ODE Algorithm 2}
\end{algorithm}

\begin{figure}
\centering
\includegraphics[scale=0.56]{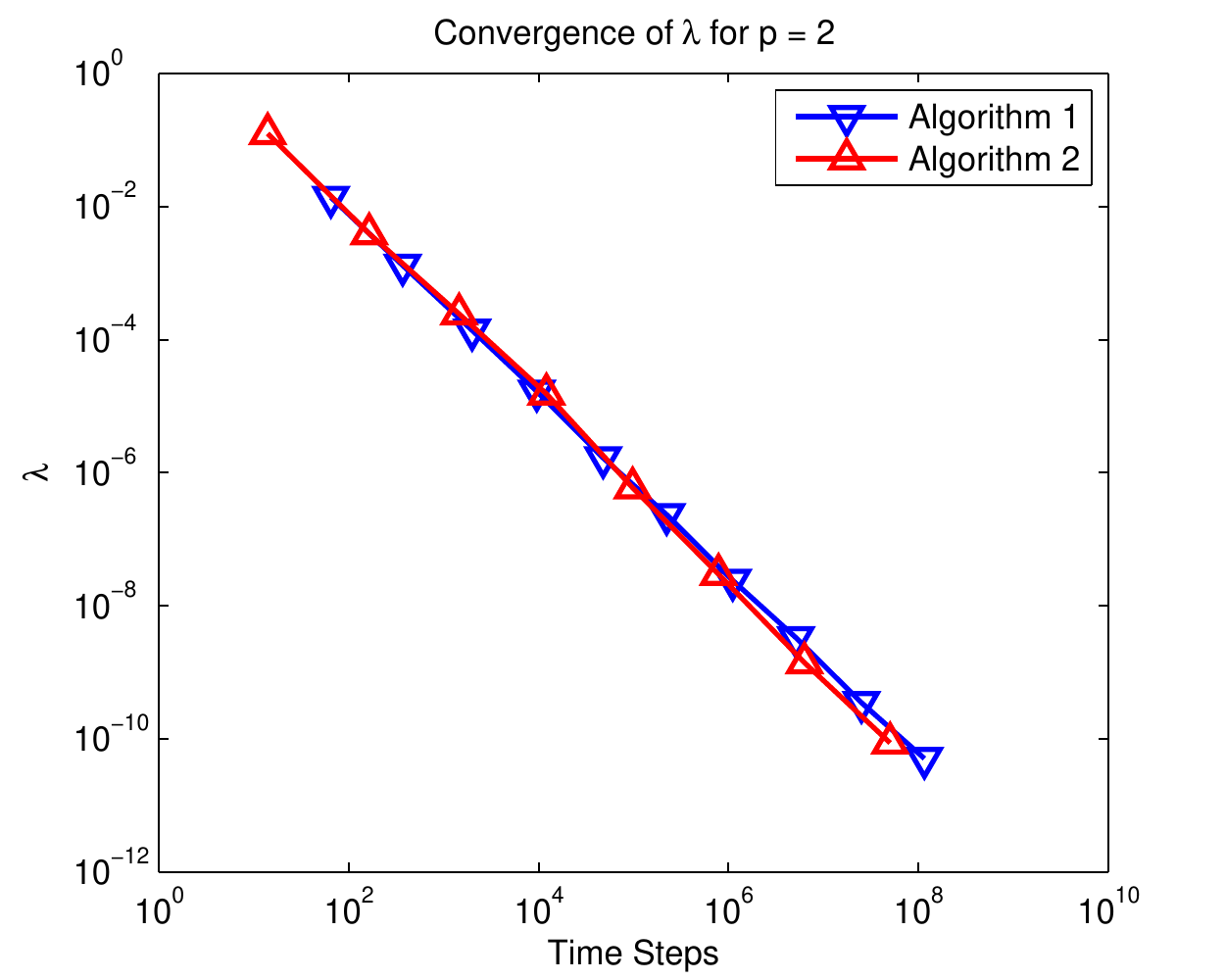} \includegraphics[scale=0.56]{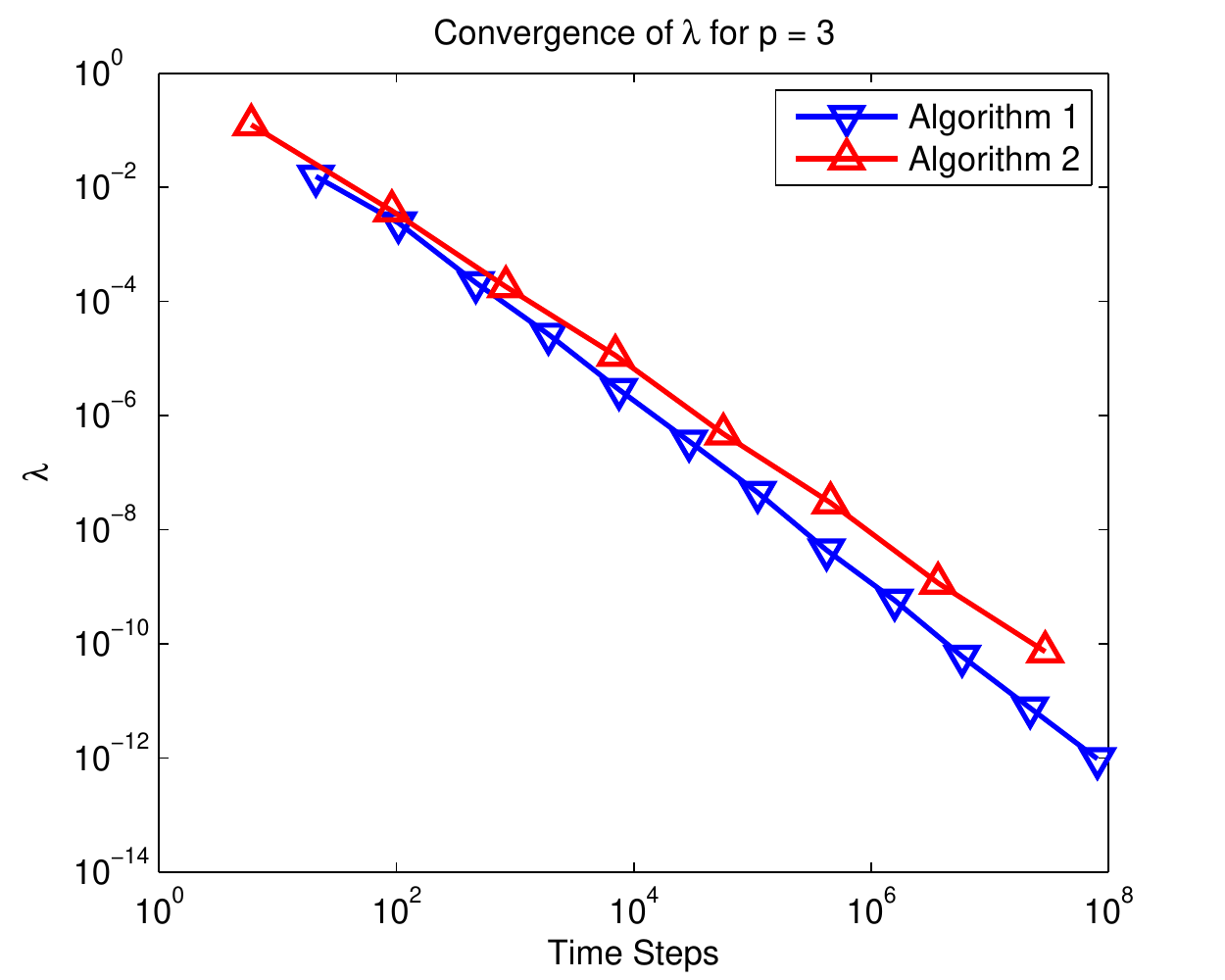}
\caption{Rates of $\lambda$ for explicit Euler under ODE Algorithms 1 and 2.}
\label{graphODEexplicit}
\end{figure}
\begin{figure}
\centering
\includegraphics[scale=0.56]{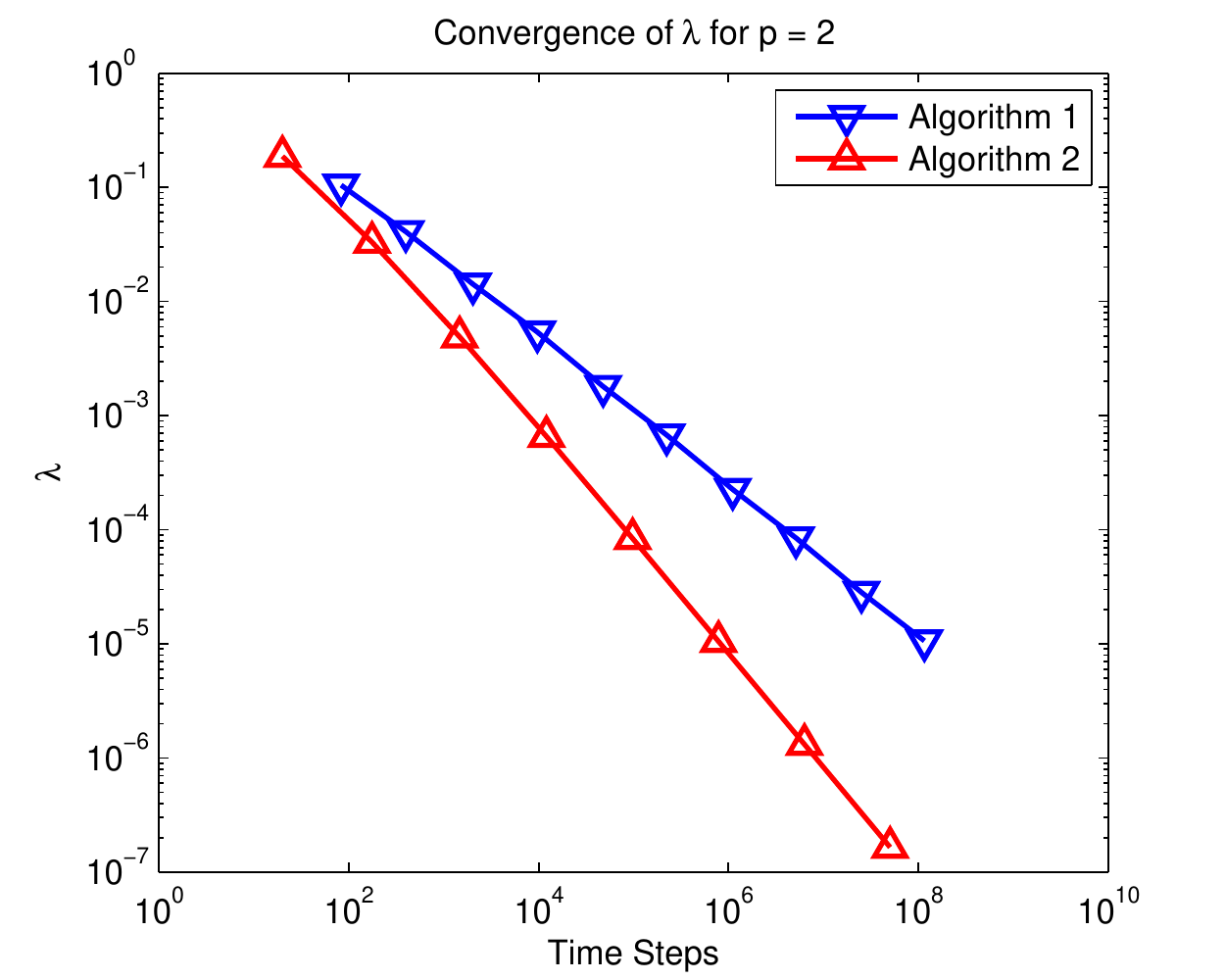} \includegraphics[scale=0.56]{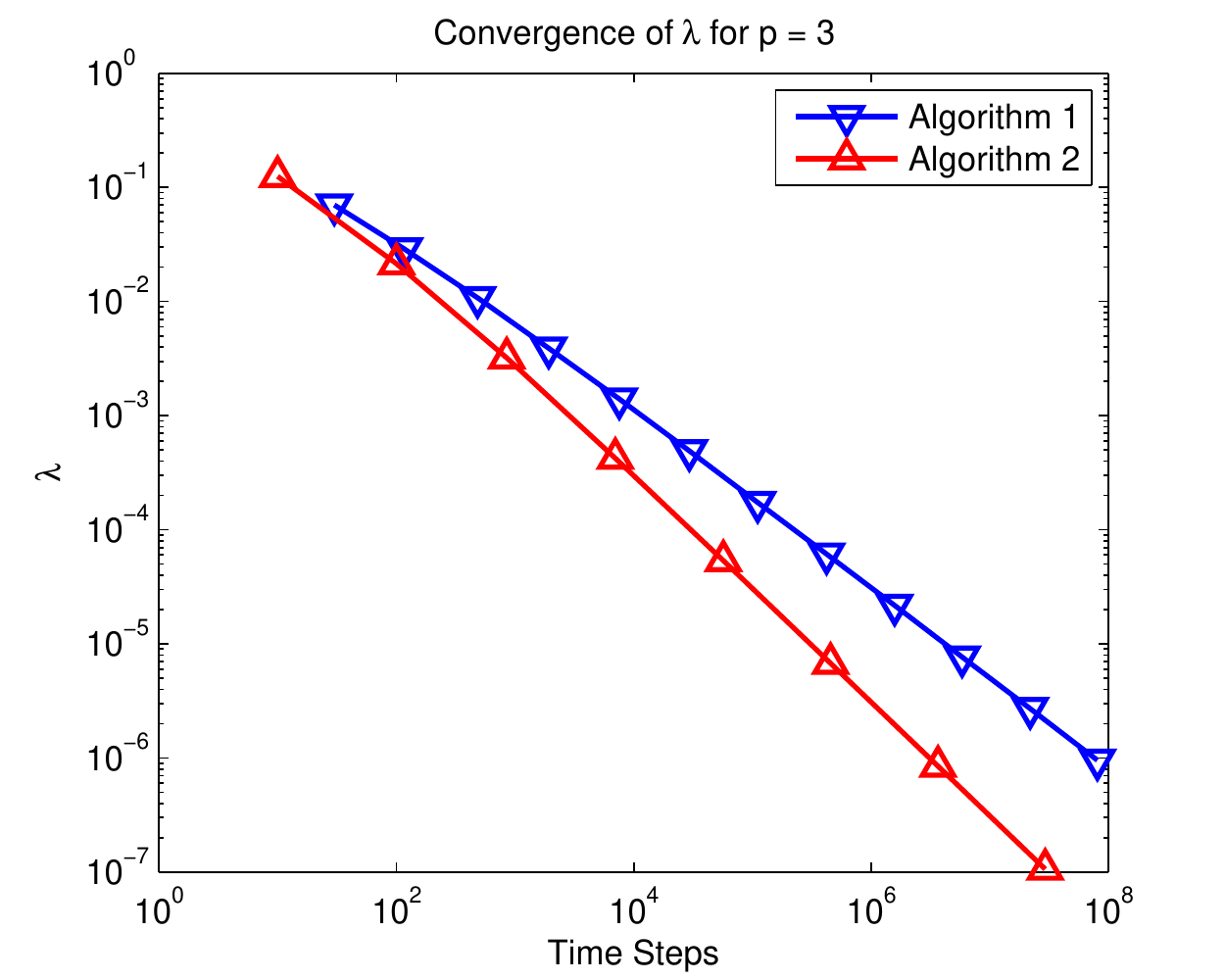}
\caption{Rates of $\lambda$ for implicit Euler under ODE Algorithms 1 and 2.}
\label{graphODEimplicit}
\end{figure}
\begin{figure}
\centering
\includegraphics[scale=0.56]{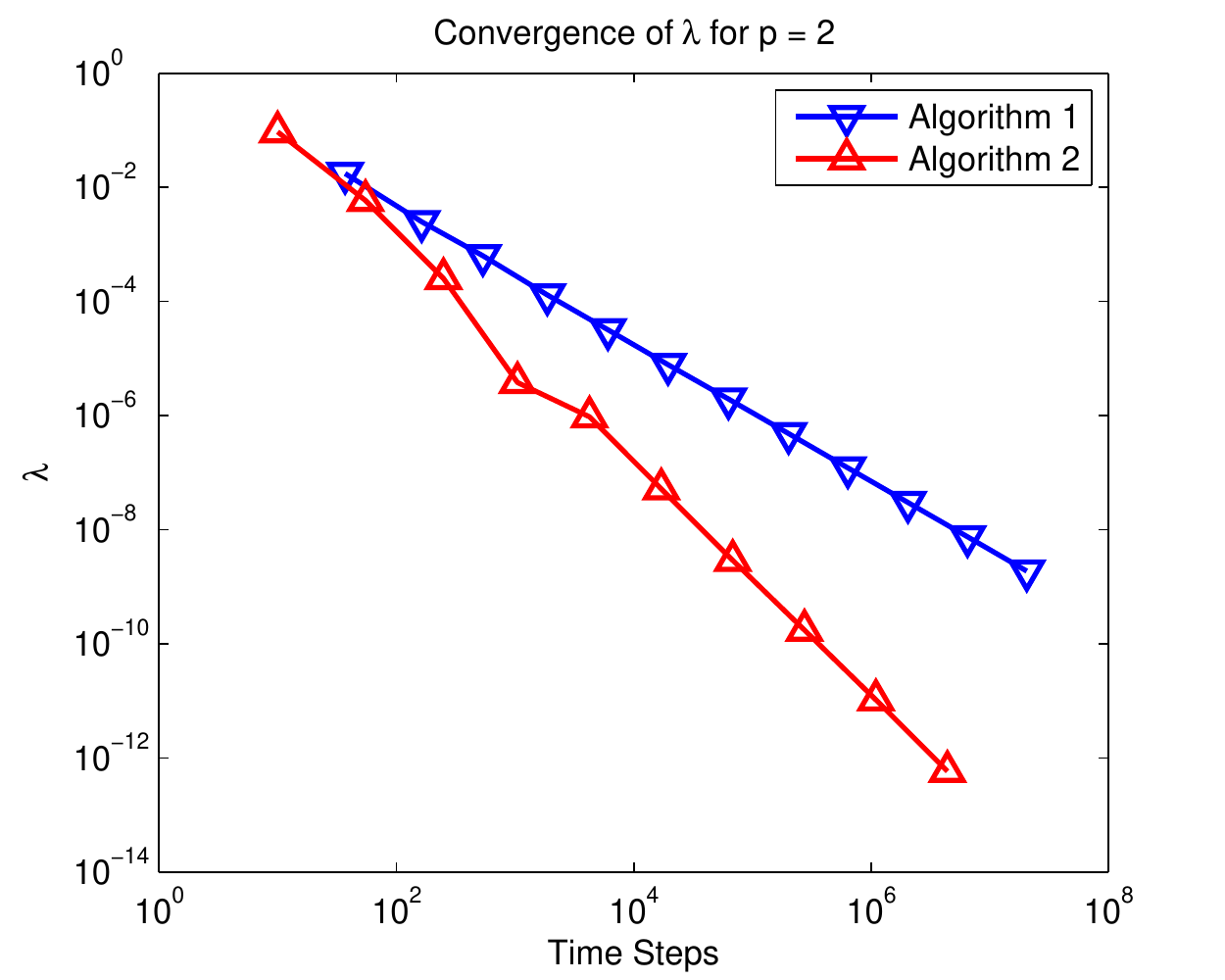} \includegraphics[scale=0.56]{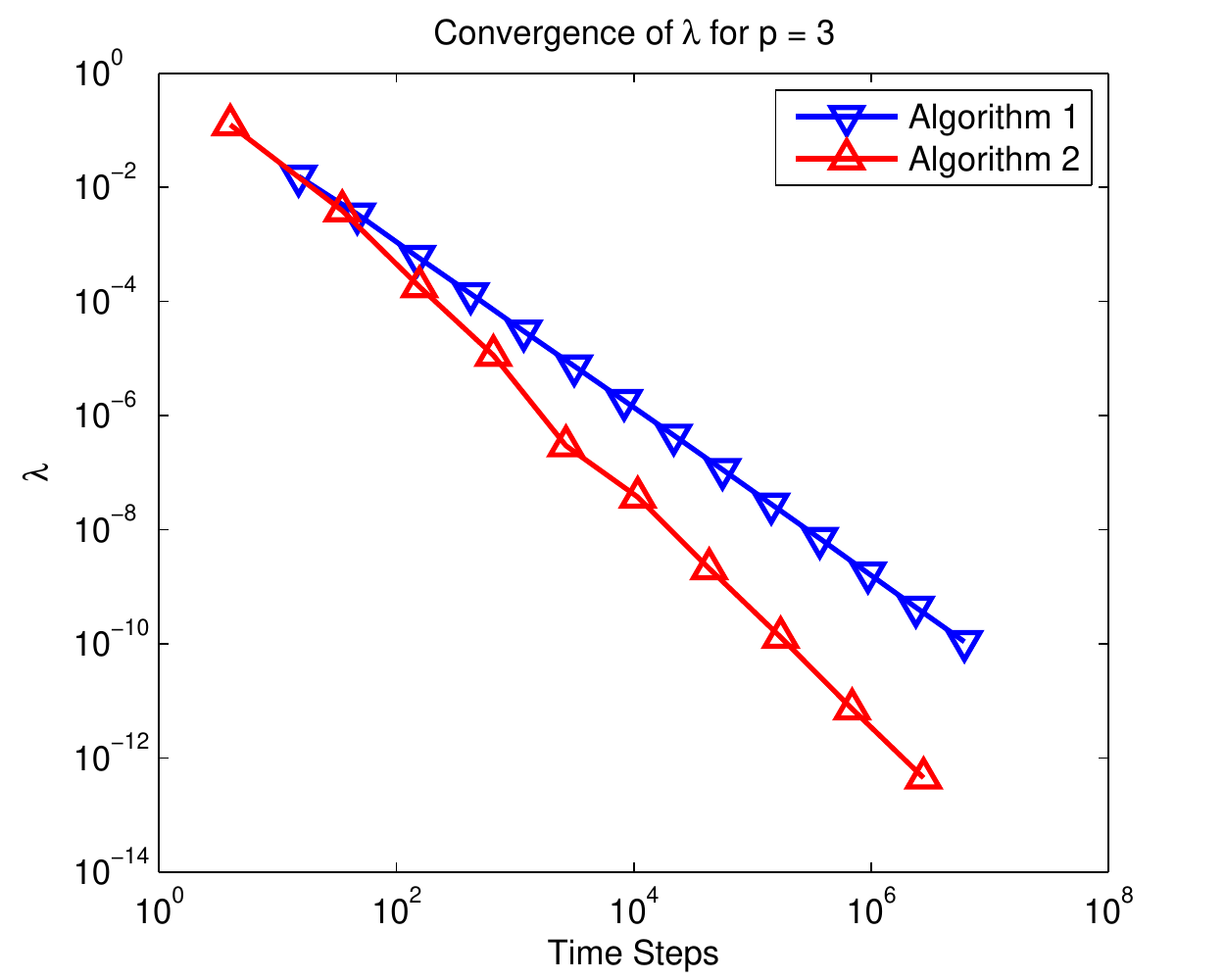}
\caption{Rates of $\lambda$ for improved Euler under ODE Algorithms 1 and 2.}
\label{graphODEimproved}
\end{figure}

The rates of $\lambda$ under Algorithm 4.2 are given in Table \ref{data2} and a more detailed comparison of the convergence of $\lambda$ for the different algorithms under the various time stepping schemes is given in Figures \ref{graphODEexplicit}, \ref{graphODEimplicit} and \ref{graphODEimproved}. These results show that under Algorithm 4.2 we have recovered our conjectured rates for the function $\lambda$ with a slight bonus rate for the explicit Euler method. Note that for $p=3$, Algorithm 4.1 converges faster than Algorithm 4.2 for the explicit Euler method; it is unknown why this is.

\begin{table}[!hbtp]
\caption{ODE Algorithm 2 Results} 
\centering 
\begin{tabular}{c c c c} 
\hline\hline 
Method & p = 2 & p = 3 \\
\hline 
Implicit Euler & $r \approx 1.00$ & $r \approx 1.00$ \\ 
Explicit Euler & $r \approx 1.45$ & $r \approx 1.43$ \\
Improved Euler & $r \approx 2.03$ & $r \approx 2.03$ \\
\hline 
\end{tabular}
\label{data2}
\end{table}
\end{section}

\begin{section}{Conclusions}

We derived an a posteriori error estimator for a class of nonlinear ODEs exhibiting blow-up and applied the estimator in two different adaptive algorithms to try and approximate the blow-up time. Both algorithms converged to the blow-up time in all test cases with Algorithm 4.2 outperforming Algorithm 4.1 in almost all test cases. In particular, we infer that explicit treatment of nonlinearities of the form \eqref{nonlinearform} appears to be advantageous in the context of adaptive algorithms based on rigorous a posteriori bounds.

\end{section}

\chapter{Adaptivity and blow-up detection for nonlinear non-stationary convection-diffusion problems}

\begin{section}{Blow-up in semilinear PDEs}\label{prelim}

For (non-fixed) $T > 0$, we consider the model problem of finding $u:\Omega \times (0,T] \to\mathbb{R}$ such that
\smallskip

\begin{equation}
\begin{aligned}
\label{blowup_model_strong}
\frac{\partial{u}}{\partial{t}} - \varepsilon\Delta{u}+ {\bf a} \cdot \nabla{u}+f(u) &= 0 \qquad & & \text{in }  \Omega \times (0,T]  \mbox{,} \\ 
u &=0 \mbox{ } &&\text{on }  \partial\Omega\times (0,T] \mbox{,} \\ 
u(\cdot,0) &=u_0 \mbox{ } &&\text{in } {\Omega}\mbox{.}
\end{aligned}
\end{equation}
It is assumed that the reaction term $f(u)$ is of the form $f(u) = f_0 - u^2$ although more general nonlinearities can be considered as discussed later in this chapter. We say that \eqref{blowup_model_strong} exhibits \emph{blow-up} if the solution $u$ has the property that
\begin{equation}
\begin{aligned}
\notag
\limsup_{t \to T^*} ||u(t)||_{L^{\infty}(\Omega)} & = \infty,
\end{aligned}
\end{equation}
for some $T^* > 0$. The value $T^*$ is referred to as the \emph{blow-up time} and if $T^* < \infty$ we say the PDE exhibits \emph{finite time blow-up}. If \eqref{blowup_model_strong} exhibits finite time blow-up then we can describe the asymptotic spatial behaviour of the solution $u$ through the \emph{blow-up set}, $\mathcal{B}$, given by 
\begin{equation}
\begin{aligned}
\notag
\mathcal{B} & := \{x \in \Omega \mbox{ } | \mbox{ }  \exists \{x_n,t_n \} \subset \Omega{\times}(0,T^*),  \mbox{ } t_n \rightarrow T^*,  \mbox{ }x_n \rightarrow x,  \mbox{ }u(x_n,t_n) \rightarrow \infty \}. 
\end{aligned}
\end{equation}
Elements of the blow-up set are referred to as \emph{blow-up points}. The asymptotic spatial behaviour of the solution to \eqref{blowup_model_strong} can be described through the blow-up set to be in one of two separate categories \cite{B77,GV02}:
\begin{itemize}
\item Point blow-up {\bf --} $\mathcal{B}$ consists of a finite number of points.

\item Regional blow-up {\bf --} The one-dimensional Hausdorff measure of $\mathcal{B}$ is positive.

\end{itemize}

\begin{figure}
\label{nascentdelta}
\centering
\includegraphics[scale=0.56]{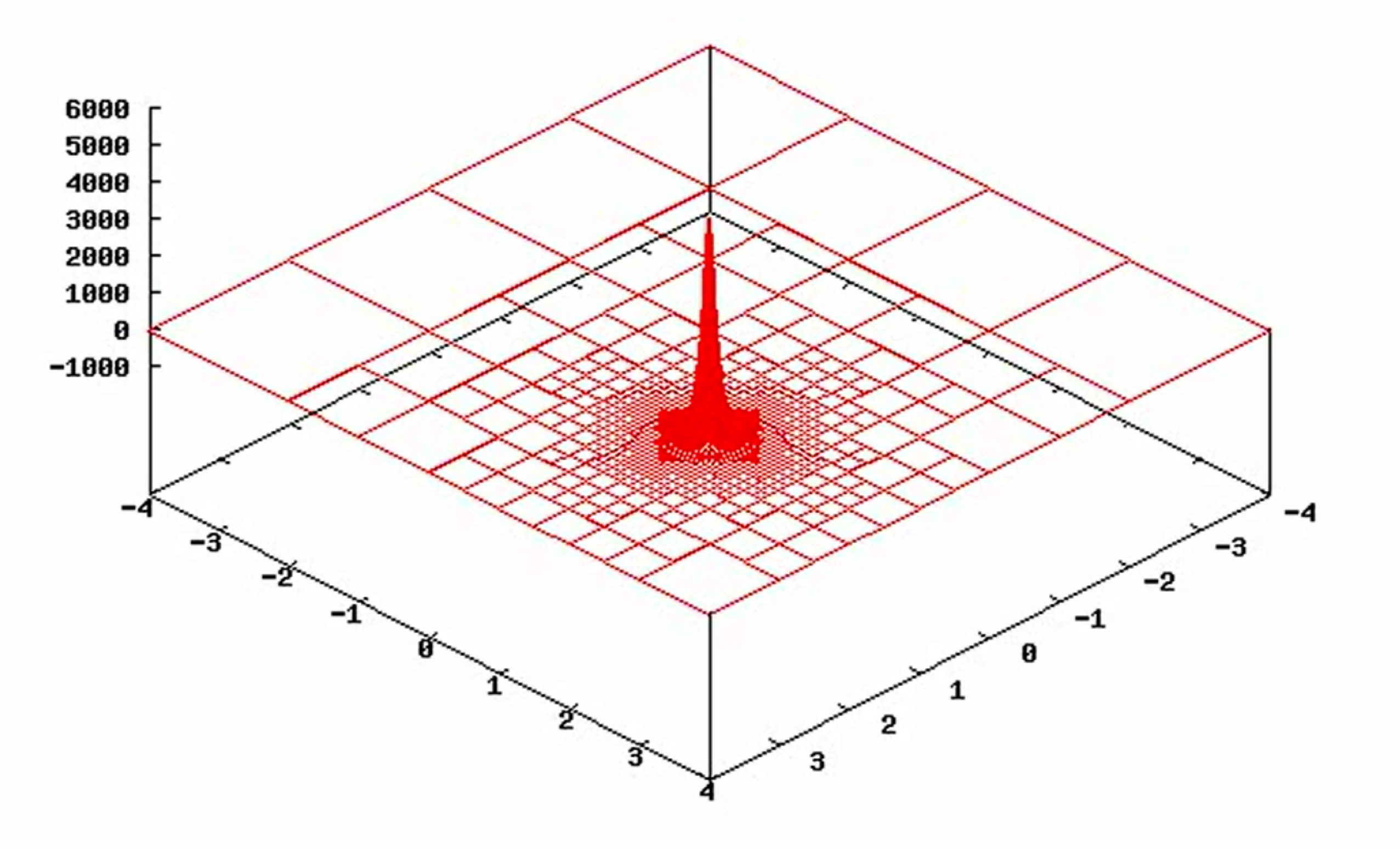}
\caption{Numerical approximation showing a typical solution profile near the blow-up time.}
\end{figure}

For single point blow-up, the solution looks like a nascent delta function close to the blow-up time (see Figure \ref{nascentdelta}); multi point blow-up or regional blow-up can cause even more complicated spatial behaviour near the blow-up time. These demanding and complex spatial and temporal features make the numerical \mbox{ } approximation of \eqref{blowup_model_strong} and related PDEs very difficult and necessitates the \mbox{ }  development of adaptive finite element methods. 

\noindent For $T<T^*$, the weak form of \eqref{blowup_model_strong} reads: find $u\in L^2 \big(0,T;H^1_0(\Omega) \big)\cap H^1 \big(0,T;L^2(\Omega) \big)$ such that for almost every $t \in (0,T]$ we have
\begin{equation}\label{blowup_model_weak}
\bigg(\frac{\partial{u}}{\partial{t}},v \bigg)+B(t;u,v)+(f(t;u),v)=0 \qquad \forall v \in H^1_0(\Omega),
\end{equation}
and the following assumptions are made on the PDE coefficients: $u_0 \in H^1_0(\Omega)$, $0 < \varepsilon \leq 1$, ${\bf a} \in [C(0,T;W^{1, \infty}(\Omega))]^2$ and $f_0 \in C(0,T;L^2(\Omega))$. For simplicity, we assume that $\nabla \cdot {\bf a} = 0$ and throughout the rest of this chapter, it is assumed that \eqref{blowup_model_weak} exhibits finite time blow-up.

\end{section}

\begin{section}{Space-time discretisation}

The numerical approximation of blow-up in nonlinear problems has been discussed in the literature. Solution profiles close to the blow-up time can be obtained through the rescaling algorithm of Berger and Kohn \cite{BK88,NZ14} or the MMPDE method \cite{BHR96,HMR08}. There is also work looking at the numerical approximation of blow-up in the  nonlinear Schr\"{o}dinger equation and its generalisations \cite{ADKM03,CS02,FI03,KMR11,TS92}. Other numerical methods for approximating blow-up in a variety of different nonlinear PDEs can be found in \cite{ABN09,DPMF05,DKKV98,FGR02,NB11}. Finally, in \cite{U00}, the author gives conditions that a numerical method must satisfy in order to asymptotically converge to the blow-up time. Most of these numerical methods rely on some form of theoretically justified rescaling; however, there is no general theory to know whether the resulting numerical approximation is reasonable or not. In this chapter, based on the work contained in \cite{CGKM15}, we shall use a simple numerical scheme to approximate \eqref{blowup_model_weak} and we will seek to perform our rescaling through rigorous a posteriori error estimates.

We consider a full discretisation of problem \eqref{blowup_model_weak} by using a finite difference method to approximate the time derivative, taking the convection-diffusion terms implicitly for stability purposes and the nonlinear reaction term explicitly in view of the conclusions drawn in the previous chapter. To this end, consider a subdivision of $[0,T]$ into time intervals of lengths $\tau_1$, ..., $\tau_n$ such that $\displaystyle \sum_{j=1}^n{\tau_j}=T$ for some $n \geq 1$ then set $t^0 := 0$ and $t^k := \displaystyle \sum_{j=1}^{k}\tau_{j}$. Denote an initial triangulation by $\zeta^0$ and we further associate to each time step $k>0$ a triangulation $\zeta^k$  which is assumed to have been obtained from $\displaystyle \zeta^{k-1}$ by locally refining and coarsening $\zeta^{k-1}$. To each mesh $\zeta^{k}$, we assign the finite element space $V_h^k := V_h(\zeta^k)$ given by~\eqref{eq:FEspace} and we also set $ {\bf a}^k:={\bf a}\big(\cdot,t^k \big)$ and $f^k:=f\big(\cdot,t^k;u_h^k \big)$ for brevity. Finally, for $t \in \big(t^k,t^{k+1}\big]$, we let $\Gamma$ denote the union of all edges in the mesh $\zeta^k \cup \zeta^{k+1}$.

The IMEX dG method then reads as follows. Set $u_h^0$ to be a projection of $u_0$ onto $V_h^0$. For $k=0$, ..., $n-1$, find $u_h^{k+1} \in V_h^{k+1}$ such that
\begin{equation}\label{dg_nonlin}
\bigg(\frac{u_h^{k+1}-u_h^k}{\tau_{k+1}},v_h^{k+1} \bigg)+B\big(t^{k+1};u_h^{k+1},v_h^{k+1} \big)+K_h \big(u_h^{k+1},v_h^{k+1} \big)+\big(f^{k},v_h^{k+1} \big)=0,
\end{equation}
for all  $v_h^{k+1} \in V_h^{k+1}$. We shall take $u_h^0$ to be the orthogonal $L^2$ projection of $u_0$ onto $V_h^0$, although other projections onto $V_h^0$ can also be used.\end{section}
 
\begin{section}{An a posteriori bound for the IMEX dG method}\label{fd_apost_sec}

The a posteriori error estimation of nonlinear parabolic problems has recently attracted attention for a variety of different PDEs \cite{B05, BM11,BBM05,CEGST14,DGN07,GM14,KNS04, V04,V98f,V98e}. With regards to blow-up, a key result is by Kyza and Makridakis \cite{K09,K01} wherein they produce an estimator for the error in the $L^{\infty}(L^{\infty})$ norm for a time semi-discrete approximation to the heat equation with polynomial nonlinearity. We shall use the ideas from \cite{K09,K01} to derive an error estimator for the IMEX dG scheme \eqref{dg_nonlin} in the $L^{\infty}(L^2)$ norm. A key novelty of the proof will be the use of a continuation argument for energy estimates, rather than the semigroup approach used in \cite{K09,K01}.

Before we begin to construct our error bound, we require some additional notation. At each time step $k$, we decompose the dG solution $u_h^k$ into a conforming part $u_{h,c}^k \in H^1_0(\Omega) \cap V_h^k$ and a non-conforming part $u_{h,d}^k \in V_h^k$ such that $u_h^k = u_{h,c}^k + u_{h,d}^k$. Further, given $t \in \big(t^k,t^{k+1} \big]$, we define $u_h(t)$ to be the linear interpolant with respect to $t$ of the values $u_h^k$ and $u_h^{k+1}$, viz.,
\begin{equation}
\notag
u_h(t):=l_k(t)u_h^k+l_{k+1}(t)u_h^{k+1},
\end{equation}
and we define $u_{h,c}(t)$ and $u_{h,d}(t)$ analogously. We can then decompose the error $e:=u-u_h=e_c-u_{h,d}$ where $e_c:=u-u_{h,c}$. It will also be useful to define the elliptic error $\epsilon^{k} := w^{k}-u_h^{k}$ where $w^k$ is given as in Definition \ref{ellipticreconstr2}.

\begin{lemma}\label{blowup_error_eq_simple_fd}
Given $t \in \big(t^k,t^{k+1} \big]$ then for any $v \in H^1_0(\Omega)$:
\begin{equation}\label{error_relation_fd}
\begin{aligned}
\notag
& \bigg(\frac{\partial{e}}{\partial{t}},v \bigg)+B(t;e,v)+(f(t;u)-f(t;u_h),v)=\bigg(-f(t;u_h)-\frac{\partial{u_h}}{\partial{t}},v \bigg)-B(t;u_h,v).
\end{aligned}
\end{equation}
\end{lemma}
\begin{proof}
This follows from \eqref{blowup_model_weak}. 
\end{proof}

\medskip

\noindent From Lemma \ref{blowup_error_eq_simple_fd} we have
\begin{equation}
\begin{aligned}
&\bigg (\frac{\partial{e}}{\partial{t}},v \bigg)+B\big(t;e,v \big)+\big(f(t;u )-f(t;u_h),v \big)=-\bigg(A^{k+1}+f^{k}+\frac{\partial{u_h}}{\partial{t}},v \bigg)\\&+B\big(t^{k+1};\epsilon^{k+1},v \big)-B\big(t;u_h,v \big)+B\big(t^{k+1};u_h^{k+1},v \big)+\big(f^{k}-f(t;u_h),v \big),
\end{aligned}
\end{equation}
which upon straightforward manipulation gives
\begin{equation}
\begin{aligned}
\label{equation2}
&\bigg(\frac{\partial{e}}{\partial{t}},v \bigg)+B\big(t;e,v \big)+\big(f(t;u )-f (t;u_h ),v \big)=-\bigg(A^{k+1}+f^{k}+\frac{\partial{u_h}}{\partial{t}},v \bigg)\\&+l_{k+1}B\big(t^{k+1};\epsilon^{k+1},v \big)+l_kB\big(t^k;\epsilon^k,v \big)-B\big(t;u_h,v \big)+l_{k+1}B\big(t^{k+1};u_h^{k+1},v\big)\\&+l_kB\big(t^k;u_h^k,v\big)+\big(f^{k}-f \big(t;u_h \big)+l_k\big(A^{k+1}-A^k \big),v \big).
\end{aligned}
\end{equation}
We are now ready to state our a posteriori estimator. The first part of our estimator is the initial condition estimator, $\eta_I$, given by
\begin{equation}
\begin{aligned}
\notag
\eta_I & := \left ( \big|\big|e(0)\big|\big|^2 +\sum_{E \in \mathcal{E}(\zeta^{0})} h_E \big|\big|\big[u_h^0 \big]\big|\big|^2_{L^2(E)} \right )^{1/2}.
\end{aligned}
\end{equation}
Due to the nature of the error bound to be presented, it is easier to separate the remainder of the estimator into two parts. As in Chapter 3, a subscript $S$ denotes parts of the estimator related to estimating space while a subscript $T$ denotes parts of the estimator related to estimating time. In this way, for $t \in \big(t^k,t^{k+1} \big]$, $\eta_A$ is given by
\begin{equation}
\begin{aligned}
\notag
\eta_A & :=   l_k \eta_{S_1,k}+l_{k+1} \eta_{S_1,k+1} + \eta_{S_2,k+1} + \eta_{T_1,k+1},
\end{aligned}
\end{equation}
where
\begin{equation}
\begin{aligned}
\notag
\eta_{S_1,k}  & :=  \left (\sum_{K \in \zeta^{k}} \frac{h_K^2}{\varepsilon}\big|\big|A^{k} + \varepsilon \Delta u_h^{k} - {\bf a}^{k} \cdot \nabla u_h^{k} \big|\big|^2_{L^2(K)}+\sum_{E \in \mathcal{E}(\zeta^{k})} \frac{h_E}{\varepsilon}\big|\big|\big[{\bf a}^{k} u_h^{k}\big]\big|\big|^2_{L^2(E)} \right. \\ & \left.+\sum_{E \in \mathcal{E}(\zeta^{k})} \frac{\gamma \varepsilon}{h_E} \big|\big|\big[u_h^{k}\big]\big|\big|^2_{L^2(E)}+\sum_{E \in \mathcal{E}^{int}(\zeta^{k})} \varepsilon h_E \big|\big|\big[\nabla u_h^{k}\big]\big|\big|^2_{L^2(E)}  \right)^{1/2}, \\
\eta_{S_2,k+1}  & :=  \left (\sum_{K \in \zeta^k \cup \zeta^{k+1}} \frac{h_K^2}{\varepsilon}\bigg|\bigg|f^{k}-I_h^{k+1}f^{k}- \frac{u_h^k - I_h^{k+1}u_h^k}{\tau_{k+1}} \bigg|\bigg|^2_{L^2(K)} \right )^{1/2}, \\ 
\eta_{T_1,k+1} & := {\varepsilon}^{-1/2}\big|\big|l_{k+1}\big({\bf a}^{k+1}-{\bf a} \big)u_h^{k+1}+l_k \big({\bf a}^k-{\bf a} \big)u_h^k \big|\big|,
\end{aligned}
\end{equation}
while $\eta_B$ is given by
\begin{equation}
\begin{aligned}
\notag
\eta_B & := \eta_{S_3,k+1}+\eta_{S_4,k+1}+\eta_{T_2,k+1},
\end{aligned}
\end{equation}
where
\begin{equation}
\begin{aligned}
\notag
\eta_{S_3,k+1}  & := \left (\sum_{K \in \zeta^k \cup \zeta^{k+1}} \sum_{E \subset \tilde{K}_E} \sigma^2_K h_E ||[u_h]||^2_{L^2(E)} \right )^{1/2}, \\
\eta_{S_4,k+1}  & :=  \left ( \sum_{E \subset \Gamma} h_E \left| \left|  \left[\frac{u_h^{k+1}-u_h^k}{\tau_{k+1}} \right]\right| \right|^2_{L^2(E)} \right)^{1/2}, \\
\eta_{T_2,k+1} & := \big|\big|f^k-f\big(t;u_h \big)+l_k \big(A^{k+1}-A^k \big)\big|\big|,
\end{aligned}
\end{equation}
with
\begin{equation}
\begin{aligned}
\notag
\sigma_K := 2||u_h||_{L^{\infty}(K)}+||[u_h]||_{L^{\infty}(\tilde{K}_E)}.
\end{aligned}
\end{equation}
Going back to \eqref{equation2}, the first term on the right can be bounded using Theorem \ref{projectionbounds} and the Cauchy-Schwarz inequality, viz.,
\begin{equation}
\begin{aligned}
\bigg(A^{k+1}+f^{k}+\frac{\partial{u_h}}{\partial{t}},v \bigg) & =\bigg(A^{k+1}+f^{k}+\frac{\partial{u_h}}{\partial{t}},v-I_h^{k+1}v \bigg) \\ &\lesssim \eta_{S_2,k+1}|||v|||.
\end{aligned}
\end{equation}
The next two terms give rise to parts of the space estimator via Theorem \ref{elliptic_apost}:
\begin{equation}
\begin{aligned}
l_{k+1}B\big(t^{k+1};\epsilon^{k+1},v \big)+l_kB\big(t^k;\epsilon^k,v \big) & \lesssim  ( l_k \eta_{S_1,k}+l_{k+1} \eta_{S_1,k+1})|||v|||.
\end{aligned}
\end{equation}
Using the definition of the bilinear form $B$ and the Cauchy-Schwarz inequality, the final four terms give rise to the time estimator:
\begin{equation}
\begin{aligned}
 l_{k+1}B\big(t^{k+1};u_h^{k+1},v \big)+l_kB\big(t^k;u_h^k,v \big) -B\big(t;u_h,v \big) & \leq \eta_{T_1,k+1}|||v|||, \\
\big(f^k-f\big(t;u_h \big)+l_k\big(A^{k+1}-A^k \big),v \big) & \leq \eta_{T_2,k+1}||v||.
\end{aligned}
\end{equation}
Setting $v=e_c$ in \eqref{equation2}, using the results above along with coercivity of the bilinear form $B$ and the Cauchy-Schwarz inequality we obtain
\begin{equation}
\begin{aligned}
& \frac{1}{2}\frac{d}{dt}||e_c||^2+|||e_c|||^2+(f(t;u )-f(t;u_h ),e_c )\lesssim   \bigg (\bigg|\bigg|\frac{\partial u_{h,d}}{\partial t}\bigg|\bigg|+\eta_{T_2,k+1} \bigg)||e_c|| \\ & +  ( l_k\eta_{S_1,k}+ l_{k+1}\eta_{S_1,k+1} + \eta_{S_2,k+1}+  \eta_{T_1,k+1} )|||e_c||| +B(t;u_{h,d},e_c).
\end{aligned}
\end{equation}
Using continuity of the bilinear form $B$ and Theorem \ref{nonconforming_bound} yields
\begin{equation}
\begin{aligned}
\label{equation3}
& \frac{1}{2}\frac{d}{dt}||e_c||^2+|||e_c|||^2+(f(t;u)-f(t;u_h),e_c )\lesssim  (\eta_{S_4,k+1}+\eta_{T_2,k+1})||e_c|| \\ &+ ( l_k\eta_{S_1,k}+ l_{k+1}\eta_{S_1,k+1} + \eta_{S_2,k+1} +  \eta_{T_1,k+1} )|||e_c|||.
\end{aligned}
\end{equation}
Using Young's inequality and the definition of our estimators, we conclude that
\begin{equation}
\begin{aligned}
\label{equation3}
&\frac{1}{2}\frac{d}{dt}||e_c||^2+\frac{1}{2}|||e_c|||^2+(f(t;u)-f(t;u_h),e_c ) \lesssim  \frac{1}{2}\eta^2_A + \eta_B||e_c||.
\end{aligned}
\end{equation}
We must now deal with the nonlinear term. We begin by noting that
\begin{equation}
\begin{aligned}
\label{horriblenonlinear}
& (f(t;u)-f(t;u_h),e_c) = (f(t;e_c-u_{h,d}+u_h)-f(t;u_h),e_c) = T_1 + T_2 ,
\end{aligned}
\end{equation}
where
\begin{equation}
\begin{aligned}
\notag
 T_1 & :=  \big(2u_h u_{h,d},e_c \big) - \big(u_{h,d}^2,e_c \big) , \\
T_2 & :=  -\big(2u_he_c,e_c \big) + \big(2e_c u_{h,d},e_c\big) -  \big(e_c^2,e_c \big).
\end{aligned}
\end{equation}
We can write the contributions to $T_1$ elementwise and then use the Cauchy-Schwarz inequality and Theorem \ref{ncbounds} to conclude that
\begin{equation}
\begin{aligned}
\label{Tbound1}
|T_1| & \leq \left (\sum_{K \in \zeta^k \cup \zeta^{k+1}} \big (2||u_h||_{L^{\infty}(K)}+||u_{h,d}||_{L^{\infty}(K)}  \big )^2||u_{h,d}||^2_{L^2(K)} \right )^{1/2}||e_c|| \\
& \lesssim \eta_{S_3,k+1}||e_c||.
\end{aligned}
\end{equation}
To bound $T_2$, we use H{\"o}lder's inequality along with Theorem \ref{ncbounds} to conclude that
\begin{equation}
\begin{aligned}
\label{Tbound2}
|T_2| & \lesssim \big (2||u_h||_{L^{\infty}(\Omega)} + ||[u_h]||_{L^{\infty}(\Gamma)} \big)||e_c||^2+||e_c||^3_{L^3(\Omega)}.
\end{aligned}
\end{equation}
Combining \eqref{equation3}, \eqref{horriblenonlinear}, \eqref{Tbound1} and \eqref{Tbound2} we obtain 
\begin{equation}
\begin{aligned}
\label{PDEerrorequation}
\frac{d}{dt}||e_c||^2+|||e_c|||^2 & \lesssim  \eta^2_A + 2\eta_B ||e_c|| + 2\sigma_{\Omega}||e_c||^2 + 2||e_c||^3_{L^3(\Omega)},
\end{aligned}
\end{equation}
with 
\begin{equation}
\begin{aligned}
\notag
\sigma_{\Omega} := 2||u_h||_{L^{\infty}(\Omega)} +C||[u_h]||_{L^{\infty}(\Gamma)},
\end{aligned}
\end{equation}
where $C$ is a generic constant. We now note the Gagliardo-Nirenberg inequality which coupled with Young's inequality yields
\begin{equation}
\begin{aligned}
\label{gagnir}
 ||e_c||^3_{L^3(\Omega)} & \leq K||e_c||^2||\nabla{e_c}|| \leq \frac{1}{2}|||e_c|||^2 + \frac{K^2}{2\varepsilon}||e_c||^4.
 \end{aligned}
\end{equation}
Combining these results with the generic constant $C$ yields
\begin{equation}
\begin{aligned}
\label{equation4}
\frac{d}{dt}||e_c||^2 & \leq  C\eta^2_A + 2C\eta_B ||e_c|| + 2\sigma_{\Omega}||e_c||^2 + K^2 \varepsilon^{-1}||e_c||^4.
\end{aligned}
\end{equation}
We now need a way to deal with the $L^2$ norms appearing on the right-hand side. To do this, we use a variant of Gronwall's inequality (see Theorem \ref{Gronwall}). In order to apply this to \eqref{equation4}, we need to introduce some new notation. We define
\begin{equation}
\begin{aligned}
\notag
G_{k+1} & := \exp \Bigg(\int_{t^k}^{t^{k+1}} \! \sigma_{\Omega} \, ds \Bigg ), \\
H_{k+1}(t) & := \exp \Bigg ({K^2}{\varepsilon}^{-1}\int_{t^k}^t \! ||e_c||^2 \, ds  \Bigg).
\end{aligned}
\end{equation}
Then application of Theorem \ref{Gronwall} to \eqref{equation4} for $t \in \big[t^k,t^{k+1}\big]$ yields
\begin{equation}
\begin{aligned}
\label{equation5}
||e_c(t)|| & \leq H_{k+1}(t)G_{k+1}\phi_{k+1},
\end{aligned}
\end{equation}
where
\begin{equation}
\begin{aligned}
\notag
\phi_{k+1} &:= \left (\big|\big|e_c \big(t^k \big)\big|\big|^2 + C\int_{t^k}^{t^{k+1}} \! \eta_A^2 \, ds \right)^{1/2}+ C \int_{t^k}^{t^{k+1}} \! \eta_B \, ds.
\end{aligned}
\end{equation}
We now need to remove the term $H_{k+1}$ from \eqref{equation5} in order to construct a usable estimator. In order to do this, we use a continuation argument. To that end, we define the set
\begin{equation}
\begin{aligned}
\notag
I_{k+1} & := \big \{t \in \big[t^k,t^{k+1} \big] \mbox{ } \big | \mbox{ } ||e_c||_{L^{\infty}(t^k,t;L^2(\Omega))} \leq \delta_{k+1}G_{k+1} \phi_{k+1} \big \},
\end{aligned}
\end{equation}
where $\delta_{k+1} > 1$ should be chosen as small as possible. We know that $I_{k+1}$ is non-empty since $t^k \in I_{k+1}$ and obviously bounded. Denote the maximal value of $t$ in $I_{k+1}$ by $t^*$ and assume that $t^* < t^{k+1}$ then from \eqref{equation5} we have
\begin{equation}
\begin{aligned}
\label{equation6}
||e_c||_{L^{\infty}(t^k,t^*;L^2(\Omega))} & \leq H(t^*)G_{k+1}\phi_{k+1} \\ 
& \leq \exp\Big (K^2 {\varepsilon}^{-1}\tau_{k+1}||e_c||_{L^{\infty}(t^k,t^*;L^2(\Omega))}^2  \Big)G_{k+1}\phi_{k+1} \\ 
& \leq \exp\Big (K^2 {\varepsilon}^{-1}\tau_{k+1}\delta^2_{k+1}G^2_{k+1}\phi^2_{k+1} \Big )G_{k+1}\phi_{k+1}.
\end{aligned}
\end{equation}
Now, suppose that
\begin{equation}
\begin{aligned}
\exp \Big (K^2 {\varepsilon}^{-1}\tau_{k+1}\delta^2_{k+1}G^2_{k+1}\phi^2_{k+1} \Big )G_{k+1}\phi_{k+1} & < \delta_{k+1}G_{k+1} \phi_{k+1},
\end{aligned}
\end{equation}
or equivalently,
\begin{equation}
\begin{aligned}
\label{cond1}
K^2 {\varepsilon}^{-1}\tau_{k+1}\delta^2_{k+1}G^2_{k+1}\phi^2_{k+1} & < \log(\delta_{k+1}),
\end{aligned}
\end{equation}
then $t^*$ cannot be the maximal value of $t$ in $I_{k+1}$ because we just showed that $\displaystyle ||e_c||_{L^{\infty}(t^k,t^*;L^2(\Omega))}$ satisfies a bound strictly less than that assumed in the set $I_{k+1}$ {\bf--} a contradiction. Therefore, providing \eqref{cond1} is satisfied, $I_{k+1} = \big[t^k,t^{k+1}\big]$ and we have our desired error bound once we select $\delta_{k+1}$. Taking the limit, we can select $\delta_{k+1}$ to be the minimiser of
\begin{equation}
\begin{aligned}
\label{cond2}
K^2 {\varepsilon}^{-1}\tau_{k+1}\delta^2_{k+1}G^2_{k+1}\phi^2_{k+1} - \log(\delta_{k+1}) = 0, \qquad \delta_{k+1} > 1.
\end{aligned}
\end{equation}
In order to obtain our error estimator, all that remains is to estimate $\phi_1$. Application of Theorem \ref{ncbounds} and the triangle inequality yields
\begin{equation}
\begin{aligned}
||e_c(0)||^2 & \lesssim ||e(0)||^2 + ||u_{h,d}(0)||^2 \leq C\eta_I^2.
\end{aligned}
\end{equation}
Therefore, if we (re)define $\phi_1$ to be 
\begin{equation}
\begin{aligned}
\notag
\phi_{1} &:= \left (C\eta^2_I + C\int_{t^0}^{t^1} \! \eta_A^2 \, ds \right)^{1/2}+ C\int_{t^0}^{t^{1}} \! \eta_B \, ds,
\end{aligned}
\end{equation}
then we have
\begin{equation}
\begin{aligned}
\big|\big|e_c\big(t^1\big)\big|\big| & \leq ||e_c||_{L^{\infty}(t^0,t^1;L^2(\Omega))} \leq \psi_1,
\end{aligned}
\end{equation}
where $\psi_1 := \delta_1G_1\phi_1$. In the same way, if we (re)define 
\begin{equation}
\begin{aligned}
\notag
\phi_{k+1} & := \left (\psi^2_k + C\int_{t^k}^{t^{k+1}} \! \eta_A^2 \, ds \right)^{1/2}+ C \int_{t^k}^{t^{k+1}} \! \eta_B \, ds, \\
\psi_{k+1} &:= \delta_{k+1}G_{k+1}\phi_{k+1},
\end{aligned}
\end{equation}
then we have
\begin{equation}
\begin{aligned}
\big|\big|e_c\big(t^{k+1}\big)\big|\big| \leq ||e_c||_{L^{\infty}(t^k,t^{k+1};L^2(\Omega))} \leq \psi_{k+1}.
\end{aligned}
\end{equation}
Hence, the following result holds.
\begin{theorem}
\label{MainTheorem}
The error of the IMEX dG discretisation of problem \eqref{blowup_model_weak} satisfies
\begin{equation}
\begin{aligned}
\notag
||e||_{L^{\infty}(0,T;L^2(\Omega))} \lesssim \psi_n + \esssup_{0 \leq t \leq T} \left ( \sum_{E \subset \Gamma} h_E ||[u_h]||^2_{L^2(E)} \right)^{1/2},
\end{aligned}
\end{equation}
providing that the solution to \eqref{cond2} exists for all time steps.
\end{theorem}
\begin{proof}
Follows from the above derivations, the triangle inequality and the bounds in Theorem \ref{ncbounds}.
\end{proof}

The estimator produced above is suboptimal with respect to the mesh-size as it is only spatially optimal in the $L^2 \big(H^1 \big)$ norm. It is possible to conduct a continuation argument for the $L^2 \big(H^1\big)$ norm rather than the $L^{\infty}\big(L^2 \big)$ norm if one desires a spatially optimal error estimator; this is stated for completeness in the theorem below. However, the resulting $\delta$ equation was observed to be more restrictive with regards to how quickly the blow-up time is approached. For this reason, we opt to use the a posteriori error estimator of Theorem~\ref{MainTheorem} in the adaptive algorithm introduced in the next section.

\smallskip

\begin{theorem}
The error of the IMEX dG discretisation of problem \eqref{blowup_model_weak} satisfies
\begin{equation}
\notag
\left( ||e(T)||^2 + \int_0^T \! \varepsilon||\nabla e||^2 \, dt \right)^{1/2} \lesssim \sum_{k=1}^n \psi_k + \esssup_{0 \leq t \leq T} 
\left( \sum_{E \subset \Gamma} h_E ||[u_h]||^2_{L^2(E)} \right)^{1/2}.
\end{equation}
Furthermore, close to the blow-up time where $||e(T)|| = ||e||_{L^{\infty}(0,T;L^2(\Omega))}$ we have
\begin{equation}
\notag
||e||_*  \lesssim \sum_{k=1}^n \psi_k + \esssup_{0 \leq t \leq T} 
\left( \sum_{E \subset \Gamma} h_E ||[u_h]||^2_{L^2(E)} \right)^{1/2},
\end{equation}
where $\psi_k$, $k=1$, ..., $n$, is defined recursively with $\psi_0 = C\eta_I$ and
\begin{equation}
\begin{aligned}
\notag
\phi_{k} & := \left (\psi^2_{k-1} + C\int_{t^{k-1}}^{t^{k}} \! \eta_A^2 \, ds + C \int_{t^{k-1}}^{t^{k}} \! \eta^2_B \, ds \right)^{1/2}, \\
{G}_{k} & := \exp(\tau_k / 2) \exp \left(\int_{t^{k-1}}^{t^{k}} \! \sigma_{\Omega} \, ds \right ), \\
\psi_{k} &:= \delta_{k}{G}_{k}\phi_{k},
\end{aligned}
\end{equation}
provided that $\delta_k > 1$ which is the smallest root of the equation
\begin{equation}
\begin{aligned}
\notag
K\varepsilon^{-1/2}\tau^{1/2}_{k}\delta_{k}G_{k}\phi_{k} - \log(\delta_{k})=0,
\end{aligned}
\end{equation}
exists for all time steps.
\end{theorem}
\begin{proof}
The proof is completely analogous to that of Theorem \ref{MainTheorem} and follows from \eqref{PDEerrorequation} by conducting a continuation argument for the $L^2 \big(H^1 \big)$ norm.
\end{proof}

\begin{remark}
Although we considered a very simple nonlinearity, the continuation argument  in this section can be modified to include any nonlinearity of the form $f(u)=f_0 + f_1u+f_2u^2+f_3u^3$. With a nonlinearity of this form, we would have to deal with the term $||e_c||^4_{L^4(\Omega)}$ in the error equation. From the Gagliardo-Nirenberg inequality, we have
\begin{equation}
\begin{aligned}
\notag
||e_c||^4_{L^4(\Omega)} \leq ||e_c||^2||\nabla e_c||^2.
\end{aligned}
\end{equation}
After application of Gronwall's inequality, we have a term of the form
\begin{equation}
\begin{aligned}
\notag
\exp\Bigg(\int_{t^k}^{t^{k+1}} \! ||\nabla e_c||^2 \, ds \Bigg),
\end{aligned}
\end{equation}
which can be bounded through a continuation argument that uses the $L^2(H^1)$ norm. Any higher order nonlinearities could not be dealt with in this way.

\begin{remark}
 It is also worth noting that although we are primarily interested in looking at blow-up problems, the estimators developed in this section are still perfectly valid if blow-up does not occur.
\end{remark}

\end{remark}

\end{section}
\begin{section}{An adaptive algorithm}\label{adaptive_sec}
The a posteriori bounds presented in the previous section will be used to drive a space-time adaptive algorithm that is designed to approximate the blow-up time of problem \eqref{blowup_model_weak}. The pseudocode of this algorithm is given in Algorithm 5.1.
\begin{algorithm} \label{adaptive_algorithm}
  \begin{algorithmic}[1]
     \State {\bf Input:} $\varepsilon$, ${\bf a}$, $f_0$, $u_0$, $\Omega$,  $\tau_1$, $\zeta^0$, $\gamma$, ${\tt ttol^+}$, ${\tt ttol^-}$, ${\tt stol^+}$, ${\tt stol^-}$.
     \State Calculate $u_h^0$.
 \State Calculate $u_h^1$ from $u_h^0$.
    \While {$\displaystyle \int_{t^0}^{t^1} \! \eta^2_{T_2,1} \, ds > {\tt ttol^+} \text{ OR } \max_K \eta^2_{S_1,1} |_K > {\tt stol^+} $}
    \State Modify $\zeta^0$ by refining all elements such that $\eta^2_{S_1,1} |_K > {\tt stol^+}$ and coarsening all elements such that $\eta^2_{S_1,1} |_K < {\tt stol^-}$.
\If   {$\displaystyle \int_{t^0}^{t^1} \! \eta^2_{T_2,1} \, ds > {\tt ttol^+}$}
   
 \State $\tau_{1} \leftarrow \tau_{1}/2$.

    \EndIf

     \State Calculate $u_h^0$.
 \State Calculate $u_h^1$ from $u_h^0$.
    \EndWhile

\State Calculate $\delta_1$.

 \State Multiply ${\tt ttol^+}$, ${\tt ttol^-}$, ${\tt stol^+}$, ${\tt stol^-}$ by the factor $G_{1}$.

     \State  Set $j = 0$,  $\zeta^1 = \zeta^0$.

\While {$\delta_{j+1}$ exists}
\State $ j \leftarrow j+1$. 
\State $\tau_{j+1} = \tau_j$. 

    \State Calculate $u_h^{j+1}$ from $u_h^j$.
    \If {  $\displaystyle\int_{t^j}^{t^{j+1}} \! \eta^2_{T_2,j+1} \, ds> {\tt ttol^+}$ }
     
 \State $\tau_{j+1} \leftarrow \tau_{j+1}/2$.
\State Calculate $u_h^{j+1}$ from $u_h^j$.
    \EndIf
    \If {  $\displaystyle\int_{t^j}^{t^{j+1}} \! \eta^2_{T_2,j+1} \, ds < {\tt ttol^-}$ }
     
 \State $\tau_{j+1} \leftarrow 2\tau_{j+1}$.
\State Calculate $u_h^{j+1}$ from $u_h^j$.
    \EndIf
    \State Create $\zeta^{j+1}$ from $\zeta^j$ by refining all elements such that $\eta^2_{S_1,j+1} |_K > {\tt stol^+}$ and coarsening all elements such that $\eta^2_{S_1,j+1} |_K < {\tt stol^-}$.
    \State Calculate $u_h^{j+1}$ from $u_h^j$.
    \State Calculate $\delta_{j+1}$.
 \State Multiply ${\tt ttol^+},{\tt ttol^-},{\tt stol^+},{\tt stol^-}$ by the factor $G_{j+1}$.
\EndWhile

\State {\bf Output:} $j$, $t^j$, $\big|\big|u_h\big (t^j \big)\big|\big|_{L^{\infty}(\Omega)}$.
  \end{algorithmic}
  \caption{Space-time adaptivity}
\end{algorithm}

As in Algorithm 3.1, both mesh refinement and coarsening are driven by the term $\eta_{S_1,k+1}$. The size of the elemental contributions to $\eta_{S_1,k+1}$ determines whether the elements are to be refined, coarsened or neither depending on two spatial thresholds ${\tt stol^+}$ and ${\tt stol^-}$. Similarly, $\eta_{T_2,k+1}$ is used to drive temporal refinement and coarsening subject to two temporal thresholds ${\tt ttol^+}$ and ${\tt ttol^-}$ on each time interval. As in Algorithm 4.2, all spatial and temporal thresholds are increased by the factor $G_{k+1}$ on the interval $\big [t^k,t^{k+1} \big]$ after the solution has been calculated. The algorithm then advances by using the previous (now fixed) time step length as a reference to compute the next approximation. The algorithm continues in this way until \eqref{cond2} no longer has a solution and the algorithm then terminates and outputs the total number of time steps, the final time and the $L^{\infty}(L^{\infty})$ norm of the IMEX dG solution.

\end{section}

\begin{section}{Numerical Experiments}

We shall numerically investigate the presented a posteriori bound and the \mbox{ } performance of the adaptive algorithm through an implementation based on the {\tt deal.II} finite element library \cite{BHK07}. All the numerical experiments have been performed using the  high performance computing facility ALICE at the University of Leicester. For all the numerical experiments, we use polynomials of degree five. Finally, we set ${\tt ttol}^- = 0.01*{\tt ttol}^+$ and ${\tt stol}^- = 10^{-6}*{\tt stol}^+$ as our temporal and spatial coarsening parameters.

\begin{remark}
All unknown constants in the error estimators are set equal to one as is standard in a posteriori error analysis. It is believed that this is reasonable despite the fact that condition \eqref{cond1} is technically a strict limitation on whether or not we can continue our computations.
\end{remark}

\begin{subsection}{Example 1}
Let $\Omega = (-4,4)^2$, $\varepsilon = 1$, ${\bf a} = (0,0)^T$, $f_0 = 0$ and $u_0 = 10 e^{-2(x^2+y^2)}$. The initial condition $u_0$ is chosen to be a Gaussian blob centred on the origin that is chosen `large enough' so that the solution exhibits blow-up; the blow-up set consists of a single point corresponding to the centre of the Gaussian. In order to observe how the error estimator behaves asymptotically, we choose a very small spatial threshold so that the spatial contribution to the error and the estimator are small. We then reduce the temporal threshold and see how far we can advance towards the blow-up time. The results are given in Table \ref{blowupdata1}.

\begin{table}[ht]
\caption{Example 1 Results} 
\centering 
\begin{tabular}{c c c c c} 
\hline\hline 
${\tt ttol^+}$ & Time Steps & Estimator & Final Time & $||u_h(T)||_{L^{\infty}(\Omega)}$ \\ 
\hline 
1 & 3 & 9.5 & 0.09375 & 12.244 \\ 
0.125 & 8 & 24.6 & 0.12500 & 14.742 \\
$(0.125)^2$ & 19 & 54.0 & 0.14844 & 18.556 \\
$(0.125)^3$ & 42 & 66.7 & 0.16406 & 23.468 \\
$(0.125)^4$ & 92 & 218.5 & 0.17969 & 32.108 \\
$(0.125)^5$ & 195 & 1142.4 & 0.19043 & 44.217 \\
$(0.125)^6$ & 405 & 1506.0 & 0.19775 & 60.493 \\
$(0.125)^7$ & 832 & 1754.1 & 0.20313 & 83.315 \\
$(0.125)^8$ & 1698 & 5554.2 & 0.20728 & 117.780 \\
$(0.125)^9$ & 3443 & 6020.4 & 0.21014 & 165.833 \\
$(0.125)^{10}$ & 6956 & 33426.7 & 0.21228 & 238.705 \\
$(0.125)^{11}$ & 14008 & 36375.0 & 0.21375 & 343.078 \\
$(0.125)^{12}$ & 28151 & 66012.8 & 0.21478 & 496.885 \\
$(0.125)^{13}$ & 56489 & 157300.0 & 0.21549 & 722.884 \\
\hline 
\end{tabular}
\label{blowupdata1}
\end{table}

We know that asymptotically the solutions to \eqref{blowup_model_weak} behave the same temporally as the solutions to \eqref{ODE} (at least in the case of zero convection) \cite{H11}. This means that if our error estimator is good, we would expect to observe similar rates for $\lambda$ to those seen in Chapter 4. Although we do not know the blow-up time for this problem, we observe from Table \ref{blowupdata1} that 
\begin{equation}
\begin{aligned}
\notag
||u_h||_{L^{\infty}(0,T;L^{\infty}(\Omega))} \propto N^{1/2}.
\end{aligned}
\end{equation}
From \cite{H11}, we know the relationship between the magnitude of the exact solution in the $L^{\infty}(L^{\infty})$ norm and the distance from the blow-up time. Thus, under the assumption that the numerical solution is scaling like the exact solution we get
\begin{equation}
\begin{aligned}
\notag
\lambda({\tt ttol^+},N) \approx ||u||^{-1}_{L^{\infty}(0,T;L^{\infty}(\Omega))} \approx ||u_h||^{-1}_{L^{\infty}(0,T;L^{\infty}(\Omega))}.
\end{aligned}
\end{equation}
Therefore, we conjecture that
\begin{equation}
\begin{aligned}
\notag
\lambda({\tt ttol^+},N) \propto N^{-1/2}.
\end{aligned}
\end{equation}
This is obviously slower than the comparable results in Chapter 4 and a possible explanation for this will be discussed in the conclusions section.

\end{subsection}

\begin{subsection}{Example 2}

Let $\Omega = (-4,4)^2$, $\varepsilon = 1$, ${\bf a} = (1,1)^T$, $f_0 = -1$ and $u_0 = 0$. This numerical example is interesting to study as not much is known about blow-up problems that contain convection. Here, the solution behaves as the solution to a standard convection-diffusion problem early on. As time progresses, the nonlinear term takes over and the solution begins to exhibit blow-up. As in Example 1, we choose to use a small spatial threshold so the spatial contribution to the error and the estimator are negligible. We then reduce the temporal threshold and see how far we can advance towards the blow-up time. The results are given in Table \ref{blowupdata2}.

\begin{table}[ht]
\caption{Example 2 Results} 
\centering 
\begin{tabular}{c c c c c} 
\hline\hline 
${\tt ttol^+}$ & Time Steps & Estimator & Final Time & $||u_h(T)||_{L^{\infty}(\Omega)}$ \\ 
\hline 
1 & 4 & 3.6 & 0.78125 & 0.886 \\ 
0.125 & 10 & 3.6 & 0.97656 & 1.322 \\
$(0.125)^2$ & 54 & 22.0 & 1.31836 & 3.269 \\
$(0.125)^3$ & 119 & 47.5 & 1.41602 & 5.107 \\
$(0.125)^4$ & 252 & 132.1 & 1.48163 & 8.059 \\
$(0.125)^5$ & 520 & 218.4 & 1.51711 & 11.819 \\
$(0.125)^6$ & 1064 & 664.6 & 1.54467 & 18.139 \\
$(0.125)^7$ & 2158 & 1466.1 & 1.56224 & 27.405 \\
$(0.125)^8$ & 4354 & 1421.7 & 1.57402 & 41.374 \\
$(0.125)^9$ & 8792 & 11423.0 & 1.58243 & 64.450 \\
$(0.125)^{10}$ & 17713 & 21497.8 & 1.58770 & 99.190 \\
$(0.125)^{11}$ & 35580 & 21097.1 & 1.59092 & 145.785 \\
$(0.125)^{12}$ & 71352 & 35862.0 & 1.59299 & 211.278 \\
\hline 
\end{tabular}
\label{blowupdata2}
\end{table}
\noindent From Table \ref{blowupdata2}, we draw the conclusion that
\begin{equation}
\begin{aligned}
\notag
||u_h||_{L^{\infty}(0,T;L^{\infty}(\Omega))} \propto N^{1/2}.
\end{aligned}
\end{equation}
Although not much is known about blow-up problems that contain convection, it is reasonable to assume that because the nonlinear term dominates close to the blow-up time, the same relationship between the magnitude of the exact solution in the $L^{\infty}(L^{\infty})$ norm and distance from the blow-up time exists as in Example 1. If this is true, then under the same reasoning as in Example 1 we conclude that
\begin{equation}
\begin{aligned}
\notag
\lambda({\tt ttol^+},N) \propto N^{-1/2}.
\end{aligned}
\end{equation}

\end{subsection}

\begin{subsection}{Example 3}

Let $\Omega = (-8,8)^2$, $\varepsilon = 1$, ${\bf a} = (0,0)^T$, $f_0 = 0$ and the `volcano' type initial condition be given by $u_0 = 10 \big(x^2+y^2 \big){e}^{-0.5(x^2+y^2)}$. The blow-up set for this example is a circle centred on the origin {\bf --} this induces layer type phenomena in the solution around the blow-up set as the blow-up time is approached making this example a good test of the spatial capabilities of the adaptive algorithm. Once more, we choose a small spatial threshold so that the spatial contribution to the error and the estimator are negligible. We then reduce the temporal threshold and see how far we can advance towards the blow-up time. The results are given in Table \ref{blowupdata3}.

\begin{table}[ht]
\caption{Example 3 Results} 
\centering 
\begin{tabular}{c c c c c} 
\hline\hline 
${\tt ttol^+}$ & Time Steps & Estimator & Final Time & $||u_h(T)||_{L^{\infty}(\Omega)}$ \\ 
\hline 
8 & 3 & 15 & 0.06250 & 10.371 \\ 
1 & 10 & 63 & 0.09375 & 14.194 \\
$0.125$ & 36 & 211 & 0.11979 & 21.842 \\
$(0.125)^2$ & 86 & 533 & 0.13412 & 31.446 \\
$(0.125)^3$ & 190 & 971 & 0.14388 & 45.122 \\
$(0.125)^4$ & 404 & 1358 & 0.15072 & 64.907 \\
$(0.125)^5$ & 880 & 5853 & 0.15601 & 98.048 \\
$(0.125)^6$ & 1853 & 10654 & 0.15942 & 146.162 \\
$(0.125)^7$ & 3831 & 21301 & 0.16176 & 219.423 \\
$(0.125)^8$ & 7851 & 143989 & 0.16336 & 332.849 \\
$(0.125)^9$ & 16137 & 287420 & 0.16442 & 505.236 \\
$(0.125)^{10}$ & 32846 & 331848 & 0.16512 & 769.652 \\
$(0.125)^{11}$ & 66442 & 626522 & 0.16558 & 1175.21 \\
\hline 
\end{tabular}
\label{blowupdata3}
\end{table}
\noindent Once again, the data implies that
\begin{equation}
\notag
\|u_h\|_{L^{\infty}(0,T;L^{\infty}(\Omega))} \propto N^{1/2}.
\end{equation}
Arguing as in Example 1, we again conclude that
\begin{equation}
\begin{aligned}
\notag
\lambda({\tt ttol^+},N) \propto N^{-1/2}.
\end{aligned}
\end{equation}

\begin{figure}[h]
\centering
\includegraphics[scale=0.55]{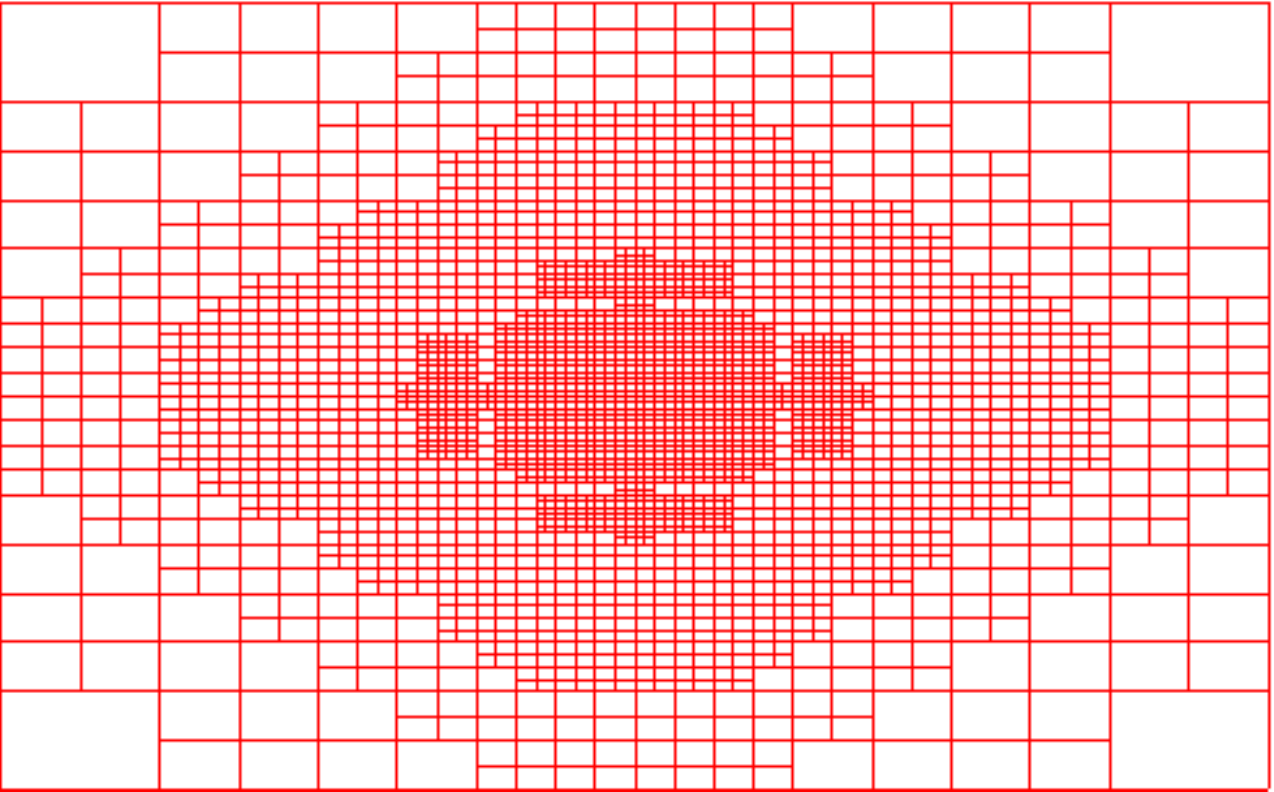} \, \includegraphics[scale=0.55]{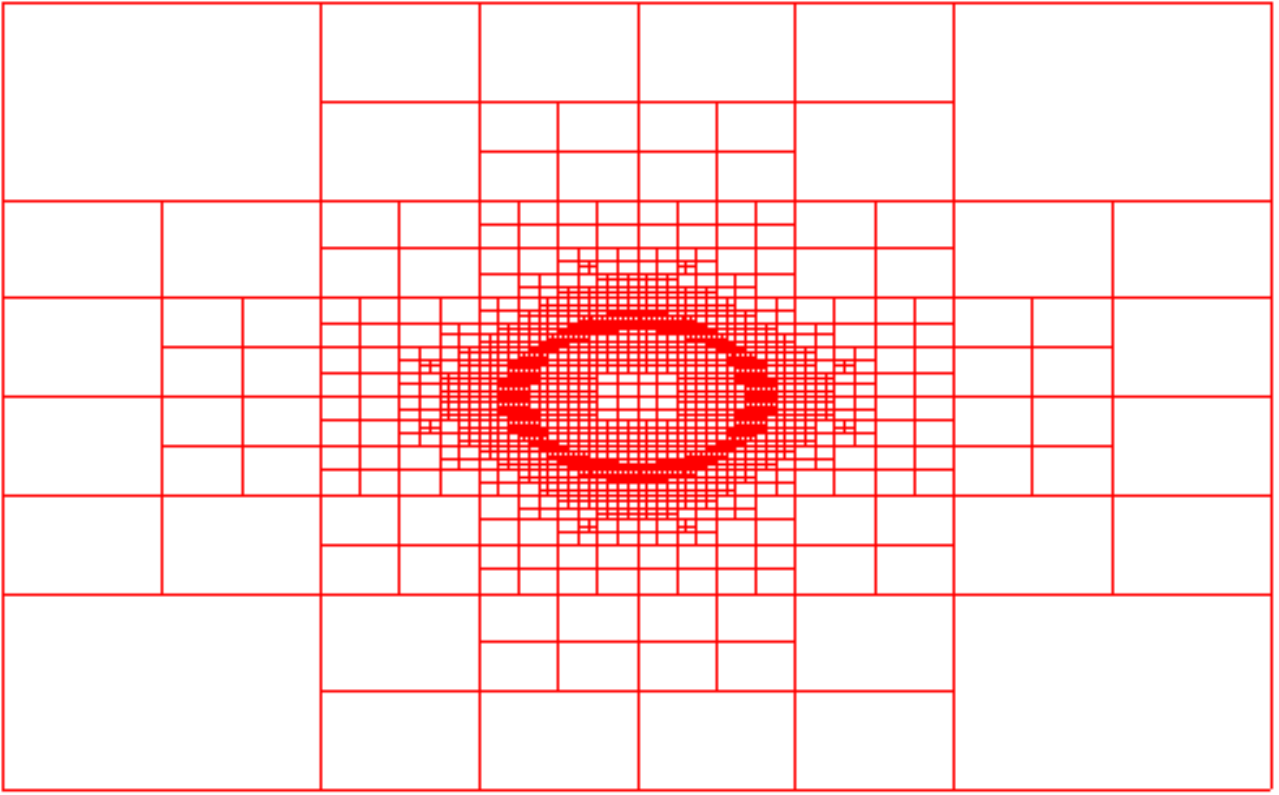}
\caption{Example 3: Initial (left) and final (right) meshes.}
\label{blowupmeshes}
\end{figure}

\noindent The numerical solution at $t=0$ and $t=T$ from the final numerical experiment (${\tt ttol^+} = (0.125)^{11}$) is shown in Figure \ref{blowupprofiles}; the corresponding  meshes are displayed in Figure \ref{blowupmeshes}. The initial mesh has a relatively homogenous distribution of elements which is to be expected since the initial condition is relatively smooth. In the final mesh, elements have been added in the vicinity of the blow-up set and removed elsewhere, notably near the origin. The distribution of elements in the final mesh strongly indicates that the adaptive algorithm is adding and removing elements in an efficient manner.

\begin{figure}
\centering
\includegraphics[scale=0.16]{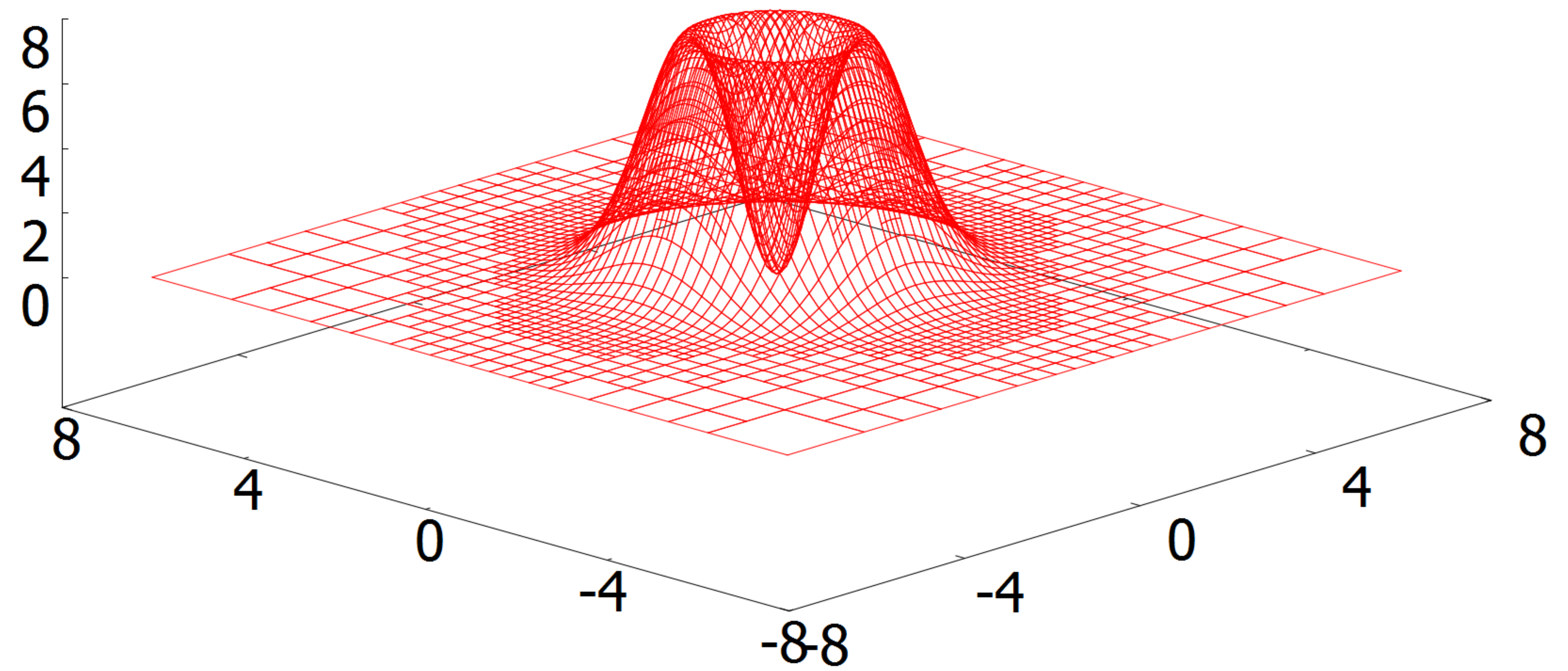} \includegraphics[scale=0.16]{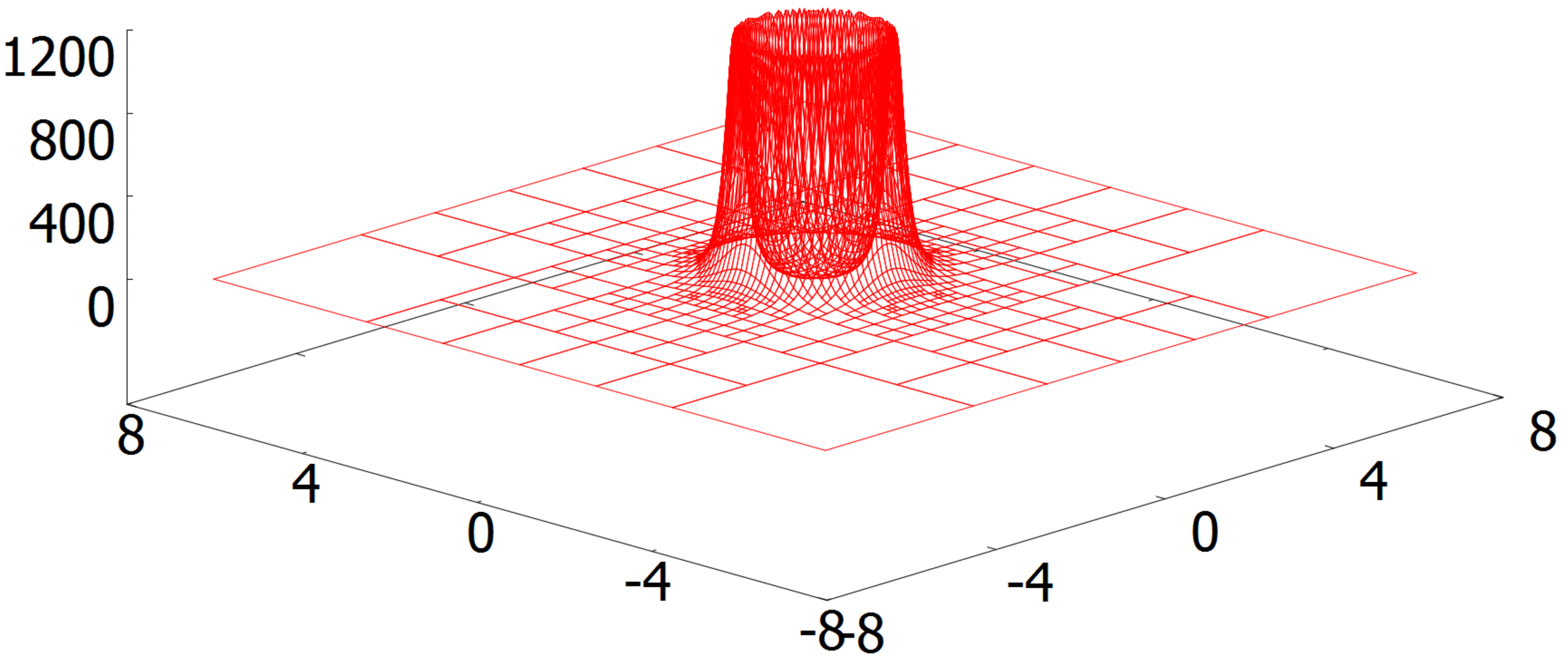}
\caption{Example 3: Initial (left) and final (right) solution profiles.}
\label{blowupprofiles}
\end{figure}
\end{subsection}
\end{section}

\begin{section}{Conclusions}

The error estimator produced performed adequately in both numerical examples but the blow-up time was approached at a much slower rate than expected given the results in Chapter 4. The reason for this lies in the significant differences between the $\delta$ equations \eqref{deltaequation} and \eqref{cond2} which can in turn be traced back to the error equations \eqref{ODEerror2} and \eqref{PDEerrorequation}. For a quadratic nonlinearity, the highest order error term in both of these equations is one power higher than the remainder of the error terms which can all be explicitly bounded by Gronwall's inequality (heuristically, the ODE and PDE error analysis are still `equivalent' at this point). However, in the PDE analysis we have no way of dealing outright with an $L^3$ norm necessitating the use of the Gagliardo-Nirenberg inequality. After application of Gronwall's inequality, a term of the form
\begin{equation}
\begin{aligned}
\exp \Bigg (\int_{t^k}^{t^{k+1}} \! ||\nabla e_c|| \, ds \Bigg),
\end{aligned}
\end{equation}
remains. With no way to estimate this directly through a continuation argument, we are forced to use the Cauchy-Schwarz inequality which destroys a large amount of information (we do something slightly different to get a better $\delta$ equation but the end result is still the same {\bf--} a loss of information caused by the necessity of having norms that are compatible with our continuation argument). If all the norms are the same (as in Chapter 4), there is no loss of information. Therefore, we conjecture that conducting an error analysis for the $L^{\infty}(L^{\infty})$ norm may lead to a recovery of the rates seen in Chapter 4.
\end{section}

\chapter{A posteriori error estimation and adaptivity for a class of nonlinear interface problems}

\begin{section}{Model problem}

We shall consider a model problem that is a simplification of the models given in \cite{CGJ13,CGJ14} which were constructed to model the mass transfer of solutes through a semi-permeable membrane. The derivation of the models in \cite{CGJ13,CGJ14}, based upon the works \cite{F08,KK58,P05,T53,Z02}, shall not be restated here because we are primarily interested in the mathematical model.

The computational domain $\boldsymbol{\Omega}$ is subdivided into two subdomains $\Omega_1$ and $\Omega_2$ such that $\boldsymbol{\Omega} = \Omega \cup \Gamma_i$ where $\Omega := \Omega_1 \cup \Omega_2$ and $\Gamma_i = \partial\Omega\backslash \partial\boldsymbol{\Omega}$ is the \emph{interface} between the two subdomains. To simplify things, we assume that the interface is a non-intersecting piecewise linear curve.

For $T>0$, consider the model problem of finding $u:\Omega\times (0,T] \rightarrow \mathbb{R}$ such that
\begin{equation}
\begin{aligned}
\label{interfacestrong}
\frac{\partial u}{\partial t} - \varepsilon \Delta u + {\bf a} \cdot \nabla u + f(u) & = 0 \qquad & & \text {in } \Omega\times(0,T], \\
u(\cdot,0) & = u_0 && \text{in } \Omega.
\end{aligned}
\end{equation}
We need to augment \eqref{interfacestrong} with suitable boundary conditions. To that end, we split the boundary $\partial\boldsymbol{\Omega} = \bar{\Gamma}_D \cup \bar{\Gamma}_N$ where $\Gamma_D$ is the Dirichlet boundary and $\Gamma_N$ is the Neumann boundary. It is assumed that the intersection of the Dirichlet boundary with each of the subdivision boundaries has positive one-dimensional Hausdorff measure. We then impose boundary conditions for all $t \in (0,T]$:
\begin{equation}
\begin{aligned}
\label{interfaceBC}
u & = 0 \qquad & & \text{on } \Gamma_D, \\
\varepsilon \nabla u \cdot {\bf n} & = g \qquad && \text{on } \Gamma_N \cap \partial\boldsymbol{\Omega}_{out}, \\
(\varepsilon \nabla u  - {\bf a}u)\cdot {\bf n} & = g \qquad && \text{on } \Gamma_N \cap \partial\boldsymbol{\Omega}_{in}.
\end{aligned}
\end{equation}
In addition to boundary conditions, we need to augment \eqref{interfacestrong} with interface conditions. To that end, we require that the following equalities are satisfied across the interface for all $t \in (0,T]$:
\begin{equation}
\begin{aligned}
\label{interfacecond}
(\varepsilon \nabla u - {\bf a}u) \cdot {\bf n} |_{\Omega_1} & = \rho (u|_{\Omega_2} - u|_{\Omega_1})-r(w_1u|_{\Omega_1} + w_2u|_{\Omega_2})({\bf a \cdot n})|_{\Omega_1}, \\
(\varepsilon \nabla u - {\bf a}u) \cdot {\bf n} |_{\Omega_2} & = \rho (u|_{\Omega_1} - u|_{\Omega_2})-r(w_1u|_{\Omega_1} + w_2u|_{\Omega_2})({\bf a \cdot n})|_{\Omega_2},
\end{aligned}
\end{equation}
where $\rho>0$ is the permeability coefficient, $r \in [0,1]$ is the friction coefficient and $w_1,w_2$ are weights that satisfy $w_1+w_2 =1$ and $0 \leq w_1,w_2 \leq 1$. 
The weak form of the model then reads as follows: find $u \in L^2 \big(0,T;H^1_D(\Omega)\big) \cap H^1\big(0,T;L^2(\Omega)\big)$ such that for almost every $t \in (0,T]$ we have
\begin{equation}
\begin{aligned}
\label{interfaceweak}
\bigg(\frac{\partial u}{\partial t},v \bigg) + B(u,v) + (f(u),v) = l(v) \qquad \forall v \in H^1_D(\Omega),
\end{aligned}
\end{equation}
with
\begin{equation}
\begin{aligned}
B(u,v) & := \int_{\Omega} \! (\varepsilon\nabla u - {\bf a}u) \cdot \nabla v \, dx - \int_{\Omega} \! \nabla \cdot {\bf a} \, uv \, dx + \int_{\Gamma_N \cap \partial\boldsymbol{\Omega}_{out}} \! {\bf a \cdot n} \, uv \, ds \\
&+\int_{\Gamma_i} \! \rho[u]\cdot [v] \, ds + \int_{\Gamma_i} \! r\{u \}_w[{\bf a}v] \, ds, \\
l(v) & := \int_{\Gamma_N} \! gv \, ds,
\end{aligned}
\end{equation}
where $\{ u \}_w := w_1u|_{\Omega_1} + w_2u|_{\Omega_2}$ is the weighted average of $u$ across $\Gamma_i$. We make the following assumptions upon the data: $g \in L^2(\Gamma_N)$, $u_0 \in H^1_D(\Omega)$, $0 < \varepsilon \leq 1$, ${\bf a} \in \big[W^{1,\infty}(\boldsymbol{\Omega}) \big]^2$ and $f : \mathbb{R} \rightarrow \mathbb{R}$ must satisfy the growth condition
\begin{equation}
\begin{aligned}
\label{growth}
|f(u)-f(v)| & \leq L|u-v|(1+|u|+|v|)^{\mu} \qquad \forall u,v \in \mathbb{R},
\end{aligned}
\end{equation}
for some $L \geq 0$ and $0 \leq \mu \leq 2$. We finish the section by introducing a coercivity result for the bilinear form $B$.
\begin{theorem}
\label{interfacecoercive}
Let $c_*$ denote the constant in the trace inequality (Theorem \ref{scaledtrace}) and define $\displaystyle \mathcal{A}_i :=||{\bf a}||_{L^{\infty}(\Gamma_i)}$. For any $v \in H^1_D(\Omega)$, the bilinear form $B$ satisfies 
\begin{equation}
\begin{aligned}
\notag
B(v,v) & \geq \frac{3}{4}|||v|||^2 + \bigg ( \frac{1}{2}\essinf_{\Omega}(-\nabla \cdot {\bf a})-c_* \alpha_{rw}\mathcal{A}_i\big(1+4c_*\alpha_{rw}\mathcal{A}_i\varepsilon^{-1} \big) \bigg ) ||v||^2,
\end{aligned}
\end{equation}
where $\displaystyle \alpha_{rw} :=  \frac{r}{2}|w_1-w_2|+\max \bigg\{\bigg |rw_1-\frac{1}{2} \bigg |,\bigg |rw_2-\frac{1}{2} \bigg | \bigg\}$ and
\begin{equation}
\begin{aligned}
\notag
 |||v||| &:= \left (\varepsilon||\nabla v||^2+ \frac{1}{2}\int_{\Gamma_N} \! |{\bf a \cdot n}|v^2 \, ds+\int_{\Gamma_i} \! \rho \, |[v]|^2 \, ds \right )^{1/2}.
\end{aligned}
\end{equation}
\end{theorem}
\begin{proof}

The multidimensional integration by parts formula implies that
\begin{equation}
\begin{aligned}
\notag
\int_{\Omega} \! ({\bf a}v) \cdot \nabla v \, dx + \frac{1}{2} \int_{\Omega} \! \nabla \cdot {\bf a} \, v^2 \, dx = \frac{1}{2}\int_{\Gamma_N} \! {\bf a \cdot n} \, v^2 \, ds + \int_{\Gamma_i} \! \{v\}[{\bf a}v] \, ds.
\end{aligned}
\end{equation}
Application of this to $B(v,v)$ yields
\begin{equation}
\begin{aligned}
\notag
B(v,v) & \geq |||v|||^2 + \frac{1}{2}\essinf_{\Omega}(-\nabla \cdot {\bf a})||v||^2  + \int_{\Gamma_i} \! ( r\{ v \}_w - \{ v \})[{\bf a}v] \, ds.
\end{aligned}
\end{equation}
Denoting $v$ evaluated on $\Omega_1$ by $v_1$ and $v$ evaluated on $\Omega_2$ by $v_2$ then using Young's inequality we obtain
\begin{equation}
\begin{aligned}
\notag
\bigg | \int_{\Gamma_i} \! ( r\{ v \}_w - \{ v \} )[{\bf a}v] \, ds \bigg | & \leq \bigg | rw_1 - \frac{1}{2} \bigg | \int_{\Gamma_i} \! |{\bf a \cdot n} | v_1^2 \, ds + \bigg | rw_2 - \frac{1}{2} \bigg | \int_{\Gamma_i} \! |{\bf a \cdot n}  | v_2^2 \, ds \\
&+r|w_1-w_2| \int_{\Gamma_i} \! |{\bf a \cdot n} | | v_1  |  | v_2  | \, ds \\
& \leq \alpha_{rw}\mathcal{A}_i\bigg(\int_{\Gamma_i} \! v_1^2 \, ds + \int_{\Gamma_i} \! v_2^2 \, ds \bigg ).
\end{aligned}
\end{equation}
Using the trace inequality with $\displaystyle \delta = \frac{1}{4}c_*^{-1}\alpha^{-1}_{rw}\mathcal{A}_i^{-1}{\varepsilon}$ we obtain 
\begin{equation}
\begin{aligned}
\notag
\bigg | \int_{\Gamma_i} \! ( r\{ v \}_w - \{ v \} )[{\bf a}v] \, ds \bigg | & \leq c_*\alpha_{rw}\mathcal{A}_i\big(\delta||\nabla v||^2 + \big(1+\delta^{-1}\big)||v||^2 \big) \\
& \leq \frac{1}{4}|||v|||^2 +c_* \alpha_{rw}\mathcal{A}_i\big(1+4c_*\alpha_{rw}\mathcal{A}_i\varepsilon^{-1} \big) ||v||^2.
\end{aligned}
\end{equation}
Therefore, combining the results
\begin{equation}
\begin{aligned}
\notag
B(v,v)  & \geq |||v|||^2 + \frac{1}{2}\essinf_{\Omega}(-\nabla \cdot {\bf a})||v||^2  - \bigg |\int_{\Gamma_i} \! \big( r\{ v \}_w - \{ v \} \big )[{\bf a}v] \, ds \bigg | \\ & \geq \frac{3}{4}|||v|||^2+ \bigg ( \frac{1}{2}\essinf_{\Omega}(-\nabla \cdot {\bf a})-c_* \alpha_{rw}\mathcal{A}_i\big(1+4c_*\alpha_{rw}\mathcal{A}_i\varepsilon^{-1} \big) \bigg ) ||v||^2.
\end{aligned}
\end{equation}
This completes the proof.
\end{proof}
\end{section}

\begin{section}{Space-time discretisation}

The discontinuous nature of the model on the interface makes a spatial dG discretisation a natural choice. Thus, with the same reasoning as in the previous chapter, we consider an  IMEX dG discretisation of problem \eqref{interfaceweak}.

To that end, consider a subdivision of $[0,T]$ into time intervals of lengths \mbox{ }\mbox{ } $\tau_1$, ..., $\tau_n$ such that $\displaystyle\sum_{j=1}^n{\tau_j}=T$ for some $n \geq 1$ and set $t^0 := 0$ and $\displaystyle t^k := \sum_{j=1}^{k}\tau_{j}$. Denote an initial triangulation by $\zeta^0$ and associate to each time step $k>0$ a triangulation $\zeta^k$  which is assumed to have been obtained from $\zeta^{k-1}$ by locally refining and coarsening $\zeta^{k-1}$. All of the meshes are assumed to be aligned with the interface in the sense that no part of the interface is contained in the interior of any element and aligned with the boundary in the sense that the points of intersection between the Dirichlet and Neumann boundaries, if they exist, must all be at the vertex of an element. To each mesh $\zeta^{k}$, we assign the finite element space $V_h^k := V_h\big(\zeta^k \big)$ given by~\eqref{eq:FEspace} and we set $f^k:=f\big(u_h^k \big)$ for brevity. For each mesh $\zeta^k$, let $\Gamma_{int}$ denote the union of all interior edges that do not lie on the interface. Finally, for $t \in \big(t^{k-1},t^k \big]$, we let $\Gamma$ denote the union of all edges in the mesh $\zeta^{k-1} \cup \zeta^{k}$ that do not lie on the interface or Neumann boundary. 

The IMEX dG method then reads as follows. Set $u_h^0$ to be a projection of $u_0$ onto $V_h^0$. For $k=1$, ..., $n$, find $u_h^{k} \in V_h^{k}$ such that \begin{equation}\label{interfacedG}
\bigg(\frac{u_h^{k}-u_h^{k-1}}{\tau_{k}},v_h^{k} \bigg)+B \big(u_h^{k},v_h^{k} \big)+K_h \big(u_h^{k},v_h^{k} \big)+\big(f^{k-1},v_h^{k} \big)=l\big(v_h^{k} \big),
\end{equation}
for all $v_h^{k} \in V_h^{k}$ where
\begin{equation}
\begin{aligned}
B \big(u_h^{k},v_h^{k} \big) & := \sum_{K \in \zeta^k} \int_K \! \big(\varepsilon \nabla u_h^{k} - {\bf a}u_h^{k} \big) \cdot \nabla v_h^{k} \, dx - \sum_{K \in \zeta^k} \int_K \! \nabla \cdot {\bf a} \, u_h^{k}v_h^{k} \, dx \\
&+\sum_{E \subset \Gamma_D \cup \Gamma_{int}} \int_{E}  \! \frac{\gamma\varepsilon}{h_E} \big[u_h^{k} \big] \cdot \big[v_h^{k} \big] \, ds+\sum_{K \in \zeta^k} \int_{\partial K_{out} \setminus \Gamma_i} \! u_h^{k}\big[{\bf a}v_h^{k} \big] \, ds \\ 
&+ \int_{\Gamma_i} \! \rho \big[u_h^{k} \big] \cdot \big[v_h^{k}\big] \, ds + \int_{\Gamma_i} \! r\big\{ u_h^{k} \big\}_w \big[{\bf a}v_h^{k}\big] \, ds,\\
K_h\big(u_h^{k},v_h^{k} \big) & := - \sum_{E \subset \Gamma_D \cup \Gamma_{int}} \int_{E} \! \big\{ \varepsilon \nabla u_h^{k} \big\} \cdot \big[v_h^{k}\big] + \big\{ \varepsilon \nabla v_h^{k} \big\} \cdot \big[u_h^{k}\big]\, ds.
\end{aligned}
\end{equation}
We shall take $u_h^0$ to be the orthogonal $L^2$ projection of $u_0$ onto $V_h^0$, although other projections onto $V_h^0$ can also be used. 

For $k>0$, the \emph{residual} $R_k$ is defined on the interior and edge of an element $K \in \zeta^k$ as follows
\begin{equation}
\notag
R_{k} := \left \{ 
   \begin{array}{l l}
    -\frac{u_h^k-u_h^{k-1}}{\tau_k}-f^{k-1}+\varepsilon \Delta u_h^k - {\bf a} \cdot \nabla u_h^k & \qquad \text{in } K \\
    0 & \qquad \text{on } \partial K \cap \Gamma_D \\
     g-\varepsilon \nabla u_h^k \cdot {\bf n} & \qquad \text{on } \partial K \cap (\Gamma_N \cap \partial\boldsymbol{\Omega}_{out})\\
     g - \big(\varepsilon \nabla u_h^k-{\bf a}u_h^k \big) \cdot {\bf n} & \qquad \text{on } \partial K \cap (\Gamma_N \cap \partial\boldsymbol{\Omega}_{in}) \\
     \big({\bf a}u_h^k-\varepsilon \nabla u_h^k \big) \cdot {\bf n}_K + \rho \big(u_h^k|_{\Omega_2}-u_h^k|_{\Omega_1} \big) & \qquad \text{on } \partial K \cap \Gamma_i \text{ if } K \subseteq \Omega_1 \\- r\,{\bf a \cdot n}_K\big\{u_h^k \big\}_w  \\
     \big({\bf a}u_h^k-\varepsilon \nabla u_h^k \big) \cdot {\bf n}_K + \rho \big(u_h^k|_{\Omega_1}-u_h^k|_{\Omega_2} \big) & \qquad \text{on } \partial K \cap \Gamma_i  \text{ if } K \subseteq \Omega_2 \\- r \, {\bf a \cdot n}_K \big\{u_h^k \big\}_w \\
     -\frac{\varepsilon}{2}\big[\nabla u_h^k \big] &\qquad \text{on } \partial K \cap \Gamma_{int}
      
   \end{array}
\right.
\end{equation}

\end{section}

\begin{section}{An a posteriori bound for the IMEX dG method}

At each time step $k$, we decompose the dG solution $u_h^k$ into a conforming part $u_{h,c}^k \in H^1_D(\Omega) \cap V_h^k$ and a non-conforming part $u_{h,d}^k \in V_h^k$ such that $u_h^k = u_{h,c}^k + u_{h,d}^k$. Given $t \in \big(t^{k-1},t^{k} \big]$, we define $u_h(t)$ to be the linear interpolant with respect to $t$ of the values $u_h^{k-1}$ and $u_h^{k}$, viz.,
\begin{equation}
\notag
u_h(t):=l_{k-1}(t)u_h^{k-1}+l_{k}(t)u_h^{k}.
\end{equation}
 We define $u_{h,c}(t)$ and $u_{h,d}(t)$ analogously. We can then decompose the error \mbox{ } $e:=u-u_h=e_c-u_{h,d}$ where $e_c:=u-u_{h,c}$. 

\begin{lemma}
\label{interfacelemma}
Given $t \in \big(t^{k-1},t^{k} \big]$ then for any $v \in H^1_D(\Omega)$ we have
\begin{equation}
\begin{aligned}
\notag
\bigg(\frac{\partial e}{\partial t},v \bigg)+B(e,v)+(f(u)-f(u_h),v) = l(v) - \bigg(\frac{\partial u_h}{\partial t}+f(u_h),v \bigg) - B(u_h,v).
\end{aligned}
\end{equation}
\end{lemma}
\begin{proof}
Follows from \eqref{interfaceweak}.
\end{proof}
\noindent From Lemma \ref{interfacelemma}, it follows that
\begin{equation}
\begin{aligned}
& \bigg(\frac{\partial e}{\partial t},v \bigg)+B\big(e,v \big)+\big(f(u)-f(u_h),v \big) =  \big(f^{k-1}-f(u_h),v \big)+B \big(u_h^{k},v \big) \\ &- B \big(u_h,v \big) +l \big(v \big) - \bigg(\frac{\partial u_h}{\partial t}+f^{k-1},v \bigg) - B\big(u_h^{k},v \big).
\end{aligned}
\end{equation}
Finally, we use \eqref{interfacedG} to conclude that for any $v_h^k \in V_h^k$:
\begin{equation}
\begin{aligned}
\label{derpzor}
& \bigg(\frac{\partial e}{\partial t},v \bigg)+B \big(e,v \big)+\big(f (u )-f (u_h ),v \big)  =  \big(f^{k-1}-f(u_h),v \big)+B\big(u_h^{k},v \big)\\& - B \big(u_h,v \big)+ l \big(v-v_h^{k} \big) - \bigg(\frac{\partial u_h}{\partial t}+f^{k-1},v-v_h^{k} \bigg) - B \big(u_h^{k},v-v_h^{k} \big)\\&+K_h \big(u_h^{k},v_h^{k} \big).
\end{aligned}
\end{equation}

We are now ready to state our a posteriori estimator. Due to the nature of the error bound to be presented, it is easier to separate the estimator into two parts. As in previous chapters, a subscript $S$ denotes parts of the estimator related to estimating space while a subscript $T$ denotes parts of the estimator related to estimating time. In this way, for $t \in \big(t^{k-1},t^{k} \big]$, $\eta_A$ is given by
\begin{equation}
\begin{aligned}
\notag
\eta_A & :=  \eta_{S_1,k}+ \eta_{S_2,k}+ \eta_{S_3,k}+ \eta_{S_4,k}+\eta_{T_1,k}+\eta_{T_2,k}+\eta_{T_3,k},
\end{aligned}
\end{equation}
where
\begin{equation}
\begin{aligned}
\notag
\eta_{S_1,k} & := \left( \sum_{K \in \zeta^k} \frac{h_K^2}{\varepsilon}||R_k||^2_{L^2(K)} + \sum_{K \in \zeta^k}\sum_{E \subset \partial K} \frac{h_E}{\varepsilon}||R_k||^2_{L^2(E)} \right. \\
&\left.+ \sum_{E \subset \Gamma_D \cup \Gamma_{int}} \frac{\gamma\varepsilon}{h_E}\big|\big|\big[u_h^k\big]\big|\big|^2_{L^2(E)} + \sum_{E \subset \Gamma_{int}} \frac{h_E}{\varepsilon}\big|\big|\big[{\bf a}u_h^k\big]\big|\big|^2_{L^2(E)} \right)^{1/2}, \\
\eta_{S_2,k} &:= \left( \sum_{E \subset \Gamma} \frac{\gamma\varepsilon}{h_E}||[u_h]||^2_{L^2(E)} + \sum_{E \subset \Gamma} \frac{h_E}{\varepsilon}||[{\bf a}u_h]||^2_{L^2(E)} \right)^{1/2}, \\
\eta_{S_3,k} &:=\left(\sum_{E \subset \Gamma_N \cap \partial\boldsymbol{\Omega}_{out}} \sum_{K \subset \tilde{E}} \sum_{E \subset \tilde{K}_E \cap \Gamma}\mathcal{A}_N ||[u_h]||^2_{L^2(E)} \right)^{1/2},\\
\eta_{S_4,k} &:=\left(\sum_{E \subset \Gamma_i} \sum_{K \subset \tilde{E}} \sum_{E \subset  \tilde{K}_E \cap \Gamma} \alpha_{\rho} ||[u_h]||^2_{L^2(E)}\right)^{1/2},\\ 
\eta_{T_1,k} &:=\left|\left|\varepsilon^{1/2} \nabla\big (u_h^{k}-u_h\big) - {\bf a}\varepsilon^{-1/2}\big(u_h^{k}-u_h\big)\right|\right|, \\
\eta_{T_2,k} &:= \left ( \int_{\Gamma_N \cap \partial \boldsymbol{\Omega}_{out}} \! |{\bf a \cdot n}|\big|u_h^{k}-u_h \big|^2 \, ds \right)^{1/2}, \\
\eta_{T_3,k} &:= \left|\left|\rho^{1/2}\big|\big[u_h^{k}-u_h \big]\big|+r\rho^{-1/2}|{\bf a}|\big|\big\{u_h^{k}-u_h\big\}_w \big| \, \right|\right|_{L^2(\Gamma_i)},
\end{aligned}
\end{equation}
with $\alpha_{\rho} := 2 \rho + 2r^2\rho^{-1}\max\big\{w^2_1,w^2_2 \big\} \mathcal{A}_i^2$ and $\mathcal{A}_N := ||{\bf a}||_{L^{\infty}(\Gamma_N \cap \partial\boldsymbol{\Omega}_{out})}$. Similarly, for $t \in \big(t^{k-1},t^{k} \big]$, $\eta_B$ is given by
\begin{equation}
\begin{aligned}
\notag
\eta_B & := \eta_{S_5,k} + \eta_{S_6,k}+\eta_{T_4,k},
\end{aligned}
\end{equation}
where
\begin{equation}
\begin{aligned}
\notag
\eta_{S_5,k}  & := \left (\sum_{K \in \zeta^{k-1} \cup \zeta^{k}} \sum_{E \subset \tilde{K}_E} \sigma^2_K h_E ||[u_h]||^2_{L^2(E)} \right )^{1/2}, \\
\eta_{S_6,k} &:=  \left ( \sum_{E \subset \Gamma} h_E \left| \left|  \left[\frac{u_h^{k}-u_h^{k-1}}{\tau_{k}} \right]\right| \right|^2_{L^2(E)} \right)^{1/2}, \\
\eta_{T_4,k} &:= \left|\left|f^{k-1} - f \big(u_h \big)-\nabla \cdot {\bf a}\big (u_h^{k} - u_h \big)\right|\right|,
\end{aligned}
\end{equation}
with
\begin{equation}
\begin{aligned}
\notag
\sigma_K := \max\big\{L,2^{\mu-1}L\big\}\big(1+2||u_h||_{L^{\infty}(K)}+||[u_h]||_{L^{\infty}(\tilde{K}_E \cap \Gamma)} \big)^{\mu} + ||\nabla \cdot {\bf a}||_{L^{\infty}(K)}.
\end{aligned}
\end{equation}
The first three terms on the right-hand side of \eqref{derpzor} approximate the temporal part of the error. To begin bounding, we decompose these terms, viz.,
\begin{equation}
\begin{aligned}
\big(f^{k-1}-f (u_h),v \big)+B \big(u_h^{k},v \big)-B \big(u_h,v \big) = T_1+ T_2 + T_3 + T_4,
\end{aligned}
\end{equation}
where
\begin{equation}
\begin{aligned}
\notag
T_1 &:= \sum_{K \in \zeta^{k-1} \cup \zeta^k} \int_{K} \!  \big (\varepsilon \nabla \big (u_h^{k} - u_h \big ) - {\bf a} \big (u_h^{k} - u_h \big) \big ) \cdot \nabla v \, dx, \\
T_2 &:= \sum_{K \in \zeta^{k-1} \cup \zeta^k} \int_{K} \! \big (f^{k-1}- f \big(u_h \big)-\nabla \cdot {\bf a}\big (u_h^{k} - u_h \big)\big)v \, dx, \\
T_3 &:= \int_{\Gamma_N \cap \partial \boldsymbol{\Omega}_{out}} \! {\bf a \cdot n}\big (u_h^{k}-u_h \big)v \, ds, \\
T_4 &:= \int_{\Gamma_i} \! \rho\big [u_h^{k} - u_h \big ] \cdot \big[v \big] \, ds + \int_{\Gamma_i} \! r\big\{u_h^{k}-u_h \big\}_w \big[{\bf a}v \big] \, ds.
\end{aligned}
\end{equation}
Bounding $T_1$ requires a simple application of the Cauchy-Schwarz inequality:
\begin{equation}
\begin{aligned}
|T_1| &  \leq \eta_{T_1,k}|||v|||.
\end{aligned}
\end{equation}
$T_2$ is also bounded by the Cauchy-Schwarz inequality, viz., 
\begin{equation}
\begin{aligned}
|T_2| & \leq \eta_{T_4,k}||v||.
\end{aligned}
\end{equation}
$T_3$ is bounded by the Cauchy-Schwarz inequality as follows
\begin{equation}
\begin{aligned}
|T_3| & \leq \eta_{T_2,k}\left(\int_{\Gamma_N \cap \partial \boldsymbol{\Omega}_{out} } \! |{\bf a \cdot n}|v^2 \, ds \right)^{1/2} \\
& \lesssim \eta_{T_2,k}|||v|||.
\end{aligned}
\end{equation}
Finally, $T_4$ is bounded by the Cauchy-Schwarz inequality, viz., 
\begin{equation}
\begin{aligned}
|T_4| & \leq \eta_{T_3,k}|||v|||.
\end{aligned}
\end{equation}
The remainder of the terms on the right-hand side of \eqref{derpzor} give rise to parts of the space estimator. We start by noting that, through application of the multidimensional integration by parts formula on each element, we have
\begin{equation}
\begin{aligned}
 l \big(v-v_h^{k} \big) - \bigg(\frac{\partial u_h}{\partial t}+f^{k-1},v-v_h^{k} \bigg)- B\big(u_h^{k},v-v_h^{k} \big) = T_5 + T_6,
\end{aligned}
\end{equation}
where
\begin{equation}
\begin{aligned}
\notag
T_5 &:= \sum_{K \in \zeta^{k}} \int_K \! R_k \big(v-v_h^{k} \big) \, dx + \sum_{K \in \zeta^{k}}\sum_{E \subset \partial K} \int_E \! R_k \big(v-v_h^{k} \big) \, dx,\\
T_6 &:= \sum_{E \subset \Gamma_{int}} \int_E \! \big[{\bf a}u_h^{k} \big]\big(v-v_h^{k} \big)  \, ds.
\end{aligned}
\end{equation}
We set $v_h^k \in H^1_D(\Omega) \cap V_h^k$ to be the finite element interpolant from Theorem \ref{hodor63}. Application of the Cauchy-Schwarz inequality together with the interpolation estimates from Theorem \ref{hodor63} yields
\begin{equation}
\begin{aligned}
|T_5| & \leq \left ( \sum_{K \in \zeta^{k}}\frac{h_K^2}{\varepsilon}|| R_k||^2_{L^2(K)} \right)^{1/2} \left ( \sum_{K \in \zeta^{k}}\frac{\varepsilon}{h^2_K}\big|\big|v-v_h^k\big|\big|^2_{L^2(K)} \right)^{1/2} \\&+ \left(\sum_{K \in \zeta^{k}}\sum_{E \subset \partial K} \frac{h_E}{\varepsilon}||R_k ||_{L^2(E)}^2\right)^{1/2}\left(\sum_{K \in \zeta^{k}}\sum_{E \subset \partial K} \frac{\varepsilon}{h_E}\big|\big|v-v_h^k \big|\big|_{L^2(E)}^2\right)^{1/2} \\
& \lesssim \eta_{S_1,k} \left[\left ( \sum_{K \in \zeta^{k}}\varepsilon||\nabla v||^2_{L^2(\tilde{K})} \right)^{1/2} + \left(\sum_{K \in \zeta^{k}}\sum_{E \subset \partial K} \varepsilon||\nabla v ||_{L^2(\tilde{E})}^2\right)^{1/2}\right ] \\
& \lesssim \eta_{S_1,k}|||v|||.
\end{aligned}
\end{equation}
$T_6$ is also bounded through the interpolation estimates of Theorem \ref{hodor63} together with the Cauchy-Schwarz inequality, viz.,
\begin{equation}
\begin{aligned}
|T_6|  & \leq \left ( \sum_{E \subset \Gamma_{int}} \frac{h_E}{\varepsilon}\big|\big|\big[{\bf a}u_h^{k} \big]\big|\big|_{L^2(E)}^2\right )^{1/2} \left ( \sum_{E \subset \Gamma_{int}} \frac{\varepsilon}{h_E}\big|\big|v-v_h^{k}\big|\big|_{L^2(E)}^2 \right )^{1/2}  \\
& \lesssim \eta_{S_1,k}\left ( \sum_{E \subset \Gamma_{int}} {\varepsilon}||\nabla{v}||_{L^2(\tilde{E})}^2 \right )^{1/2} \\
& \lesssim \eta_{S_1,k}|||v|||.
\end{aligned}
\end{equation}
Finally, $K_h\big(u_h^{k},v_h^k\big)$ is bounded through the Cauchy-Schwarz inequality and the inverse estimate along with the shape-regularity of the mesh as follows
\begin{equation}
\begin{aligned}
&\big|K_h\big(u_h^k,v_h^k\big)\big| \leq \sum_{E \subset \Gamma_D \cup \Gamma_{int}} \int_{E} \! \ \varepsilon \big|\nabla v_h^{k} \big| \big|\big[u_h^{k}\big]\big| \, ds \\
& \leq \left (\sum_{E \subset \Gamma_D \cup \Gamma_{int}} \frac{\gamma\varepsilon}{h_E}\big|\big|\big[u^k_h\big]\big|\big|^2_{L^2(E)} \right )^{1/2}\left (\sum_{E \subset \Gamma_D \cup \Gamma_{int}}\varepsilon h_E||\nabla v||^2_{L^2(E)} \right )^{1/2} \\
& \lesssim \eta_{S_1,k}\left (\sum_{E \subset \Gamma_D \cup \Gamma_{int}} \varepsilon ||\nabla v||^2_{L^2(\tilde{E})} \right )^{1/2} \\
& \lesssim \eta_{S_1,k}|||v|||.
\end{aligned}
\end{equation}
Putting together all these results we obtain
\begin{equation}
\begin{aligned}
\bigg(\frac{\partial e_c}{\partial t},v \bigg)+B(e_c,v ) & \lesssim |(f(u)-f (u_h),v )| + \bigg|\bigg ( \frac{\partial u_{h,d}}{\partial t},v \bigg)\bigg| + |B(u_{h,d},v)| \\&+ (\eta_{S_1,k}+\eta_{T_1,k}+\eta_{T_2,k}+\eta_{T_3,k})|||v|||+\eta_{T_4,k}||v||.
\end{aligned}
\end{equation}
To bound $B(u_{h,d},v)$, we note that
\begin{equation}
\begin{aligned}
B(u_{h,d},v) &=  T_7 + T_8 + T_9+T_{10},
\end{aligned}
\end{equation}
where
\begin{equation}
\begin{aligned}
\notag
T_7 & :=  \sum_{K \in \zeta^{k-1} \cup \zeta^k } \int_K \! (\varepsilon \nabla u_{h,d} - {\bf a}u_{h,d}) \cdot \nabla v \, dx, \\
T_8 & :=  - \sum_{K \in \zeta^{k-1} \cup \zeta^k } \int_K \! \nabla \cdot {\bf a} \, u_{h,d}v \, dx,\\
T_9 & := \sum_{K \in \zeta^{k-1} \cup \zeta^k } \int_{\partial K_{out} \cap \Gamma_N} \! {\bf a \cdot n} \, u_{h,d}v \, ds,\\ 
T_{10} & :=  \int_{\Gamma_i} \! \rho \big[u_{h,d} \big] \cdot \big[v\big] \, ds + \int_{\Gamma_i} \! r\big\{ u_{h,d} \big\}_w \big[{\bf a}v\big] \, ds.
\end{aligned}
\end{equation}
To bound $T_7$, we use the Cauchy-Schwarz inequality along with the estimates from Theorem \ref{ncbounds} to conclude that
\begin{equation}
\begin{aligned}
\label{stfu8}
|T_{7}| \lesssim \eta_{S_2,k}|||v|||,
\end{aligned}
\end{equation}
while $T_8$ is bounded through H\"{o}lder's inequality, the Cauchy-Schwarz inequality and Theorem \ref{ncbounds}, viz.,
\begin{equation}
\begin{aligned}
|T_8| & \leq \left( \sum_{K \in \zeta^{k-1} \cup \zeta^k }||\nabla \cdot {\bf a}||^2_{L^{\infty}(K)}||u_{h,d}||^2_{L^2(K)} \right)^{1/2}||v|| \\
& \lesssim \eta_{S_5,k}||v||.
\end{aligned}
\end{equation}
$T_9$ is bounded using H\"{o}lder's inequality, the Cauchy-Schwarz inequality, the trace inequality (with $\delta = h_K$) and the bounds from Theorem \ref{ncbounds} along the shape-regularity of the mesh as follows
\begin{equation}
\begin{aligned}
\label{stfu9}
& |T_9| \lesssim \left(\sum_{K \in \zeta^{k-1} \cup \zeta^k } \int_{\partial K_{out} \cap \Gamma_N} \! |{\bf a \cdot n}||u_{h,d}|^2 \, ds \right)^{1/2}|||v||| \\
& \lesssim \left(\sum_{E \subset \Gamma_N \cap \partial \boldsymbol{\Omega}_{out}} \mathcal{A}_N||u_{h,d}||_{L^2(E)}^2 \right)^{1/2}|||v||| \\
& \lesssim \left(\sum_{E \subset \Gamma_N \cap \partial \boldsymbol{\Omega}_{out}}\sum_{K \subset \tilde{E}} \mathcal{A}_N \Big( h_K||\nabla u_{h,d}||_{L^2(K)}^2 + h^{-1}_K||u_{h,d}||_{L^2(K)}^2 \Big)\right)^{1/2}|||v||| \\
& \lesssim \eta_{S_3,k}|||v|||.
\end{aligned}
\end{equation}
Finally, $T_{10}$ is bounded using H\"{o}lder's inequality, Young's inequality, the \mbox{ }\mbox{ }\mbox{ } Cauchy-Schwarz inequality, the trace inequality and the bounds from Theorem \ref{ncbounds} along the shape-regularity of the mesh, viz.,
\begin{equation}
\begin{aligned}
\label{stfu10}
|T_{10}| & \leq \left(\int_{\Gamma_i} \! \rho |[u_{h,d}]|^2 \, ds + \int_{\Gamma_i} \! r^2\rho^{-1}\mathcal{A}^2_i|\{ u_{h,d}\}_w|^2 \, ds \right)^{1/2}|||v||| \\
& \leq \left(\int_{\Gamma_i} \!\alpha_{\rho}\big(u^2_{h,d} |_{\Omega_1}+ u_{h,d}^2 |_{\Omega_2} \big) \, ds \right)^{1/2}|||v||| \\
& \lesssim \left(\sum_{E \subset \Gamma_i}\sum_{K \subset \tilde{E}}\alpha_{\rho}\Big( h_K||\nabla u_{h,d}||_{L^2(K)}^2 +h^{-1}_K||u_{h,d}||_{L^2(K)}^2\Big) \right)^{1/2}|||v||| \\
& \lesssim \eta_{S_4,k}|||v|||.
\end{aligned}
\end{equation}
To bound the remaining nonconforming term, we use the Cauchy-Schwarz inequality and the bounds from Theorem \ref{ncbounds} as follows
\begin{equation}
\begin{aligned}
\bigg(\frac{\partial u_{h,d}}{\partial t},v \bigg) \leq \bigg | \bigg |\frac{\partial u_{h,d}}{\partial t} \bigg | \bigg | ||v|| \lesssim \eta_{S_6,k}||v||.
\end{aligned}
\end{equation}
Combining these results, using the definition of our estimators and setting $v = e_c$ we obtain
\begin{equation}
\begin{aligned}
\label{stfu}
\frac{1}{2}\frac{d}{dt} \big(||e_c||^2 \big)+B(e_c,e_c) & \lesssim |(f(u)-f (u_h),e_c)|+\eta_A|||e_c|||+\eta_B||e_c||.
\end{aligned}
\end{equation}
We must now bound the nonlinear term. The growth condition \eqref{growth} and the triangle inequality imply that
\begin{equation}
\begin{aligned}
|(f(u)-f(u_h),e_c)| & \leq L \int_{\Omega} \! |e||e_c|(1+|u|+|u_h|)^{\mu} \, dx \\
& \leq L \int_{\Omega} \! (|e_c|+|u_{h,d}|)|e_c|(1+2|u_h|+|u_{h,d}|+|e_c|)^{\mu} \, dx. \\
\end{aligned}
\end{equation}
Thus, using the power mean inequality we have
\begin{equation}
\begin{aligned}
|(f(u)-f(u_h),e_c)| & \leq T_{11}+T_{12}+T_{13}+T_{14}, \\
\end{aligned}
\end{equation}
where
\begin{equation}
\begin{aligned}
\notag
T_{11} &:= \max\big\{L,2^{\mu-1}L \big\} \sum_{K \in \zeta^{k-1} \cup \zeta^{k}} \int_K \! (1+2|u_h|+|u_{h,d}|)^{\mu}|u_{h,d}||e_c| \, dx, \\
T_{12} &:= \max\big\{L,2^{\mu-1}L \big\}\int_{\Omega} \! |u_{h,d}||e_c|^{1+\mu} \, dx, \\
T_{13} &:= \max\big\{L,2^{\mu-1}L \big\}\int_{\Omega} \! (1+2|u_h| + |u_{h,d}|)^{\mu}|e_c|^2 \, dx, \\
T_{14} &:= \max\big\{L,2^{\mu-1}L \big\}||e_c||^{2+\mu}_{L^{2+\mu}(\Omega)}.
\end{aligned}
\end{equation}
To bound $T_{11}$, we use the Cauchy-Schwarz inequality, H\"{o}lder's inequality and Theorem \ref{ncbounds} to conclude that
\begin{equation}
\begin{aligned}
T_{11} \lesssim \eta_{S_5,k}||e_c||,
\end{aligned}
\end{equation}
 while $T_{12}$ is bounded through H\"{o}lder's inequality, Theorem \ref{ncbounds} and $\displaystyle L^p$ embeddings if $0 \leq \mu < 1$ or the Gagliardo-Nirenberg inequality if $1 \leq \mu \leq 2$, viz.,
\begin{equation}
\begin{aligned}
T_{12} & \lesssim \max\big\{L,2^{\mu-1}L \big\}||[u_h]||_{L^{\infty}(\Gamma)}||e_c||^{1+\mu} \qquad \qquad && \text{if } 0 \leq \mu < 1,  \\
T_{12} & \lesssim \max\big\{L,2^{\mu-1}L \big\}||[u_h]||_{L^{\infty}(\Gamma)}||e_c||^2||\nabla e_c||^{\mu - 1} && \text{if } 1 \leq \mu \leq 2.
\end{aligned}
\end{equation}
To bound $T_{13}$, we use H\"{o}lder's inequality and Theorem \ref{ncbounds} as follows
\begin{equation}
\begin{aligned}
T_{13} & \leq \max\big\{L,2^{\mu-1}L \big\}\big(1+2||u_h||_{L^{\infty}(\Omega)} + ||u_{h,d}||_{L^{\infty}(\Omega)} \big)^{\mu}||e_c||^2  \\
& \lesssim \max\big\{L,2^{\mu-1}L \big\}\big(1+2||u_h||_{L^{\infty}(\Omega)} + ||[u_h]||_{L^{\infty}(\Gamma)} \big)^{\mu}||e_c||^2.
\end{aligned}
\end{equation}
Finally, the Gagliardo-Nirenberg inequality implies that
\begin{equation}
\begin{aligned}
T_{14} \lesssim \max\big\{L,2^{\mu-1}L \big\}||e_c||^2||\nabla e_c||^{\mu}.
\end{aligned}
\end{equation}
Let $C$ and $K$ denote generic positive constants and define $\alpha_L := \max\big\{2L,2^{\mu}L \big\}$ for brevity. Applying the above bounds to \eqref{stfu} and using the coercivity of the bilinear form $B$ (Theorem \ref{interfacecoercive}) and Young's inequality yields 
\begin{equation}
\begin{aligned}
\frac{d}{dt} \big(||e_c||^2 \big)+|||e_c|||^2 & \leq C\eta^2_A+C\eta_B||e_c|| +\sigma_1||e_c||^{1+\mu}+ ( \sigma_{\Omega}+\sigma_2||\nabla e_c||^{\mu-1}\\&+K\alpha_L||\nabla e_c||^{\mu}    )||e_c||^2,
\end{aligned}
\end{equation}
where
\begin{equation}
\begin{aligned}
\notag
\sigma_1 & := \left \{ 
   \begin{array}{l l}
   K\alpha_L||[u_h]||_{L^{\infty}(\Gamma)} & \qquad \text{if } 0 \leq \mu < 1  \\
    0 & \qquad \text{if } 1 \leq \mu \leq 2
   \end{array}, 
\right. \\
\sigma_2 & := \left \{ 
   \begin{array}{l l}
   0 & \qquad \text{if } 0 \leq \mu < 1  \\
    K\alpha_L||[u_h]||_{L^{\infty}(\Gamma)} & \qquad \text{if } 1 \leq \mu \leq 2
   \end{array}, 
\right. \\
\sigma_\Omega &:=  \alpha_L\big(1+2||u_h||_{L^{\infty}(\Omega)} + K||[u_h]||_{L^{\infty}(\Gamma)} \big)^{\mu}-\essinf_{\Omega}(-\nabla \cdot {\bf a})
\\&+2c_* \alpha_{rw}\mathcal{A}_i\big(1+4c_*\alpha_{rw}\mathcal{A}_i\varepsilon^{-1} \big).
\end{aligned}
\end{equation}
Another application of Young's inequality yields
\begin{equation}
\begin{aligned}
\label{stfu2}
\frac{d}{dt} \big(||e_c||^2 \big)+|||e_c|||^2 & \leq C\big(\eta^2_A+T\eta^2_B \big) +\sigma_1||e_c||^{1+\mu}+ \bigg(\frac{1}{2T}+ \sigma_{\Omega}\\&+\sigma_2||\nabla e_c||^{\mu-1}+K\alpha_L||\nabla e_c||^{\mu}    \bigg)||e_c||^2.
\end{aligned}
\end{equation}
Application of Gronwall's inequality to \eqref{stfu2} together with the bound
\begin{equation}
\label{stfu11}
||e_c(0)||^2 \lesssim ||e(0)||^2 + \esssup_{0 \leq t \leq T}||u_{h,d}||^2 \lesssim ||e(0)||^2 + \esssup_{0 \leq t \leq T} \sum_{E \subset \Gamma} h_E ||[u_h]||^2_{L^2(E)},
\end{equation}
implies that for any $t \in [0,T]$ we have
\begin{equation}
\begin{aligned}
\label{stfu3}
||e_c(t)||^2_* \leq CH(t) G \bigg ( \phi + \int_0^t \! \sigma_1||e_c||^{1+\mu} \, ds \bigg),
\end{aligned}
\end{equation}
where $||\cdot||_*$ is the $L^2 (H^1) + L^{\infty}(L^2)$ type norm 

\begin{equation}
\begin{aligned}
\notag
||u(t)||_* := \left(||u||^2_{L^{\infty}(0,t;L^2(\Omega))}+\int_0^t \! |||u|||^2 \, ds\right)^{1/2},
\end{aligned}
\end{equation}
and 
\begin{equation}
\begin{aligned}
\notag
H(t) & := \exp \bigg (\int_0^t \! \sigma_2||\nabla e_c||^{\mu-1} \, ds + K\alpha_L \int_0^t \!||\nabla e_c||^{\mu} \, ds \bigg ),  \\
G & := \exp \bigg(\int_0^T \! \sigma_{\Omega} \, ds \bigg),  \\
\phi &: = ||e(0)||^2 + \int_0^T \! \eta^2_A \, ds + T \int_0^T \! \eta^2_B \, ds + \esssup_{0 \leq t \leq T} \sum_{E \subset \Gamma} h_E ||[u_h]||^2_{L^2(E)}.
\end{aligned}
\end{equation}
In order to construct a practical error estimator from \eqref{stfu3}, we employ a continuation argument. To that end, we define the set
\begin{equation}
\begin{aligned}
\notag
I & := \big \{t \in [0,T] \mbox{ } \big | \mbox{ } ||e_c(t)||^2_* \leq \delta G \phi \big \},
\end{aligned}
\end{equation}
where $\delta > C$ is a parameter to be chosen. Clearly $I$ is bounded; furthermore, we know $I$ is non-empty because $0 \in I$. Let $t^*$ denote the maximal value of $t$ in $I$ and assume that $t^* < T$. We proceed as in previous chapters by bounding the remaining error terms in \eqref{stfu3}. Firstly, H\"{o}lder's inequality implies that
\begin{equation}
\begin{aligned}
\int_0^{t^*} \! \sigma_1 ||e_c||^{1+\mu} \, ds \leq ||e_c(t^*)||_*^{1+\mu}\int_0^{T} \! \sigma_1 \, ds \leq (dG\phi ) ^{\frac{1+\mu}{2}}\int_0^{T} \! \sigma_1 \, ds,
\end{aligned}
\end{equation}
while through the Cauchy-Schwarz inequality and $L^p$ embeddings we obtain that
\begin{equation}
\begin{aligned}
\int_0^{t^*} \! \sigma_2 ||\nabla e_c||^{\mu-1} \, ds & \leq \bigg ( \int_0^{T} \! \sigma^2_2 \, ds \bigg)^{1/2}\bigg(\int_0^{t^*} \! ||\nabla e_c||^{2\mu-2} \, ds \bigg)^{1/2} \\
& \leq T^{1-\frac{\mu}{2}} \bigg ( \int_0^{T} \! \sigma^2_2 \, ds \bigg)^{1/2}\bigg(\int_0^{t^*} \! ||\nabla e_c||^2 \, ds \bigg)^{\frac{\mu-1}{2}} \\
& \leq T^{1-\frac{\mu}{2}} \bigg ( \int_0^{T} \! \sigma^2_2 \, ds \bigg)^{1/2}(\delta G \phi)^{\frac{\mu-1}{2}}.
\end{aligned}
\end{equation}
Finally, we use the properties of $L^p$ embeddings to conclude that
\begin{equation}
\begin{aligned}
\int_0^{t^*} \! ||\nabla e_c||^{\mu} \, ds \leq T^{1 - \frac{\mu}{2}} \bigg(\int_0^{t^*} \! ||\nabla e_c||^2 \, ds \bigg)^{\mu/2} \leq T^{1 - \frac{\mu}{2}} (\delta G \phi)^{\mu/2}.
\end{aligned}
\end{equation}
Putting these results into \eqref{stfu3}, we conclude that
\begin{equation}
\begin{aligned}
\label{stfu4}
||e_c(t^*)||^2_* \leq C\psi G \bigg ( \phi + (dG\phi ) ^{\frac{1+\mu}{2}}\int_0^{T} \! \sigma_1 \, ds\bigg),
\end{aligned}
\end{equation}
where
\begin{equation}
\begin{aligned}
\psi:=\exp\left(T^{1-\frac{\mu}{2}} \bigg ( \int_0^{T} \! \sigma^2_2 \, ds \bigg)^{1/2}(\delta G \phi)^{\frac{\mu-1}{2}} + K \alpha_L T^{1 - \frac{\mu}{2}} (\delta G \phi)^{\mu/2}   \right).
\end{aligned}
\end{equation}
Now, suppose that the upper bound in \eqref{stfu4} is strictly less than the upper bound of the set $I$, viz.,
\begin{equation}
\begin{aligned}
C\psi G \bigg ( \phi + (dG\phi ) ^{\frac{1+\mu}{2}}\int_0^{T} \! \sigma_1 \, ds\bigg) < \delta G \phi,
\end{aligned}
\end{equation}
or equivalently,
\begin{equation}
\begin{aligned}
\label{stfu5}
C\psi  \bigg ( \phi + (dG\phi ) ^{\frac{1+\mu}{2}}\int_0^{T} \! \sigma_1 \, ds\bigg) < \delta \phi,
\end{aligned}
\end{equation}
then $t^*$ cannot be the maximal value of $t$ in $I$ because we just showed that $\displaystyle ||e_c(t^*)||^2_*$ satisfies a bound strictly less than that assumed in the set $I$ {\bf--} a contradiction. Therefore, providing \eqref{stfu5} is satisfied, $I = [0,T]$ and we have our desired error bound once we select $\delta$. Taking the limit, we can select $\delta$ to be the minimiser of
\begin{equation}
\begin{aligned}
\label{stfu6}
C\psi \bigg ( \phi + (dG\phi ) ^{\frac{1+\mu}{2}}\int_0^{T} \! \sigma_1 \, ds\bigg) - \delta \phi = 0, \qquad \delta > C.
\end{aligned}
\end{equation}
In order to state the final theorem, we need to extend the energy norm to include functions in $V_h$. To that end, for $t \in \big(t^{k-1}, t^{k} \big]$, we (re)define
\begin{equation}
\begin{aligned}
\notag
 |||v||| &:= \left (\sum_{K \in \zeta^{k-1} \cup \zeta^k} \varepsilon||\nabla v||_{L^2(K)}^2+ \frac{1}{2}\int_{\Gamma_N} \! |{\bf a \cdot n}|v^2 \, ds+\int_{\Gamma_i} \! \rho \, |[v]|^2 \, ds \right. \\ & \left.  +\sum_{E \subset \Gamma} \frac{\gamma\varepsilon}{h_E}||[v]||^2_{L^2(E)} + \sum_{E \subset \Gamma} \frac{h_E}{\varepsilon}||[{\bf a}v]||^2_{L^2(E)}          \right )^{1/2}.
\end{aligned}
\end{equation}
We then have the following result.

\begin{theorem}
The error of the IMEX dG discretisation of problem \eqref{interfaceweak} satisfies
\begin{equation}
\begin{aligned}
\notag
||e(T)||_* &  \lesssim \sqrt{ G\phi },
\end{aligned}
\end{equation}
provided that the solution to \eqref{stfu6} exists.
\end{theorem}
\begin{proof}
From the triangle inequality, we have
\begin{equation}
\begin{aligned}
\notag
||e(T)||_* &  \leq ||e_c(T)||_* + ||u_{h,d}(T)||_* \lesssim \sqrt{ G\phi } + ||u_{h,d}(T)||_* .
\end{aligned}
\end{equation}
Thus, all that remains is to bound $||u_{h,d}(T)||_*$; the $L^{\infty}(L^2)$ part of this term was bounded in \eqref{stfu11} while the $L^2(H^1)$ part of the term was bounded in \eqref{stfu8}, \eqref{stfu9} and \eqref{stfu10}. Thus,
\begin{equation}
\begin{aligned}
\notag
||u_{h,d}(T)||_* \lesssim  \sqrt{ G\phi }.
\end{aligned}
\end{equation}
This completes the proof.
\end{proof}

\end{section}

\begin{section}{Numerical experiments}

We shall numerically investigate the presented a posteriori bound through an implementation based on the {\tt deal.II} finite element library \cite{BHK07}. In particular, we shall use Algorithm 3.1 from Chapter 3. Spatial refinement and coarsening are driven by the term $\eta_{S_1,k}$ subject to a spatial refinement threshold ${\tt stol}^+$ and a spatial coarsening threshold ${\tt stol}^-$. As in Chapter 3, we define
\begin{equation}
\hat{\eta}_{T,k}^2 := \int_{t^{k-1}}^{t^k} \! \big(\eta_{T_1,k}+\eta_{T_2,k}+\eta_{T_3,k} \big) ^2 \, dt+ T \int_{t^{k-1}}^{t^k} \! \eta^2_{T_4,k} \, dt,
\end{equation}
the sum of which bounds the full time estimator. Temporal refinement is then carried out using $\hat{\eta}_{T,k}$ subject to a temporal threshold ${\tt ttol}$ on each time interval.

In all our numerical experiments, we use polynomials of degree two and an initial  $4\times4$ uniform quadrilateral mesh. Finally, the spatial coarsening threshold is set to ${\tt stol}^- = 0.001*{\tt stol}^+$.

\begin{subsection}{Example 1}
Let $\Omega_1 = (-1,0)\times(-1,1)$, $\Omega_2=(0,1)\times(-1,1)$, ${\bf a} = (1,1)^T$, $f = -1$, $u_0 = 0$ and $T=1$. For the interface parameters, we set $\rho = 0.1$, $r =0.5$, $w_1 = 1$ and $w_2 = 0$. Under this choice of interface parameters, the solution to \eqref{interfaceweak} exhibits both boundary and interface layers of width $\mathcal{O}(\varepsilon)$. Solution profiles and meshes produced by the adaptive algorithm at the final time are given in Figures \ref{interfacemeshes} and \ref{interfacesolution}, respectively. 
The meshes generated by the adaptive algorithm clearly show that the error estimator is picking up both the interface layer and the boundary layer. 

\begin{figure}
\centering
\includegraphics[scale=0.34]{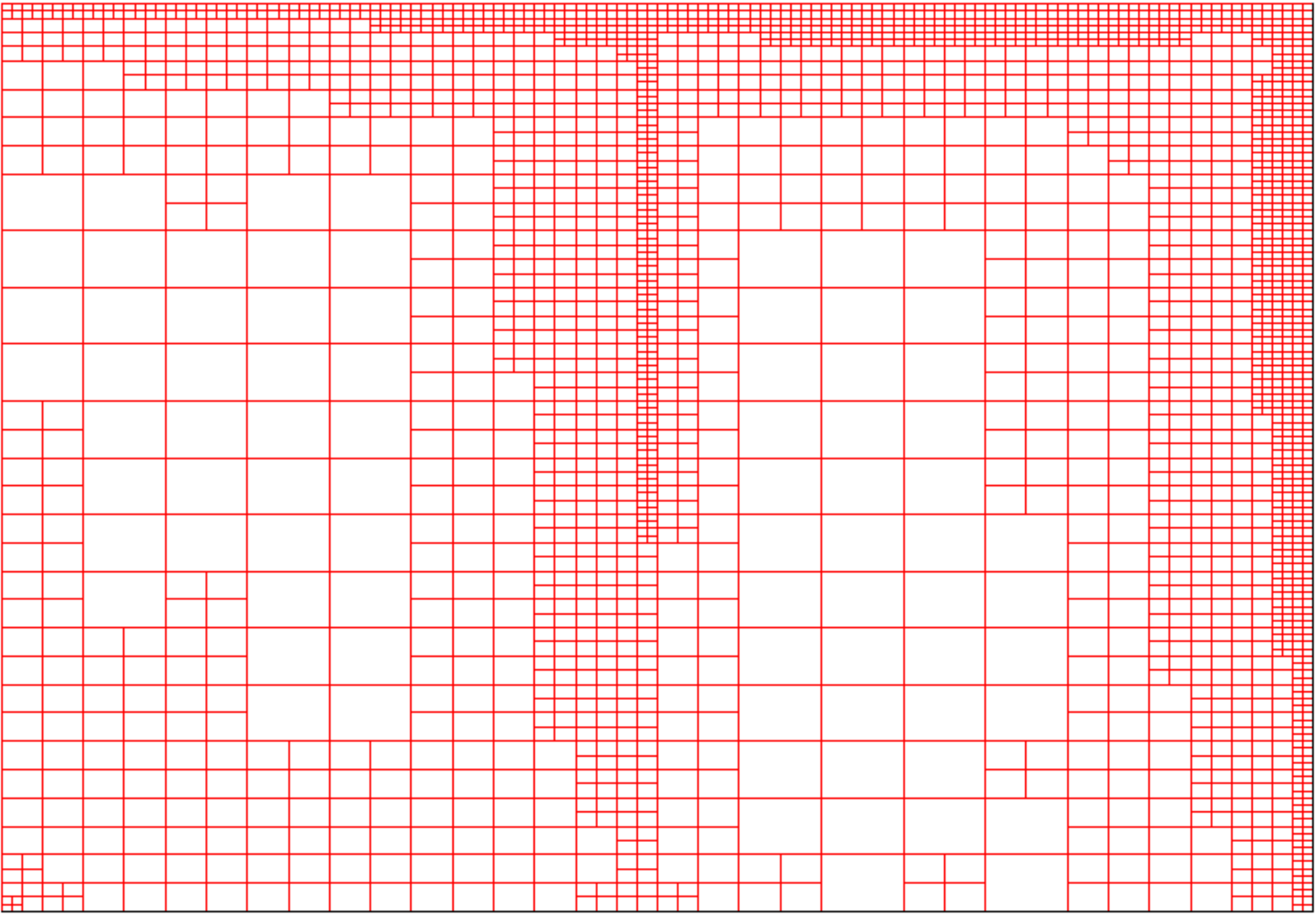} \, \includegraphics[scale=0.34]{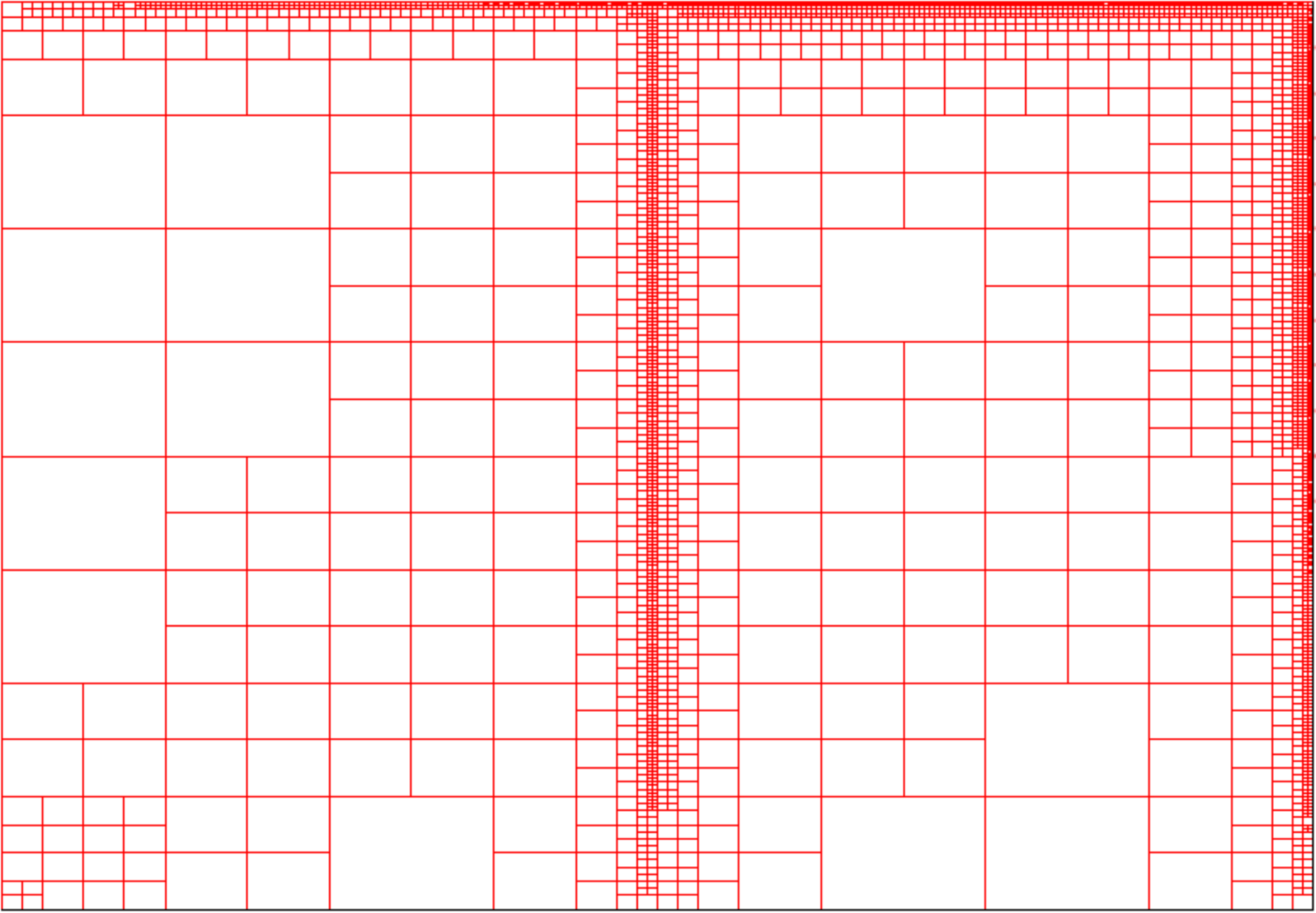}
\caption{Example 1: Meshes produced by the adaptive algorithm for $\varepsilon = 0.1$ (left) and $\varepsilon = 10^{-2}$ (right).}
\label{interfacemeshes}
\end{figure}
\begin{figure}
\centering
\includegraphics[scale=0.23]{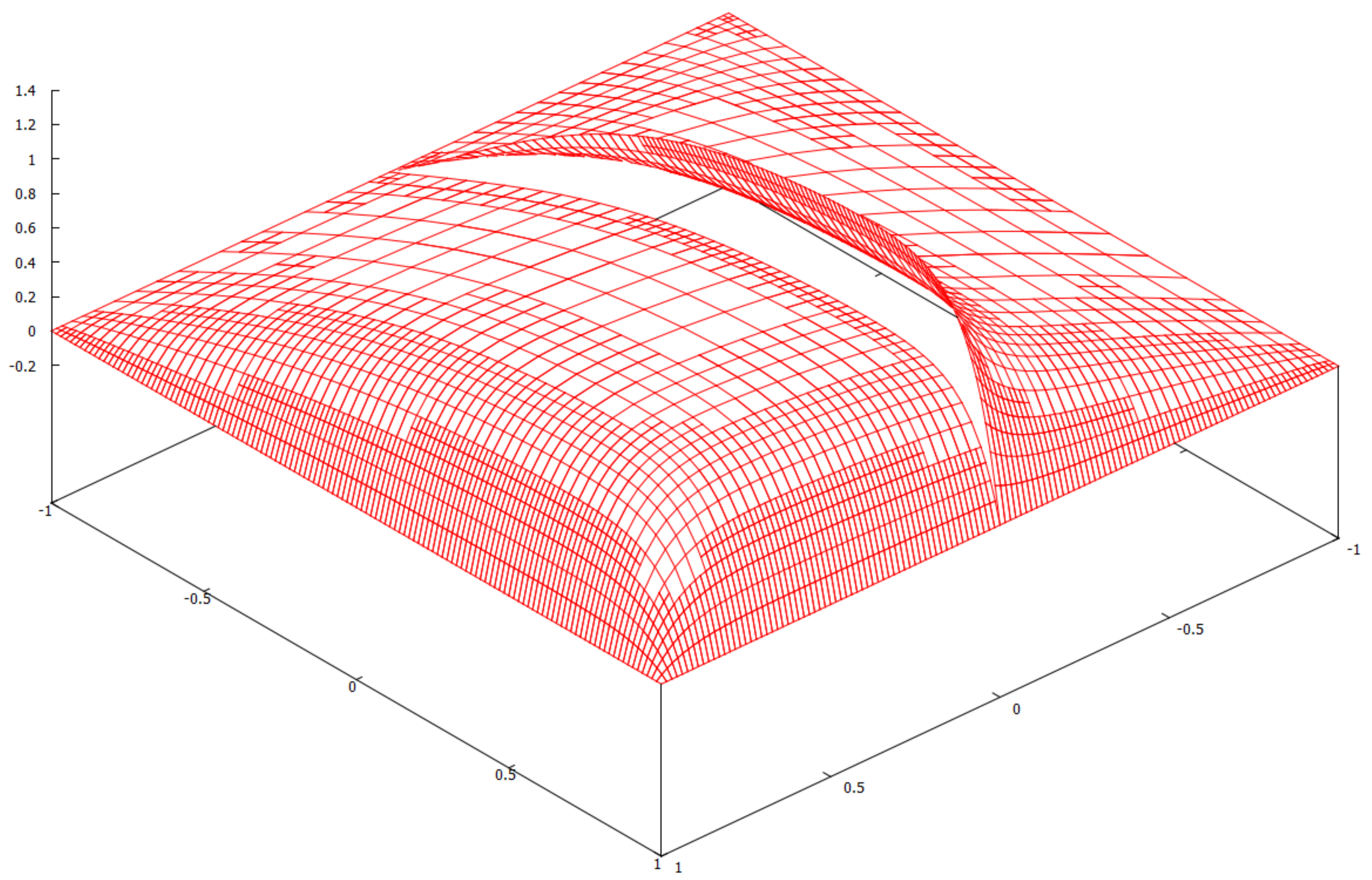} \includegraphics[scale=0.23]{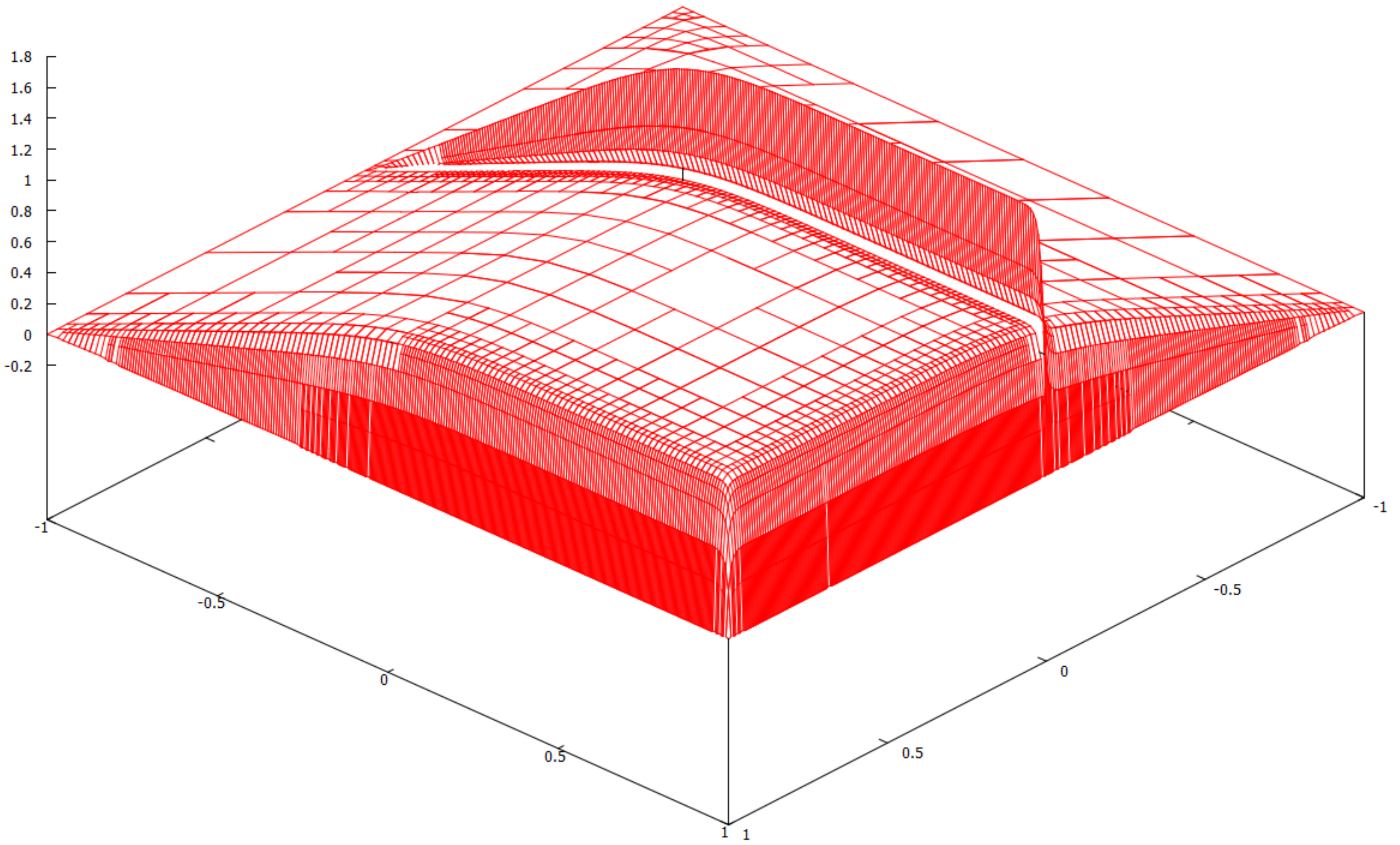}
\caption{Example 1: Solution profiles for $\varepsilon = 0.1$ (left) and $\varepsilon = 10^{-2}$ (right).}
\label{interfacesolution}
\end{figure}

To observe the rates of convergence of the error estimator $\phi$, we begin by fixing a small temporal threshold; the spatial threshold is then reduced to observe the spatial rates of the estimator. We then fix a small spatial threshold so that all layers are sufficiently resolved and reduce the temporal threshold to observe the temporal rates of the estimator. The results, given in Figure \ref{interfacerates}, show that the space and time estimators are of optimal order.

\begin{figure}
\centering
\includegraphics[scale=0.48]{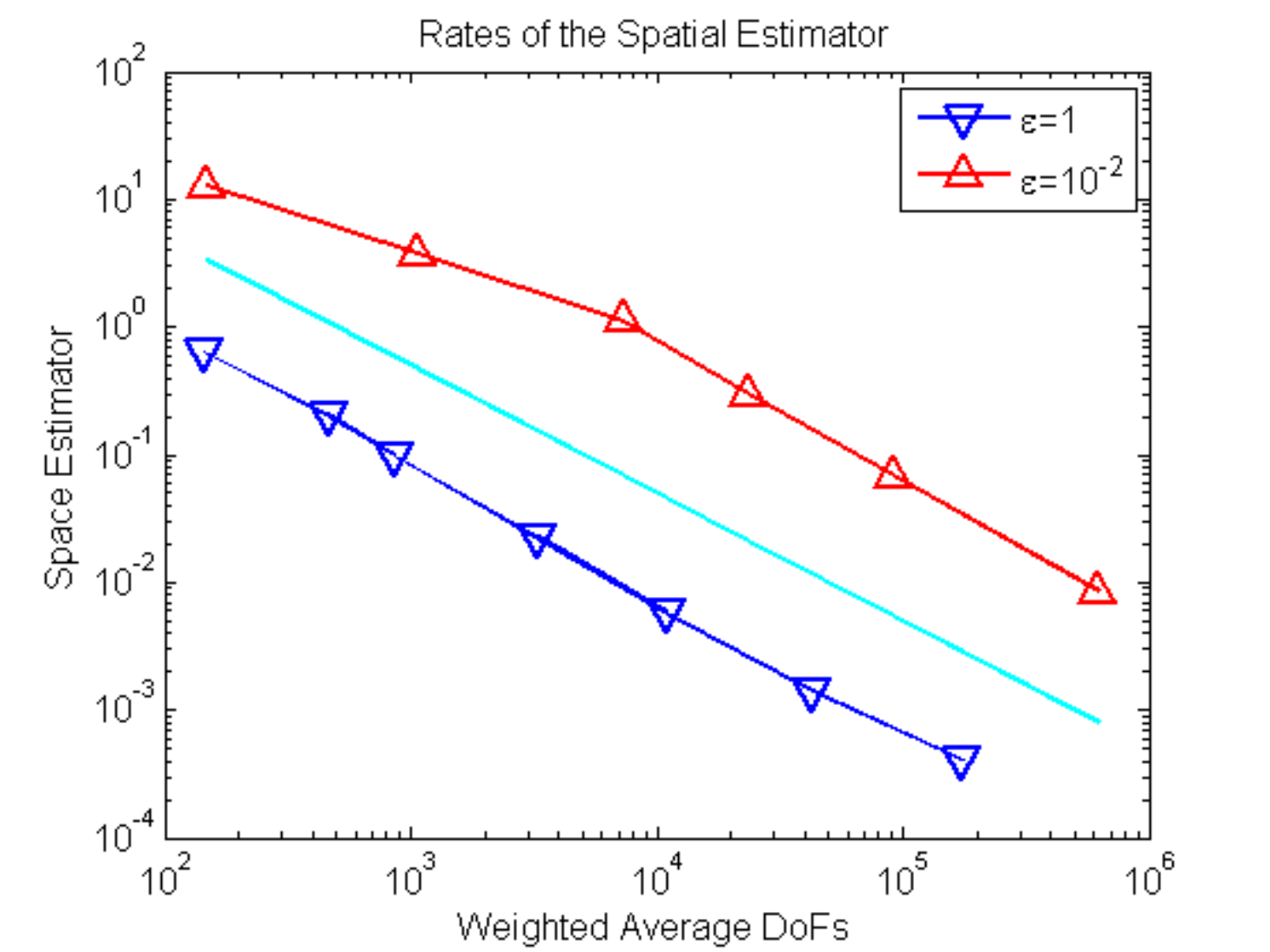} \includegraphics[scale=0.48]{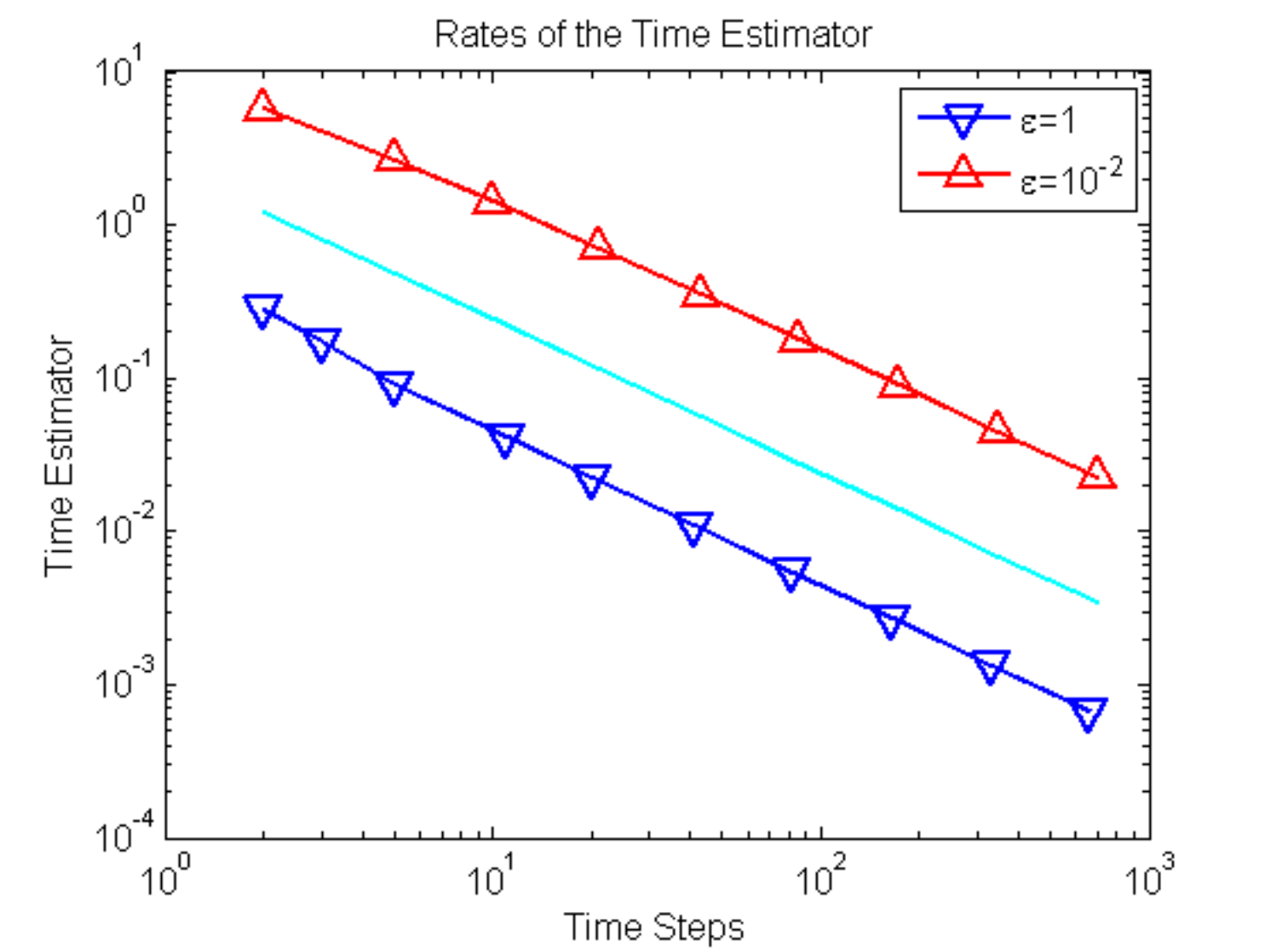}
\caption{Example 1: Spatial and temporal rates.}
\label{interfacerates}
\end{figure}

\end{subsection}

\end{section}

\begin{section}{Conclusions}

We derived an a posteriori error estimator for a nonlinear interface problem that is used to model the flow of solutes through semi-permiable membranes. The error estimator displayed optimal spatial and temporal rates under Algorithm 3.1. Furthermore, the error estimator was able to detect and refine the interface layer without wasting degrees of freedom on the opposite side of the interface. The constant from Gronwall's inequality is of the order $\exp (\varepsilon^{-1})$, which is impractical for the convection-dominated regime. A different treatment of the interface terms in Theorem \ref{interfacecoercive} may yield a tighter error bound, but it is not currently clear how to address this issue.

\end{section}

\chapter{Summary and outlook}

The aim of this work was to advance the understanding of adaptive algorithms for spatial finite element discretisations of parabolic problems {\bf --} this was achieved in two ways. Firstly, in Chapter 3 an adaptive algorithm was proposed that utilised an error estimator derived for a backward Euler dG discretisation of a  linear non-stationary convection-diffusion equation. This adaptive algorithm was applied to test problems in Chapter 3 as well as to a nonlinear interface problem in Chapter 6; in all test cases the error estimators were reduced at the theoretically expected rate with respect to the discretisation parameters. Secondly, adaptive algorithms designed to converge to the blow-up time of an ODE with polynomial nonlinearity were explored in Chapter 4.  This led to the development of an adaptive algorithm in Chapter 5 that was designed to approximate the blow-up time of a semilinear parabolic PDE. The adaptive algorithm was then applied in two numerical experiments and shown to approximate the blow-up time of both problems. We shall now discuss some ways in which the results of this work could be extended on a chapter by chapter basis.

In Chapter 3, we derived an a posteriori error estimator for a backward Euler dG discretisation of a linear non-stationary convection-diffusion equation and developed an adaptive algorithm to utilise the error estimator. There are several ways that the work in this chapter could be extended:
\begin{itemize}
\item The extension of the error estimator to include a variable diffusion coefficient.
\item The extension of the error estimator to higher order time stepping schemes.
\item The proof of lower bounds for the given a posteriori error estimator.
\item A rigorous proof that the adaptive algorithm minimises the spatial and temporal parts of the estimator.
\end{itemize}

In Chapter 4 and Chapter 5, we investigated the numerical approximation of blow-up through a posteriori error estimation and looked into the creation of adaptive algorithms designed to approximate the blow-up time. The work in these chapters could be furthered by:

\begin{itemize}
\item The extension of the error estimators to include more general nonlinearities.
\item The extension of the error estimators to higher order time stepping schemes. In particular, it would be of great interest to study $hp$ time stepping schemes for blow-up problems.
\item Conducting the error analysis for a different norm. In particular, conducting an error analysis for the $L^{\infty}(L^{\infty})$ norm may yield a faster approach to the blow-up time.
\end{itemize}

Finally, in Chapter 6 we derived an a posteriori error estimator for an IMEX dG discretisation of a nonlinear interface problem. The error estimator was then applied to a test problem using the adaptive algorithm from Chapter 3. The work in this chapter could be extended by:

\begin{itemize}
\item Removing or weakening the exponential dependence on $\varepsilon$ from the error estimator, possibly via a spectral estimate.
\item The extension of the error estimator to include variable diffusion, time dependent coefficients and data that is (possibly) discontinuous across the interface.
\item The extension of the error estimator to the full system of equations considered in \cite{CGJ13,CGJ14}.
\end{itemize}

\bibliographystyle{plain}
\bibliography{bibliography}

\end{document}